\documentclass[english]{article} 
\usepackage[absolute,overlay]{textpos} 
\usepackage{amsthm}
\usepackage{amsmath}
\usepackage{amssymb}
\usepackage[utf8]{inputenc}
\usepackage[english]{babel}
\usepackage{graphicx}
\usepackage{geometry}
\usepackage{caption}
\usepackage{subcaption}
\usepackage{pdfpages}
\usepackage{comment}
\usepackage{bm}
\usepackage{etoolbox}
\usepackage{mathtools}
\usepackage{tikz}

\usepackage{algpseudocode}
\usepackage{algorithm}
\usepackage{mathrsfs} 
\usepackage{multirow}

\usepackage{pgfplots}
\usepackage{pgfplotstable}
\usepackage{stmaryrd} 

\usepackage{hyperref}

\usepackage[style=numeric, giveninits=true, sorting=none, maxnames=99]{biblatex}
\usepackage[noabbrev,capitalise]{cleveref}
\crefformat{appendix}{#2#1#3}

\usepackage{booktabs}

\usetikzlibrary{fillbetween}
\usepgfplotslibrary{fillbetween}

\usetikzlibrary{shapes.misc}
\usetikzlibrary{math} 

\usetikzlibrary{arrows}

\definecolor{darkspringgreen}{rgb}{0., 0.55, 0.3}
\definecolor{dartmouthgreen}{rgb}{0.05, 0.5, 0.06}
\definecolor{etonblue}{rgb}{0.59, 0.78, 0.64}
\definecolor{airforceblue}{rgb}{0., 0.4, 0.66}
\definecolor{arylideyellow}{rgb}{0.91, 0.84, 0.42}
\definecolor{emerald}{rgb}{0.31, 0.78, 0.47}
\definecolor{uclagold}{rgb}{1.0, 0.7, 0.0}
\definecolor{cadmiumorange}{rgb}{0.93, 0.53, 0.18}

\theoremstyle{thmstyleone}
\newtheorem{theorem}{Theorem}

\theoremstyle{thmstyletwo}

\newtheorem{remark}{Remark}

\theoremstyle{thmstylethree}

\newcommand{\lopd}[0]{\mathcal{L}_\Delta}
\newcommand{\lopdt}[0]{\mathcal{L}_{\Delta}}

\newcommand{\usol}[0]{\underline{\uvec{u}}_\Delta}

\newcommand{\uvec}[2][3]{\boldsymbol{#2\mkern-#1mu}\mkern#1mu}

\newcommand\norm[1]{\left\lVert#1\right\rVert}
\newcommand\abs[1]{\left\lvert#1\right\rvert}

\newcommand{\R}{\mathbb{R}}

\newcommand{\bu}{\uvec{u}}
\newcommand{\bbu}{\underline{\uvec{u}}}

\newcommand{\undu}{\underline{\uvec{u}}}

\newcommand{\undr}{\underline{\uvec{r}}}

\newcommand{\uex}[0]{\underline{\uvec{u}}_{ex}}

\newcommand{\Temp}{T}

\newcommand{\E}{\boldsymbol{E}}

\newcommand{\lSSP}{\ell \text{SSP}}

\newcommand{\bF}{\uvec{f}}
\newcommand{\diff}[1]{{\mathrm{d}{#1}}}

\newcommand{\ubar}{\overline{\uvec{u}}}

\newcommand{\xip}{x_{i+\frac{1}{2}}}
\newcommand{\xin}{x_{i-\frac{1}{2}}}
\newcommand{\yjp}{y_{j+\frac{1}{2}}}
\newcommand{\yjn}{y_{j-\frac{1}{2}}}
\newcommand{\iip}{i+\frac{1}{2}}
\newcommand{\iin}{i-\frac{1}{2}}

\newcommand{\jjp}{j+\frac{1}{2}}
\newcommand{\jjn}{j-\frac{1}{2}}

\newcommand{\jph}{{j+\frac{1}{2}}}
\newcommand{\jmh}{{j-\frac{1}{2}}}

\newcommand{\iph}{{i+\frac{1}{2}}}
\newcommand{\imh}{{i-\frac{1}{2}}}

\newcommand{\qbar}{\overline{q}}

\begin{document}
\author{Lorenzo Micalizzi\footnote{Affiliation: Department of Mathematics, North Carolina State University, SAS Hall, 2108, 2311 Stinson Dr, Raleigh, NC 27607, United States. Email: lmicali@ncsu.edu}, and Eleuterio F. Toro\footnote{Affiliation: Laboratory of Applied Mathematics, DICAM, University of Trento, Via Mesiano 77, 38123 Trento, Italy. Email: eleuterio.toro@unitn.it}}
\title{Impact of Numerical Fluxes on High Order Semidiscrete WENO--DeC Finite Volume Schemes}

\maketitle

\abstract{
	
	The numerical flux determines the performance of numerical methods for solving hyperbolic partial differential equations (PDEs). 
	In this work, we compare a selection of 8 numerical fluxes in the framework of nonlinear semidiscrete finite volume (FV) schemes, based on Weighted Essentially Non--Oscillatory (WENO) spatial reconstruction and Deferred Correction (DeC) time discretization. 
	The methodology is implemented and systematically assessed for order of accuracy in space and time up to seven. 
	The numerical fluxes selected in the present study represent the two existing classes of fluxes, namely centred and upwind. Centred fluxes do not explicitly use wave propagation information, while, upwind fluxes do so from the solution of the Riemann problem via a wave model containing $A$ waves. Upwind fluxes include two subclasses: complete and incomplete fluxes. For complete upwind fluxes,  $A = E$, where $E$ is the number of characteristic fields in the exact problem.  For incomplete upwind ones,  $A < E$.
	Our study is conducted for the one-- and two--dimensional Euler equations, for which we consider the following numerical fluxes: Lax--Friedrichs (LxF), First--Order Centred (FORCE), Rusanov (Rus), Harten--Lax--van Leer (HLL), Central--Upwind (CU), Low--Dissipation Central--Upwind (LDCU), HLLC, and the flux computed through the exact Riemann solver (Ex.RS).
	
	We find that the numerical flux has an effect on the performance of the methods. The magnitude of the effect depends on the type of numerical flux and on the order of accuracy of the scheme. It also depends on the type of problem; that is, whether the solution is smooth or discontinuous, whether discontinuities are linear or nonlinear, whether linear discontinuities are fast-- or slowly--moving, and whether the solution is evolved for short or long time. 
	For the special case of smooth solutions, the expected convergence rates are attained for all fluxes and all orders. However, errors are still larger for the simpler fluxes, though differences diminish as the order of accuracy increases. 
	For all selected cases involving discontinuities, differences among fluxes arise for all orders of accuracy considered. Moreover, there are flow situations for which the differences are huge, independently of the order of accuracy of the scheme. 
	The best fluxes are the complete upwind ones. The difference between the best centred flux, FORCE, and incomplete upwind ones is not dramatic, which constitutes and advantage for good centred methods due to their simplicity and generality.
	
}

\section{Introduction}

The vast majority of numerical schemes for solving hyperbolic partial differential equations (PDEs) is based on a discrete representation of their underlying principle: the rate of change in time of some quantities inside a given spatial region is given by what crosses the surface of the same region, expressed by the (normal) flux function, plus what is generated/dissipated within the region, expressed by the source function (if present).
Therefore,  a numerical method will require corresponding expressions for the numerical flux and the numerical source.
Historically, Godunov~\cite{Godunov} is credited for having proposed a numerical flux as an integral average of the physical flux evaluated at the solution of the Riemann problem at the interface between the elements of a tessellation of the spatial domain. The resulting Godunov upwind method is a conservative generalization of  the CIR scheme, first presented in~\cite{courant1952solution} by Courant, Isaacson and Rees. 
Even after six decades, the design of numerical fluxes remains a fundamental task in the construction of finite volume  (FV)~\cite{hirsch2007numerical,ToroBook,leveque2002finite,godlewski2021numerical,toro2024computational}, finite difference~\cite{leveque2007finite} and Discontinuous Galerkin (DG) finite element methods~\cite{reed1973triangular,cockburn2000development,cockburn2001runge}.
This is so, both in the frameworks of semidiscrete and fully--discrete schemes.
Apart from the basic properties of consistency and Lipschitz--continuity, the scheme designer aims for generous stability properties, monotonicity (for the scalar case),  minimal numerical diffusion and efficiency~\cite{ToroBook}.  It is known that for the scalar case the Godunov upwind method is the scheme with the smallest local truncation error, within the class of monotone schemes.  

The largely pending challenge is to design numerical fluxes with desirable properties for solving nonlinear systems in multiple space dimensions. 
Over the last few decades, many numerical fluxes with different properties have been put forward,  which then prompts a relevant and natural question: 
among the available fluxes in the current literature,  which ones achieve the optimal balance between accuracy and computational cost?  
At the first order level the current literature provides useful, even if not exhaustive, information on the performance of various numerical fluxes~\cite{hirsch2007numerical,ToroBook,leveque2002finite,godlewski2021numerical,toro2024computational}. 
Much less is known in the setting of higher order (in space and time) numerical methods. 
A reasonable expectation is that the adoption of a high order discretization could compensate for the deficiency of less accurate numerical fluxes.

Despite being the numerical flux a crucial element, the design of numerical schemes for hyperbolic PDEs requires the definition of other important components which depend on the specific discretization framework adopted.
Broadly speaking, there are two major discretization frameworks, namely the 
semidiscrete (or method of lines) and the fully--discrete approaches. 
In the former setting, the discretization in space is separated from the one in time.  Given a discretization in space, the problem remains continuous in time as a system of ordinary differential equations (ODEs). In principle, any ODEs solver can then be deployed to complete the scheme. 
Instead, in the fully--discrete setting, the discretizations in space and time are intertwined, fully coupled and simultaneously designed in a single step.
In the high order semidiscrete FV framework, one needs a suitable spatial reconstruction and a suitable time--stepping strategy.
The spatial reconstruction must be nonlinear, so as to prevent or reduce spurious oscillations in the vicinity of  discontinuities or large gradients (even in smooth problems). 
The need for the nonlinear character of the spatial reconstruction emerges from Godunov's theorem~\cite{Godunov} as a necessary condition for a monotone scheme; see~\cite{ToroBook} for statement and proof.

In this paper, we study the performance of 8 existing numerical fluxes in the setting of an arbitrary high order semidiscrete FV framework, comprising Weighted Essentially Non--Oscillatory (WENO)~\cite{liu1994weighted,shu1998essentially,shu1989efficient} spatial reconstruction and Deferred Correction (DeC)~\cite{micalizzi2023new,ciallella2022arbitrary,Decremi,minion2003semi,Decoriginal} time discretization.
%
More in detail, we carry out a systematic comparison of the following numerical fluxes: Lax--Friedrichs (LxF)~\cite{lax1954weak}, First--Order Centred (FORCE)~\cite{Toro1996,toro2000centred,chen2003centred}, Rusanov~\cite{Rusanov1961}~(Rus), Harten--Lax--van Leer (HLL)~\cite{harten1983upstream}, Central--Upwind (CU)~\cite{kurganov2001semidiscrete,kurganov2000new}, Low--Dissipation Central--Upwind (LDCU)~\cite{kurganov2023new}, HLLC~\cite{toro1992restoration,toro1994restoration},  the Godunov flux from the exact Riemann solver~\cite{Godunov}~(Ex.RS). Most of these are thoroughly described in~\cite{ToroBook}.

Both the spatial and time discretizations adopted here are well established.  As a matter of fact, the WENO--DeC approach has already been investigated up to order 5 in space and time; the performance of such a framework has been shown to be very satisfactory in tackling challenging and realistic problems~\cite{ciallella2022arbitrary,ciallella2023arbitrary,ciallella2024high}.
In the present work, we investigate the performance  of WENO--DeC schemes up to order of accuracy 7 in space and time.  Moreover, in addition to the very high order extension, we compare the performance of many numerical fluxes available in the literature within such a framework. 
The assessment of the resulting methods is through the time--dependent, one-- and two--dimensional compressible Euler equations.  A judicious choice of suitable problems is performed with the aim of understanding the strengths and limitations of the numerical fluxes under investigation in combination with the aforementioned semidiscrete approach. 
Key features of the chosen problems include: robustness in the presence of very strong shocks; accuracy in resolving waves associated with intermediate, linear, characteristic fields; and long--time evolution.

From the systematic assessment of the methods following the above criteria, we anticipate the following conclusions:
\begin{itemize}
	\item There are several physical situations in which significant differences are seen in the performance of the numerical fluxes under investigation, with HLLC and Ex.RS outperforming by far the other competitors.
	\item There are  tests for which the choice of the numerical flux has less impact on the methods performance.  But even in such cases, it can consistently be observed that  the performance of LxF,  FORCE and Rus is inferior to that of the remaining  numerical fluxes,  which give similar results amongst themselves.  In particular,  the performance of LxF is always the worst, while the relative performance of FORCE and Rus depends on the specific problem.  
	\item For all investigated numerical fluxes, there is an advantage in increasing the space--time order of accuracy of the discretization. The advantage is much more evident in the more diffusive numerical fluxes,  that is the centred fluxes LxF and FORCE,  and the incomplete upwind numerical flux Rus.  The benefits of the enhanced higher order accuracy are less evident as the sophistication of the numerical flux increases.  As a matter of fact, to a certain extent and depending on the test problem, increasing the space--time order of accuracy compensates for the deficiencies of a more diffusive numerical flux.  However,  enhanced higher order accuracy per se, within the range of considered orders, is not sufficient to attain the accuracy delivered by sophisticated upwind fluxes derived from  complete Riemann solvers, namely HLLC and Ex.RS.  Conversely,  depending on the physical situation, a suitable numerical flux, even for a first order method, may be equivalent to implementing higher order space--time discretizations.  This is typically the case for slowly--moving linear waves associated with intermediate characteristic fields and for very long--time evolutions of traveling waves.
	
	\item A surprising outcome has emerged from the implementation of the centred  (essentially one-dimensional) numerical fluxes LxF and FORCE in a two--dimensional setting via a simultaneous updating formula.  A von Neumann stability analysis of these numerical fluxes in a first order setting shows them to be linearly unstable in two and three space dimensions~\cite{toro2000centred,ToroBook}.  Curiously,  we found that increasing the order of accuracy has a stabilizing effect,  as shown in some of  our numerical experiments, though an explanation remains illusive.
	
\end{itemize}

There are other works in the literature concerned with the influence of numerical fluxes on high order methods.  Investigations in the semidiscrete framework include~\cite{leidi2024performance,qiu2007numerical,qiu2008development,qiu2006numerical,
	hongxia2020numerical,san2015evaluation}.  
In~\cite{leidi2024performance}, the performances of Rus, HLL and a low--dissipation version of HLLC (referred to as ``LHCLL'' in the reference) are compared for  low Mach number flows, in a FV setting with various spatial reconstructions up to order 7.
DG schemes with 8 numerical fluxes and  spatial accuracy up to order 3 are studied in~\cite{qiu2007numerical,qiu2008development}, and with 9 numerical fluxes and spatial accuracy up to order 4
in~\cite{qiu2006numerical}. 
In~\cite{hongxia2020numerical},  in a FV setting with fifth order HWENO space reconstruction the performance of 8 numerical fluxes is assessed.
In~\cite{san2015evaluation},  results  from 6  numerical fluxes in a FV framework with WENO space reconstruction up to order 7 are compared on the Kelvin--Helmholtz instability problem.   In all previously alluded comparative analyses, a semidiscrete approach with third order time integration was adopted.
Two further related works~\cite{titarev2005weno,toro2005tvd} are worth mentioning,  in which the basic monotone flux utilized in the high order methods is replaced by a total variation diminishing (TVD) flux: in \cite{titarev2005weno} the approach is implemented in a semidiscrete framework, while in  \cite{toro2005tvd} this is done in a fully--discrete framework.

In the context of the present work concerned with semidiscrete methods, it is important to remark that a broadly adopted practice consists in employing high order space discretizations along with lower order time discretizations. See for example~\cite{evstigneev2016construction,Evstigneev2016OnTC,gerolymos2009very,
	balsara2000monotonicity,shi2003resolution,hermes2012linear,gao2020seventh},  in which very high order spatial reconstructions are considered but the order of accuracy of the selected time discretizations never exceeds 4.
Many published works consider very high order space discretizations in combination with strong stability preserving (SSP)~\cite{shu1988total,shu1988efficient} or linearly strong stability preserving ($\lSSP$)~\cite{gottlieb2001strong} Runge--Kutta (RK) in time. The accuracy barrier for such ODEs solver, if non-negative RK coefficients are to be preserved~\cite{shu1988efficient},  is order 4~\cite{ruuth2002two} on nonlinear problems. 
Actually, $\lSSP$ RK methods can be arbitrarily high order accurate but only on linear problems.
With the main goal of preventing loss in accuracy due to the mismatch between temporal and spatial order, in some works, a well--tuned reduction of the time step is performed. The main benefit of this strategy is to formally make the accuracy of the scheme equal to the one of the space discretization, however, the severely reduced time step causes excessive numerical diffusion and a huge increase in computational cost, making the scheme unsuitable for practical applications.
In many other works, no adaptation of the time step is considered and the formal order of the scheme is limited by the one of the lower order time discretization.
Such a practice is based on the questionable assumption that the spatial error always dominates the time error. Preliminary investigations performed by the authors seem to contradict this expectation, though a thorough study of this issue is left for future works.
In fact, as already stated,  the WENO--DeC framework adopted in this paper allows for the construction of arbitrarily high order schemes, which are distinguished by the fact that the temporal order of accuracy matches the one of the WENO spatial reconstruction, that is to say $2r-1$, where $r-1$ is the degree of component ENO polynomials making up the WENO polynomial.
Related works on arbitrarily high order frameworks  are available in the literature; see for example~\cite{veiga2024improving,velasco2023spectral,abgrall2023extensions,Decremi,micalizzi2024novel,abgrall2019high,bacigaluppi2023posteriori,abgrall2020high}. Some of them~\cite{veiga2024improving,velasco2023spectral,abgrall2023extensions} are obtained through a simple method of lines approach, adopting a spatial discretization of the PDE and solving in time the resulting ODEs system with a sufficiently accurate time integration method.
In the other mentioned references, more involved space--time discretizations are considered, such as the continuous Galerkin--DeC framework described in~\cite{Decremi,micalizzi2024novel,abgrall2019high,bacigaluppi2023posteriori,abgrall2020high}.
An alternative approach to construct schemes of arbitrary space--time accuracy is the fully--discrete ADER methodology, first communicated in the early works \cite{toro2001towards,grptoro,titarev2002ader,schwartzkopff2002ader}.  
In the fully--discrete ADER approach, the discretizations in space and time are inextricably coupled via the solution of the  Generalized Riemann problem~\cite{grptoro}, $GRP_{m}$, at cell interfaces, in which the initial conditions consist of nonlinear reconstructed polynomials of arbitrary degree $m$, leading to a method of order of accuracy equal to $m+1$ in space and time. 
The ADER method is a one--step scheme in which the nonlinear reconstruction is performed only once per time step.  
The ADER methodology has been developed in both the FV and DG frameworks.  
Further developments of the ADER methodology can be found for example in~\cite{dumbser2005ader,dumbser2006arbitrary,dumbser2006building,dumbser2008unified,
	ADERNSE,dumbser2009very,boscheri2019high,popov2024space,toro2024ader,micalizzi2023efficient}. 
Elementary introductions to ADER can be found in~\cite[Chapters 19 and 20]{ToroBook} and in~\cite[Chapter 14]{toro2024computational}. 

The rest of the paper is structured as follows.
In Section~\ref{sec:basic_notions} we recall the governing equations and the FV method
in the semidiscrete setting.
In Section~\ref{sec:space} we describe the space discretization, namely, the WENO reconstruction in Section~\ref{sec:WENO} and the numerical fluxes under investigation in Section~\ref{sec:fluxes}, 
while, an outline of the DeC time discretization is given in Section~\ref{sec:DeC}.
Numerical results are reported in Section~\ref{sec:numerical_results}.
Conclusions and future perspectives are found in Section~\ref{sec:conclusions}.

\section{Governing equations and Finite Volume method}\label{sec:basic_notions}
In this section, we briefly recall the governing equations and the FV method.
The former topic is described in Section~\ref{sec:governing_equations}, while, the latter one is described in Section~\ref{sec:FV}.
In particular, since in the numerical results we investigate one- and two--dimensional examples, we stick to a description considering two space dimensions.
Generalization to a higher number of space dimensions and restriction to one space dimension are relatively straightforward.

\subsection{Hyperbolic systems of conservation laws}\label{sec:governing_equations}
The governing equations considered herein are hyperbolic systems of conservation laws in the form
\begin{equation}\label{eq:sys}
	\frac{\partial}{\partial t} \uvec{u}(x,y,t) + \frac{\partial}{\partial x}\uvec{f}(\uvec{u}(x,y,t))+\frac{\partial}{\partial y}\uvec{g}(\uvec{u}(x,y,t)) = \uvec{0}, \quad (x,y,t)\in\Omega\times\mathbb{R}^+_0,
\end{equation}
where $\uvec{u}:\Omega\times\mathbb{R}^+_0\longrightarrow \mathbb{R}^{N_c}$ is the unknown vector of conserved variables, with $N_c\in \mathbb{N}^+$ number of scalar components in the system, $\uvec{f},\uvec{g}:\mathbb{R}^{N_c}\longrightarrow\mathbb{R}^{N_c}$ are the fluxes in the $x$- and $y$-direction respectively, and $\Omega\subseteq \mathbb{R}^2$ is the two--dimensional space domain, here assumed to be a rectangle $[x_L,x_R]\times[y_D,y_U]$.
Other shapes of $\Omega$ can be considered, however, in the numerical experiments, we adopt a Cartesian setting, for which a rectangular shape of the space domain is required.
In order for system~\eqref{eq:sys} to be hyperbolic, we must have that any linear combination through real coefficients $\omega_1,\omega_2\in \mathbb{R}$ of the Jacobians of the fluxes is real--diagonalizable, i.e., we require the matrix $\omega_1 \frac{\partial\uvec{f}}{\partial \uvec{u}}(\uvec{u})+\omega_2 \frac{\partial\uvec{g}}{\partial \uvec{u}}(\uvec{u})$ to have $N_c$ real eigenvalues and a corresponding set of linearly independent eigenvectors.
Given a unit vector $\uvec{n}:=(n_1,n_2)^T\in \mathbb{R}^2$, the eigenvalues of the normal Jacobian matrix in direction $\uvec{n}$, 
\begin{equation}
J_{\uvec{n}}(\uvec{u}):=n_1 \frac{\partial\uvec{f}}{\partial \uvec{u}}(\uvec{u})+n_2 \frac{\partial\uvec{g}}{\partial \uvec{u}}(\uvec{u}),
\end{equation}
represent the wave (or characteristic) speeds of the system along the $\uvec{n}$ direction.

More in detail, we focus on the Euler equations of compressible fluid dynamics obtained by
\begin{equation}
	\uvec{u}:=\begin{pmatrix}\rho\\ \rho u\\ \rho v\\E\end{pmatrix},\quad\uvec{f}(\uvec{u}):=\begin{pmatrix}\rho u\\\rho u^2+p\\\rho uv \\(E+p)u\end{pmatrix},\quad\uvec{g}(\uvec{u}):=\begin{pmatrix}\rho v\\\rho uv\\\rho v^2+p \\(E+p)v\end{pmatrix},
	\label{eq:conservative_variables_2D}
\end{equation}
where $\rho$ is the density of the fluid, $u$ and $v$ are the components of the velocity in the $x$- and $y$-direction respectively, $p$ is the pressure, and $E$ the total energy.
The system is closed by an equation of state, given in terms of the specific internal energy $e$
\begin{align}
	e&:=e(\rho,p),
\end{align}
and the total energy is given by
\begin{align}
	E&=\rho \left[e + \frac{1}{2}(u^2+v^2) \right].
\end{align}
Here, we consider the case of ideal fluid for which
\begin{align}
	e(\rho,p):=\frac{p}{\rho(\gamma-1)}
\end{align}
with $\gamma$ adiabatic coefficient given by the ratio between specific heats at constant pressure and volume. Here, we assume $\gamma:=1.4$.
For this system, the wave speeds along the $\uvec{n}$ direction are $v_{\uvec{n}}+c$, $v_{\uvec{n}}-c$ and $v_{\uvec{n}}$, with algebraic multiplicity 1, 1 and 2 respectively, where $v_{\uvec{n}}:=u n_1+v n_2$ is the velocity projection along the $\uvec{n}$ direction, and the sound speed $c$ is given by
\begin{align}
	c:=\sqrt{\frac{\frac{p}{\rho^2}-\frac{\partial}{\partial \rho}e}{\frac{\partial}{\partial p}e}},
\end{align}
which in the considered ideal case reduces to $c:=\sqrt{\gamma \frac{p}{\rho}}$.

\subsection{Finite Volume method}\label{sec:FV}
The FV method was introduced in 1959 through the famous work of Godunov~\cite{Godunov}, generalizing the CIR scheme~\cite{courant1952solution} introduced from Courant, Isaacson and Rees, and it quickly became one of the most popular techniques for numerically solving hyperbolic PDEs. Key contributions on this topic include~\cite{hirsch2007numerical,godlewski2021numerical,toro2024computational,shu1988efficient,shu1989efficient,shu1998essentially,leveque2002finite,ToroBook}. The interested reader can find more information in~\cite{ciallella2022arbitrary,ciallella2024high,ciallella2023arbitrary,abgrall1994essentially,titarev2004finite,toro2006musta,dumbser2007quadrature}.
The method relies on an integral formulation of the governing equations over some spatial control volumes.
Let us therefore consider a Cartesian tessellation of $\Omega$, i.e., a collection of non--overlapping rectangles $C_{i,j}:=[x_{\imh},x_{\iph}]\times [y_{\jmh},y_{\jph}]$ covering $\Omega$ exactly. We assume the tessellation to be uniform, i.e., we consider $\xip-\xin=\Delta x$ and $\yjp-\yjn=\Delta y$ for all $i$ and $j$.
Thus, we proceed by integrating system~\eqref{eq:sys} over each control volume $C_{i,j}$, obtaining the semidiscrete evolution formula
\begin{equation}\label{eq:FV_semidiscretization}
	\frac{d}{dt}\ubar_{i,j}(t) + \frac{1}{\Delta x}(\uvec{f}_{\iip,j}-\uvec{f}_{\iin,j}) + \frac{1}{\Delta y}(\uvec{g}_{i,\jjp}-\uvec{g}_{i,\jjn}) = \uvec{0},
\end{equation}
for the cell average $\ubar_{i,j}$ of the solution over $C_{i,j}$.
The quantities $\uvec{f}_{\iip,j}$ and $\uvec{g}_{i,\jjp}$ are the averages of the fluxes over cell boundaries at time $t$
\begin{align}
	\uvec{f}_{\iip,j} &:= \frac{1}{\Delta y}\int_{\yjn}^{\yjp}\bF(\bu(\xip,y,t))\;dy,   \label{eq:flux F} \\
	\uvec{g}_{i,\jjp} &:= \frac{1}{\Delta x}\int_{\xin}^{\xip}\uvec{g}(\bu(x,\yjp,t))\;dx  .  \label{eq:flux G}
\end{align}
In order to obtain a numerical scheme usable in practice, we need to approximate $\uvec{f}_{\iip,j}$ and $\uvec{g}_{i,\jjp}$ through adoption of suitable quadrature formulas for the surface integrals and suitable reconstruction, in the quadrature points, of the functions under the integrals.

Typically, for this purpose, a space reconstruction of the solution is considered locally in each control volume using the cell averages, and this is used for computing the needed flux approximations in the surface quadrature points via numerical fluxes.
Let us focus on $\uvec{f}_{\iip,j}$, as $\uvec{g}_{i,\jjp}$ is obtained similarly. For the sake of clarity and in order to lighten the notation, we omit the time dependency for the rest of the section, being however clear that all quantities are associated to a generic time $t$.
The reconstruction of the solution is local to each control volume, therefore, in each quadrature point $(\xip,y_q)$ of the face $\left\lbrace \xip \right\rbrace\times [y_{\jmh},y_{\jph}]$, we have two approximations for $\uvec{u}$ from the two neighboring cells $C_{i,j}$ and $C_{i+1,j}$, respectively given by
\begin{align}
	\bu^L_{\iip}(y_q) \approx \bu(\xip^-,y_q)\;,\;\;\;\bu^R_{\iip}(y_q) \approx \bu(\xip^+,y_q).
\end{align}
The flux in the quadrature point is thus computed using a numerical flux $\widehat{\bF}(\bu^L_{\iip}(y_q),\bu^R_{\iip}(y_q)) \approx \bF(\bu(\xip,y_q))$ taking in input the reconstructed values in the quadrature point.
Thus, the average flux in the $x$-direction is computed as 
\begin{equation}\label{eq:flux with GP}
	\uvec{f}_{\iip,j} = \frac{1}{\Delta y} \sum_{q=1}^{N_w} w_q \widehat\bF(\bu_{\iip}^L(y_q),\bu_{\iip}^R(y_q)),
\end{equation}
where $y_q \in [y_{j-1/2},y_{j+1/2}]$ and $w_q$, for $q=1,\dots,N_w$, constitute quadrature points and weights of a quadrature formula.

The resulting space discretization is $P$-th order accurate provided that the employed reconstruction and the adopted quadrature formulas are $P$-th order accurate.
In principle, the numerical flux has no direct impact on the order of accuracy, however, it plays a crucial role in terms of stability. Also the diffusion of the method is strictly related to the adopted numerical flux.
In the numerical simulations, we will consider 8 numerical fluxes, and they are detailed in Section~\ref{sec:fluxes}.

It is essential, when dealing with simulations involving discontinuities or steep gradients in the context of high order space reconstructions, to have a proper non--oscillatory reconstruction in order to prevent spurious oscillations, which could lead to simulation blow--ups.
This is why we assume here the ``nonlinear'' WENO space reconstruction detailed in Section~\ref{sec:WENO}.

\begin{remark}[On the importance of nonlinearity]
	When aiming to solve hyperbolic equations with methods of accuracy order greater than one, it is mandatory to recall Godunov's theorem~\cite{Godunov}.  
	This result states that there are no monotone (for the scalar case) linear schemes of order of accuracy greater than one. 
	For details on the statement and proof, see \cite{ToroBook,toro2024computational}. 
	Therefore, a necessary (but not sufficient) property to incorporate in the design of high order numerical methods is nonlinearity.
	TVD methods~\cite{kolgan1972application,sweby1984high,harten1987uniformly1,van1973towards} were the first attempt to construct nonlinear schemes of accuracy order up to two. 
	Essentially Non-Oscillatory (ENO)~\cite{harten1987uniformly1,harten1987uniformly2} methods represented the first attempt to go beyond second order.  
	WENO~\cite{liu1994weighted,jiang1996efficient} schemes represent an improvement of ENO methods. 
\end{remark}

Clearly, once the space discretization has been fixed, the semidiscrete evolution formulas of all cell averages~\eqref{eq:FV_semidiscretization} constitute an ODEs system to be solved in time through numerical integration. The overall scheme has order at least $P$ if and only if both the space and the time discretizations have order of accuracy at least $P$.
Most of the works dealing with very high order space reconstructions consider lower order time integration schemes~\cite{evstigneev2016construction,Evstigneev2016OnTC,gerolymos2009very,balsara2000monotonicity,shi2003resolution,hermes2012linear,gao2020seventh}, namely, SSP or $\lSSP$ RK schemes. Even though $\lSSP$ RK schemes can be arbitrary high order accurate on linear problems, the accuracy of both schemes is limited to order 4~\cite{ruuth2002two} on general nonlinear problems. Rigorously speaking, this is true up to possible modifications involving downwind computation of numerical fluxes in some stages when some RK coefficients are negative~\cite{shu1988efficient}, which are however never considered in the mentioned references.
Loss in accuracy is, sometimes, prevented by suitably reducing the time step, leading to schemes which are not usable in concrete applications.
Here, aiming at a truly arbitrary high order framework both in space and time, we consider a DeC time discretization~\cite{micalizzi2023new,ciallella2022arbitrary,Decremi}, described in Section~\ref{sec:DeC}.

\begin{remark}[On source terms]
	In this work, we only focus on conservation laws, i.e., on hyperbolic PDEs with no source terms.
	In presence of a source term, one must discretize its cell averages in a suitable way, with order of accuracy matching the one of the spatial reconstruction used to compute the numerical fluxes to prevent order degradation and associated loss in efficiency.
	To this end, if it depends on the solution, one could perform the WENO reconstruction of the solution also in quadrature points internal to the cells to guarantee a sufficiently accurate computation of its cell average, as done in~\cite{ciallella2022arbitrary,ciallella2024high}.
	Let us also remark that the presence of the source induces some other challenges. Source terms can be stiff and therefore induce limitations on the time step, see for example~\cite{dumbser2008finite,boscarino2018implicit,huang2018bound,abgrall2020high,boscheri2023all,caballero2024semi}. Moreover, one may be interested in the preservation of steady equilibria dependent on the source term, see for example~\cite{xing2006high,CaPa,mantri2021well,berberich2021high,ciallella2023arbitrary,barsukow2023well,micalizzi2024novel,ciallella2024high,Maria}.
	These aspects will be studied in future projects, but they are out of the scope of this investigation.
\end{remark}

\section{Space discretization}\label{sec:space}
In this section, we describe the space discretization.
In particular, in Section~\ref{sec:WENO}, we detail the WENO space reconstruction, and in Section~\ref{sec:fluxes} we present the numerical fluxes investigated in this work.

\subsection{Weighted Essentially Non--Oscillatory reconstruction}\label{sec:WENO}

The WENO space reconstruction has been introduced in 1994~\cite{liu1994weighted}. Due to its desirable features, such has its (arbitrary) high order accuracy, robustness and flexibility, it has been employed and developed in many subsequent works. Giving an exhaustive review of all the applications of WENO in a few lines is rather difficult, therefore, we refer the reader to the seminal works~\cite{shu1998essentially,jiang1996efficient,shu1988efficient,shu1989efficient,abgrall1994essentially} and references therein.

In this section, we will describe the WENO reconstruction for a scalar component $q$ in a Cartesian framework. Clearly, the same reconstruction must be performed in the same way for all the components.
We will also start by considering a one--dimensional context, in Section~\ref{sec:WENO1D}. 
The extension to higher dimensions, detailed in Section~\ref{sec:WENO2D}, consists in the application of the one--dimensional reconstruction along different directions in subsequent ``sweeps''.

As numerically demonstrated in~\cite{qiu2002construction,miyoshi2020short,peng2019adaptive,ghosh2012compact}, in order to prevent spurious oscillations in high order spatial discretizations, it is essential to apply the WENO reconstruction to characteristic variables rather than to conserved ones. We explain how to do this in Section~\ref{sec:characteristic_WENO}.

\subsubsection{One--dimensional case}\label{sec:WENO1D}
Let us consider a quantity $q(\xi)$ for which cell averages $\qbar_i$ are available over a collection of uniform cells $C_i:=[\xi_{\imh},\xi_{\iph}]$ of length $\Delta \xi$, and let us focus on the local reconstruction in the generic cell $C_i$.
The WENO reconstruction of order $2r-1$ is obtained by considering a ``big'' high order stencil of $2r-1$ cells
\begin{equation}
	\label{eq:HO_stencil}
	\mathcal{S}^{HO}:=\left\lbrace C_{i-(r-1)},\dots,C_{i+(r-1)} \right\rbrace,
\end{equation}
allowing for $(2r-1)$-th order of accuracy.
Along with such a stencil, we identify $r$ ``small'' low order stencils of $r$ cells
\begin{equation}
	\label{eq:LO_sentcil}
	\mathcal{S}^{LO}_\ell:=\left\lbrace C_{i-(r-1)+\ell},\dots,C_{i+\ell} \right\rbrace, \quad \ell=0,\dots,r-1,
\end{equation}
allowing for $r$-th order of accuracy.

Since we work with cell averages, given a generic stencil $\mathcal{S}:=\left\lbrace C_{s},\dots,C_{f} \right\rbrace$ with $N_{\mathcal{S}}:=f-s+1$ cells, either low or high order, we cannot trivially construct the Lagrange interpolation polynomial of $q$, as its point values are not readily available. 
Instead, we can consider the interpolation of a primitive $Q$ of $q$. More specifically, the point values of $Q$ are available at cell interfaces as
\begin{align}
	Q(\xi_{s-\frac{1}{2}})&=0, \label{eq:arbitrary_initial_value}\\
	Q(\xi_{s-\frac{1}{2}+k})&=\sum_{m=0}^{k-1} \qbar_{s+m}, \quad k=1,\dots,N_{\mathcal{S}}.
\end{align}
Thus, the interpolation polynomial $Q_h(\xi)$ of degree $N_{\mathcal{S}}$ associated to the cell interfaces involved in the stencil $\mathcal{S}$ can be constructed as
\begin{align}
	Q_h(\xi):=\sum_{k=0}^{N_{\mathcal{S}}}Q(\xi_{s-\frac{1}{2}+k})\varphi_k(\xi),
\end{align}
where $\varphi_k$ are the Lagrange polynomials of degree $N_{\mathcal{S}}$ associated to the cell interfaces involved in the interpolation.
Its derivative yields the desired approximation of $q$ 
\begin{align}
	q_h(\xi):=\sum_{k=0}^{N_{\mathcal{S}}}Q(\xi_{s-\frac{1}{2}+k})\frac{d}{d \xi}\varphi_k(\xi).
\end{align}

	Notice that if $N_{\mathcal{S}}$ is the number of cells involved in the stencil, then, we are actually considering $N_{\mathcal{S}}+1$ interface values for the interpolation of $Q$, leading in smooth cases to $(N_{\mathcal{S}}+1)$-th order of accuracy for $Q_h$.
	Therefore, its derivative $q_h$, a polynomial of degree $N_{\mathcal{S}}-1$, is $N_{\mathcal{S}}$-th order accurate.
Observe also that if we fix some point $\xi^*\in[\xi_{\imh},\xi_{\iph}]$, then, the values of the derivative of the Lagrange polynomials, $\frac{d}{d \xi}\varphi_k$, are prescribed and $q_h(\xi^*)$ is simply a linear combination of the cell averages involved in the stencil. This is crucial for the development of the current discussion. In fact, the described construction can be applied to the high order stencil~\eqref{eq:HO_stencil} and to all the low order stencils~\eqref{eq:LO_sentcil}, obtaining in a given point $\xi^*$ multiple approximations for $q(\xi^*)$:
a $(2r-1)$-th order accurate approximation $q_h^{HO}(\xi^*)$, with $q_h^{HO}$ polynomial of degree $2r-2$, 
and $r$ $r$-th order accurate approximations $q_h^{\ell,LO}(\xi^*)$ $\ell=0,\dots,r-1$, with $q_h^{\ell,LO}$ polynomial of degree $r-1$ for any $\ell$.
We look then for some coefficients $d_\ell^{\xi^*}$, which depend on $\xi^*$, such that the linear combination of the low order approximations through these coefficients gives the high order one
\begin{equation}
	q_h^{HO}(\xi^*)=\sum_{\ell=0}^{r-1}d_\ell^{\xi^*} q_h^{\ell,LO}(\xi^*).
	\label{eq:linear_weights}
\end{equation}
This amounts at solving an over--determined linear system in the unknown coefficients $d_\ell^{\xi^*}$.
Such coefficients are commonly referred to as ``linear weights''. The WENO method is based on using~\eqref{eq:linear_weights} with suitably modified weights, namely, the approximation adopted reads
\begin{equation}
	q_h^{WENO}(\xi^*)=\sum_{\ell=0}^{r-1}\omega_\ell^{\xi^*} q_h^{\ell,LO}(\xi^*),
	\label{eq:WENO_approximation}
\end{equation}
where the coefficients $\omega_\ell^{\xi^*}$ are defined in such a way to retrieve the linear weights in smooth cases and to select the approximations associated to the smoothest stencils in case of non--regular solutions.
The most popular definition for such coefficients, usually called ``nonlinear weights'', is~\cite{shu1998essentially}
\begin{equation}
\omega_\ell^{\xi^*} := \frac{\alpha_\ell^{\xi^*}}{\sum^{r-1}_{k=0}\alpha_k^{\xi^*}},\quad \alpha_\ell^{\xi^*} := \frac{d_\ell^{\xi^*}}{(\beta_\ell+\epsilon_{{\small \text{WENO}}})^2},
\end{equation}
where $\epsilon_{{\small \text{WENO}}}$ is a small constant used to prevent divisions by zero and $\beta_m$ are smoothness indicators associated to the stencils
\begin{equation}
	\beta_\ell := \sum_{k=1}^{r-1} \int_{\xi_{i-1/2}}^{\xi_{i+1/2}} \left(\frac{d^k}{d\xi^k} q_h^{\ell,LO}(\xi)\right)^2 \Delta\xi^{2k-1}\diff{\xi},\quad \ell = 0,\ldots,r-1.
\end{equation}
The coefficient $\epsilon_{{\small \text{WENO}}}$ is commonly~\cite{shu1998essentially,jiang1996efficient,ciallella2022arbitrary,ciallella2023arbitrary,ciallella2024high} set to $10^{-6}$ and this is the value we assume also here.

\begin{remark}[On optimality and symmetry of stencils]
	The described procedure, for fixed $r$, considers a stencil of $2r-1$ neighboring cells to achieve the optimal order of accuracy $2r-1$.
	This results in a reconstruction with odd order of accuracy only.
	In principle, one could extend the procedure to obtain reconstructions with even order, however, this would create a conflict between symmetry and optimality of the stencil.
	In fact, in order to obtain a $k$-th order accurate reconstruction, $k$ pieces of information are required. In the FV framework, this amounts to considering a stencil of at least $k$ cells.
	It is rather reasonable to consider a stencil of $k$ neighboring cells, including the cell in which the solution has to be reconstructed.
	If $k$ is even, there are two options and both of them are non--symmetric with respect to such a cell.
	Symmetry can only be achieved relaxing the optimality constraint and considering a symmetric stencil with $k+1$ cells. This is the case for the well-known second order minmod-type reconstructions~\cite{chakravarthy1983high,vanleer1977towardsiii,roe1986characteristic,sweby1984high,van1982comparative,van1974towards}.
	Therefore, for FV methods of even order, assuming a stencil of neighboring cells containing the cell in which the reconstruction has to be performed, either the stencil is non--optimal or it is non--symmetric.
\end{remark}

\subsubsection{Multidimensional case}\label{sec:WENO2D}
Let us focus on the case of two dimensions, as this is the case of interest in this work. The extension to higher dimensions follows analogous guidelines.

In order to reconstruct the local values 
\begin{align}
	q^R_{\imh}(y_q) \approx q(x_\imh^+,y_q)\;,\;\;\;q^L_{\iip}(y_q) \approx q(\xip^-,y_q),
\end{align}
in the cell $C_{i,j}$, needed to compute the numerical fluxes, one considers a stencil of $(2r-1)\times(2r-1)$ cells
\begin{equation}
	\label{eq:HO_stencil_2d}
	\mathcal{S}:=\left\lbrace C_{\ell,m}, \quad   \ell=i-(r-1),\dots,i+(r-1), \quad m=j-(r-1),\dots,j+(r-1) \right\rbrace.
\end{equation}
Then, for each $m=j-(r-1),\dots,j+(r-1)$, one performs a first one--dimensional WENO sweep in the $x$-direction, reconstructing the averages at the cell interfaces $x_{i+1/2}$ with respect to the $y$-direction
\begin{align}
	\qbar^R_{m} \approx \frac{1}{\Delta y}\int_{y_{m-\frac{1}{2}}}^{y_{m+\frac{1}{2}}}q(x_\imh^+,y)dy,\quad 
	\qbar^L_{m} \approx \frac{1}{\Delta y}\int_{y_{m-\frac{1}{2}}}^{y_{m+\frac{1}{2}}} q(\xip^-,y)dy .
\end{align}
This is followed by another WENO sweep in the $y$-direction, taking in input $\qbar^R_{m}$ $m=j-(r-1),\dots,j+(r-1)$ to reconstruct $q^R_{\imh}(y_q)$, and taking in input $\qbar^L_{m}$ $m=j-(r-1),\dots,j+(r-1)$ to reconstruct $q^L_{\iph}(y_q)$.
An analogous strategy, performing a first sweep in the $y$-direction and a second one in the $x$-direction is used to compute the needed approximations in the faces shared among the elements in the $y$-direction.

All the needed WENO coefficients used in the context of this work have been computed through the Matlab script provided in~\cite{ourrepo} and produced in the context of~\cite{ciallella2022arbitrary}.

\begin{remark}[On negative linear weights]\label{rmk:negative_weights}
	As underlined in~\cite{shi2002technique,ciallella2022arbitrary}, for a given surface quadrature one may get negative linear weights, which affect in a negative way the stability of the scheme.
	In such cases, particular care and special treatments are required.
	As in~\cite{ciallella2022arbitrary}, the four--point Gauss--Legendre quadrature rule has been adopted for WENO5 in two space dimensions for the surface integrals in order to avoid negative linear weights occurring for the three--point one.
	In all other two--dimensional cases, the Gauss--Legendre quadrature formula with the minimal number of points needed in order to reach the desired accuracy has been considered for the surface integrals.
	%
	%
	%
	%
	Let us remark that, as stated in~\cite{shi2002technique}, negative linear weights do not occur in one--dimensional FV WENO reconstructions for any order of accuracy, as well as, in finite difference WENO reconstructions in any spatial dimension for any order of accuracy.
\end{remark}

For more details, the reader can consult \cite{titarev2004finite,shu1998essentially,ciallella2022arbitrary,lore_phd_thesis}.

\subsubsection{Reconstruction of characteristic variables}\label{sec:characteristic_WENO}

Applying the reconstruction directly to the conserved variables is well known~\cite{qiu2002construction,miyoshi2020short,peng2019adaptive,ghosh2012compact} to cause oscillations in high order space discretizations, therefore, it is recommended to apply it to the characteristic ones. 
Namely, the reconstruction should be applied to the quantities of the vector $L_{\uvec{n}}\ubar_{i,j}$ and, afterwards, the reconstructed variables should be multiplied by $R_{\uvec{n}}$, where $R_{\uvec{n}}$ is the matrix of the right eigenvectors along the direction $\uvec{n}$ of the reconstruction and $L_{\uvec{n}}:=R^{-1}_{\uvec{n}}$, its inverse, is the matrix of the left eigenvectors along the same direction.
More in detail, for each local reconstruction in the generic cell $C_{i,j}$, when reconstructing characteristic variables, the matrices $R_{\uvec{n}}$ and $L_{\uvec{n}}$ are ``frozen'', with $R_{\uvec{n}}:=R_{\uvec{n}}(\ubar_{i,j})$ and $L_{\uvec{n}}:=L_{\uvec{n}}(\ubar_{i,j})$.
Clearly, for each one--dimensional sweep, the proper matrices, corresponding to the direction along which the sweep is performed, must be used. 
Indeed, this option is more expensive than the simple direct reconstruction of conserved variables and does not increase the order of accuracy, however, it results essential for obtaining non--oscillatory results in tests involving discontinuous solutions. 

\subsection{Numerical fluxes}\label{sec:fluxes}
In this section, we present the numerical fluxes under investigation, but let us give some general information on this topic first.
In determining the numerical flux at the interface, one counts on two main ingredients, namely, the governing PDE~\eqref{eq:sys} in the normal direction to the interface and the reconstructions at left and right sides of the interface.
These two items precisely define the initial value problem called ``Riemann problem''~\cite{riemann1860fortpflanzung,Godunov,ToroBook}. 
Any expression for the numerical flux will necessarily make use of the Riemann problem, even if (very) approximately.

Numerical fluxes can broadly be divided into two major classes, namely, centred (or non--upwind) fluxes and upwind fluxes. 
Centred fluxes do not explicitly use wave propagation information emerging from the solution of the Riemann problem. 
Examples among the numerical fluxes considered here include LxF and FORCE. 
On the contrary, upwind fluxes explicitly use wave propagation information emerging from the Riemann problem. 
Top in the hierarchy of upwind fluxes is the flux obtained from the exact solution of the Riemann problem, Ex.RS, as first proposed by Godunov~\cite{Godunov}; other examples included here are Rus, HLL, CU, LDCU and HLLC. 
Upwind fluxes are further classified as complete or incomplete. 
A complete upwind numerical flux adopts a wave model that includes all the characteristic fields of the exact problem. An incomplete upwind numerical flux uses a wave model with less waves than characteristic fields in the exact problem. 
Rus is incomplete for any system of equations, i.e., from two equations on; 
HLL, CU and LDCU are incomplete for any system of more than two equations; 
HLLC is complete for the Euler equations in one, two and three space dimensions, because the intermediate wave incorporated into the wave model represents the multiplicity (3) that includes contact discontinuities and shear waves in the transverse directions. This is also applicable to the three--dimensional Euler equations for multicomponent flows~\cite{ToroBook}. HLLC-type numerical fluxes have been constructed also for larger systems, see~\cite{tokareva2010hllc}.
Finally, by definition, Ex.RS is also complete.
Generally, upwind fluxes perform better than centred fluxes, however, the latter ones are simpler to implement for general hyperbolic systems. Among the upwind fluxes, complete ones tend to give better results, especially when resolving slowly moving waves associated with intermediate, linear characteristic fields; examples include contact discontinuities, shear waves, and material interfaces.

The 8 numerical fluxes under investigation are detailed in the following.
Let us focus on the flux along the $x$-direction, but analogous definitions hold also for the one along the $y$-direction.
\begin{itemize}
	\item[1)] \textbf{Lax--Friedrichs (LxF)}\\
	The LxF numerical flux has been introduced in~\cite{lax1954weak} and it reads
	\begin{equation}
		\widehat{\bF}^{\text{LxF}}({\bu}^L,{\bu}^R) := \frac{1}{2}\left(\bF({\bu}^R) + \bF({\bu}^L)\right) - \frac{1}{2}\frac{\Delta x}{\Delta t}\left({\bu}^R - {\bu}^L\right).
	\end{equation}
	The main advantages are its cheap expression and its simplicity, which does require any knowledge of the Riemann problem structure.
	However, it is characterized by a (very) high level of diffusion, in fact it is the most diffusive monotone flux in the context of scalar problems, which makes it not suitable for real applications.

	\item[2)] \textbf{First--Order Centred (FORCE)}\\	
	The FORCE numerical flux, introduced in~\cite{Toro1996,toro2000centred,chen2003centred}, is defined as the average between the LxF numerical flux and the Richtmyer (or two--step Lax--Wendroff) numerical flux~\cite{richtmyer1967difference}, and reads
	\begin{equation}
		\widehat{\bF}^{\text{FORCE}}({\bu}^L,{\bu}^R) := \frac{1}{2}\left[\widehat{\bF}^{LxF}({\bu}^L,{\bu}^R) + \widehat{\bF}^{\text{Richtm}}({\bu}^L,{\bu}^R)\right], 
	\end{equation}
	where the  Richtmyer (or two--step Lax--Wendroff) numerical flux is given by
	\begin{align}
		\widehat{\bF}^{\text{Richtm}}({\bu}^L,{\bu}^R)&:=\bF(\uvec{u}^*({\bu}^L,{\bu}^R)), \\ \uvec{u}^*({\bu}^L,{\bu}^R)&:=\frac{1}{2}( {\bu}^L+{\bu}^R) - \frac{1}{2}\left[\frac{\Delta t}{\Delta x}(\bF({\bu}^R) - \bF({\bu}^L) \right].
	\end{align}
	Being the average between the LxF numerical flux and the Richtmyer (or two--step Lax--Wendroff) numerical flux, it is less diffusive than LxF~\cite{ToroBook}.
	Moreover, just like LxF, it does not require any element of the Riemann problem solution, which makes it suitable for complicated systems where Riemann solvers may not be available.

	\item[3)] \textbf{Rusanov (Rus)}\\
	This numerical flux has been introduced in~\cite{Rusanov1961}. Its expression is similar to the one of LxF, in fact it is sometimes referred as local LxF or mistakenly as LxF, and it reads
	\begin{equation}
		\widehat{\bF}^{\text{Rus}}({\bu}^L,{\bu}^R) := \frac{1}{2}\left(\bF({\bu}^R) + \bF({\bu}^L)\right) - \frac{1}{2}s\left({\bu}^R - {\bu}^L\right),
	\end{equation}
	where $s$ is the maximum in absolute value of the local wave speed in the $x$-direction associated to the states ${\bu}^L$ and ${\bu}^R$.
	This numerical flux is still rather diffusive but not as much as LxF.
	Due to its simple expression, which allows for an explicit handling very useful in analytical proofs, and due to its robustness it is broadly used in many applications.

	\item[4)] \textbf{Harten--Lax--van Leer (HLL)}\\
	This numerical flux was introduced in~\cite{harten1983upstream} and its derivation is more involved than the one of the previous numerical fluxes. In fact, it is based on a two--wave model approximate Riemann solver~\cite{ToroBook}. Its expression reads
	\begin{equation}
		\widehat{\bF}^{\text{HLL}}({\bu}^L,{\bu}^R):=\begin{cases}
			\bF({\bu}^L),   & \text{if} \quad 0\leq s^L, \\  
			\frac{s^R\bF({\bu}^L)-s^L\bF({\bu}^R)+s^L s^R({\bu}^R-{\bu}^L)}{s^R-s^L}, \quad &\text{if} \quad s^L\leq 0\leq s^R,\\
			\bF({\bu}^R),	& \text{if} \quad  0 \geq s^R,
		\end{cases}
	\label{eq:HLL}
	\end{equation}
	where $s^L$ and $s^R$ are estimates for the slowest and fastest local wave speeds from the augmented Riemann problem in the $x$-direction between the states ${\bu}^L$ and ${\bu}^R$.
	Here, for the wave speed estimates, we assume the rigorous analytical bounds from~\cite{toro2020bounds}, obtained through an adaptive approximate--state Riemann solver including the Primitive Variable Riemann solver, the Two--Rarefaction Riemann solver and the Two--Shock Riemann solver, as described in \cite[Section 9.5]{ToroBook}.	
	Being more sophisticated with respect to the previous numerical fluxes, HLL is less diffusive, and it generally provides sharper results.
		
	\item[5)] \textbf{Central--Upwind (CU)}\\
	This numerical flux was derived in~\cite{kurganov2001semidiscrete}, in the context of central schemes for hyperbolic conservation laws, taking the limit as $\Delta t\rightarrow 0$ of the discretization proposed in~\cite{kurganov2000new}, and it reads
	\begin{equation}
		\widehat{\bF}^{\text{CU}}({\bu}^L,{\bu}^R):=\frac{a^R\bF({\bu}^L)-a^L\bF({\bu}^R)+a^L a^R({\bu}^R-{\bu}^L)}{a^R-a^L}, 
	\label{eq:CU}
	\end{equation}	
	where $a^L$ and $a^R$ are estimates for minimum and maximum of the one--sided local wave speeds from the augmented Riemann problem in the $x$-direction between the states ${\bu}^L$ and ${\bu}^R$, that is to say $a^L:=\min{(s^L,0)}$ and $a^R:=\max{(s^R,0)}$, where $s^L$ and $s^R$ have the same meaning as in the HLL numerical flux.
	Simple algebra shows that, even though the derivation is different, this numerical flux is equivalent to the HLL one. Here, in order to distinguish them, we consider for the CU numerical flux the simple wave speed estimates adopted in~\cite{kurganov2001semidiscrete,kurganov2023new,diaz2019path,kurganov2002central,kurganov2020well}, $s^L:=\min{(u^L-c^L,u^R-c^R)}$ and $s^R:=\max{(u^L+c^L,u^R+c^R)}$, initially proposed by Davis~\cite{davis1988simplified}.
	Let us remark that other wave speed estimates exist in literature, e.g., the ones proposed by Einfeldt in \cite{einfeldt1988godunov}.
	
	\item[6)] \textbf{Low--Dissipation Central--Upwind (LDCU)}\\
	This numerical flux has been proposed in~\cite{kurganov2023new} and it is in an improvement of the previous one obtained, essentially, by adding an antidiffusive contribution. Its expression reads
	\begin{equation}
		\widehat{\bF}^{\text{LDCU}}({\bu}^L,{\bu}^R):=\frac{a^R\bF({\bu}^L)-a^L\bF({\bu}^R)+a^L a^R({\bu}^R-{\bu}^L-\delta{\bu})}{a^R-a^L} 
	\label{eq:LDCU}
	\end{equation}	
	where $a^L$ and $a^R$ have the same meaning as in the CU numerical flux, and $\delta\uvec{u}$ is a ``built-in'' antidiffusion term defined as
	\begin{equation}
		\delta\uvec{u}:={\rm minmod}\left(\uvec{u}^R-\uvec{u}^*,\uvec{u}^*-\uvec{u}^L\right),
	\end{equation}
	with $\uvec{u}^*$ being an intermediate value obtained as
	\begin{equation}
		\uvec{u}^*:=\frac{a^R\uvec{u}^R-a^L\uvec{u}^L-\left[\bF({\bu}^R)-\bF({\bu}^L)\right]}{a^R-a^L},
	\end{equation}
	and the minmod function
	\begin{equation*}
		{\rm minmod}(z_1,z_2):=
		\left\{\begin{aligned}
			&\min_{i=1,2}z_i,&&\mbox{if}~z_i>0,\quad i=1,2,\\
			&\max_{i=1,2}z_i,&&\mbox{if}~z_i<0,\quad i=1,2,\\
			&0,&&\mbox{otherwise},
		\end{aligned}
		\right.
	\end{equation*}
	being applied in a component-wise fashion.

	\item[7)] \textbf{HLLC}\\	
	This numerical flux has been firstly proposed in~\cite{toro1992restoration,toro1994restoration}. It is built upon the HLL numerical flux in order to overcome the difficulties of two--wave models in resolving contact surfaces, shear
	waves and material interfaces. In fact, in HLLC (where the ``C'' stays for ``contact''), a more physically accurate three--wave model is assumed~\cite{ToroBook} leading to
	\begin{equation}
	\widehat{\bF}^{\text{HLLC}}({\bu}^L,{\bu}^R):=\begin{cases}
		\bF({\bu}^L),   & \text{if} \quad 0\leq s^L, \\  
		\bF({\bu}^L)+s^L (\uvec{u}^{*L}-\uvec{u}^{L}), \quad &\text{if} \quad s^L\leq 0\leq s^{*},\\
		\bF({\bu}^R)+s^R (\uvec{u}^{*R}-\uvec{u}^{R}), \quad &\text{if} \quad s^{*}\leq 0\leq s^R,\\
		\bF({\bu}^R),	& \text{if} \quad 0 \geq s^R,
		\end{cases}
	\label{eq:HLLC}
	\end{equation}
	where, again, $s^L$ and $s^R$ are estimates for the slowest and fastest local wave speeds from the augmented Riemann problem in the $x$-direction between the states ${\bu}^L$ and ${\bu}^R$, while, $s^{*}$, $\uvec{u}^{*L}$ and $\uvec{u}^{*R}$ are defined as follows
	\begin{align}
		s^{*} &:= \frac{p^R - p^L + \rho^L u^L (s^L - u^L) - \rho^R u^R (s^R - u^R)}{\rho^L (s^L - u^L) - \rho^R (s^R - u^R)},\\
		\mathbf{u}^{*K} &:= \rho^K \left( \frac{s^K - u^K}{s^K - s^{*}} \right)
		\begin{bmatrix}
			1 \\
			s^{*} \\
			v^K \\
			\frac{E^K}{\rho^K} + (s^{*} - u^K) \left[ s^{*} + \frac{p^K}{\rho^K (s^K - u^K)} \right]
		\end{bmatrix},		
	\end{align}
	with $K$ being a generic handle for $L$ and $R$.
	
	Also in this case, as for the HLL, we assume here the same rigorous estimates for $s^L$ and $s^R$ from~\cite{toro2020bounds} computed as described in~\cite[Section 9.5]{ToroBook}.

	\item[8)] \textbf{Exact Riemann solver~(Ex.RS)}\\
	This numerical flux, sometimes referred as Godunov's flux, has been introduced in~\cite{Godunov}, and reads
	\begin{equation}
		\widehat{\bF}^{\text{Ex.RS}}({\bu}^L,{\bu}^R) := \bF({\bu}^*),
	\end{equation}
	where ${\bu}^*$ is the exact solution of the augmented Riemann problem in the $x$-direction between the states ${\bu}^L$ and ${\bu}^R$ for $\xi:=\frac{x}{t}=0$.
	
	This is the most precise numerical flux and the least diffusive monotone flux in the context of scalar problems.
	Its main drawback is that, for general nonlinear systems, it cannot be expressed in closed--form and it requires iterative strategies to be computed, which may increase the computational cost with respect to other numerical fluxes.
	Here, in the context of the Euler equations, we adopt the strategy introduced in~\cite{toro1989fast} and thoroughly described in \cite[Chapter 4]{ToroBook}.	
	
\end{itemize}

\begin{remark}[On speed estimates]
	HLL~\eqref{eq:HLL}, CU~\eqref{eq:CU}, LDCU~\eqref{eq:LDCU} and HLLC~\eqref{eq:HLLC} numerical fluxes require wave speed estimates. 
	According to Harten, Lax and van Leer~\cite{harten1983upstream}, such estimates must bound the true wave speeds.
	In~\cite{toro2020bounds} it is shown that, with one exception, all existing wave speed estimates fail to bound the true wave speeds. 
	The only exception is given by the estimates proposed in~\cite{toro1992restoration,toro1994restoration}, as rigorously proven in~\cite{guermond2016fast}.
\end{remark}

\section{Time discretization}\label{sec:DeC}
The time integration technique adopted here is a DeC scheme, and it is described in this section.
The first publication on DeC is due to Fox and Goodwin in 1949~\cite{fox1949some}. In its original conception, it consisted in a finite difference strategy for initial value problems. The key aspect of the scheme was an iterative procedure involving subsequent corrections of the error with respect to the exact solution.
The approach gained popularity in 2000, when Dutt, Greengard and Rokhlin~\cite{Decoriginal} proposed it in a more modern framework.
They constructed arbitrary high order methods for ODEs based on an iterative procedure gaining one order of accuracy at each iteration.
Since then, many applications, improvements and developments of the approach have followed.
In this context, it is important to mention the works of Minion and collaborators~\cite{minion2003semi,minion2004semi,layton2005implications,huang2006accelerating,minion2011hybrid,speck2015multi}.
In~\cite{Decremi}, Abgrall introduced a new abstract formulation of the DeC approach with application to the time discretization of continuous finite element methods avoiding solutions of linear systems.
Several works have been developed based on this formalism~\cite{abgrall2019high,abgrall2020high,michel2021spectral,michel2023spectral,bacigaluppi2023posteriori,ciallella2022arbitrary,abgrall2021relaxation,ciallella2023arbitrary,abgrall2024staggered,abgrall2020multidimensional,micalizzi2024novel,offner2020arbitrary}. 
In \cite{han2021dec,micalizzi2023efficient,veiga2024improving}, it was shown how ADER methods, both for ODEs and PDEs, can be put in the DeC framework introduced by Abgrall.
Further works based on the DeC approach are~\cite{boscarino2016error,boscarino2018implicit,hamon2019multi,franco2018multigrid,minion2015interweaving,benedusi2021experimental,liu2008strong,ong2020deferred,ketcheson2014comparison,christlieb2009comments,christlieb2010integral}.

Many DeC schemes exist in literature nowadays, as can be inferred from the short non--exhaustive review presented in the previous lines. Here, we consider the explicit DeC approach referred as ``bDeC'' in~\cite{micalizzi2023new,lore_phd_thesis} and based on Abgrall's formalism~\cite{Decremi}. The same method has been also used in other references, see for example~\cite{han2021dec,ciallella2022arbitrary,ciallella2023arbitrary}.

We will start by describing Abgrall's DeC framework, in Section~\ref{sec:DeC_Remi}, and we will continue with the description of the bDeC method for ODEs, in Section~\ref{sec:bDeC}.

\subsection{Abgrall's Deferred Correction formalism}\label{sec:DeC_Remi}
Assume that we want to solve some analytical problem with exact solution $\uex$, 
and assume that we have two operators $\lopdt^1,\lopdt^2: X \rightarrow Y$, 
depending on a same discretization parameter $\Delta$ and associated to two different discretizations of the aforementioned problem, 
defined between the normed vector spaces $(X,\norm{\cdot}_X)$ and $(Y,\norm{\cdot}_Y)$. 

More in detail, we assume the operators to have different nature: $\lopdt^2$ is associated with a high order implicit discretization of the analytical problem, instead, $\lopdt^1$ is associated to a low order explicit one.
Being interested in a high order approximation of the solution of the analytical problem, we would like to solve the operator $\lopdt^2$, namely, we would like to find $\underline{\uvec{u}}_\Delta \in X$ such that $\lopd^2(\underline{\uvec{u}}_\Delta)=\uvec{0}_Y$, however, this task is not easy due to the implicit nature of the operator.
On the other hand, we assume to be able to easily solve the problem $\lopd^1(\underline{\uvec{u}})=\underline{\uvec{r}}$ for given $\underline{\uvec{r}} \in Y$. In particular, when $\underline{\uvec{r}}=\uvec{0}_Y$, we get the solution of the operator $\lopd^1$, which is a low order accurate approximation of the analytical solution.
The situation is the following: we would prefer to solve the operator $\lopd^1$ rather than $\lopd^2$, as this would be much simpler, but this would result in an approximation of the solution not accurate enough.

The next theorem provides an easy recipe to approximate with arbitrary high order of accuracy the solution $\undu_\Delta$ of the operator $\lopd^2$ through an easy explicit iterative procedure.

\begin{theorem}[DeC, Abgrall~\cite{Decremi}]\label{th:DeC}
	Let us assume that the operators $\lopdt^1,\lopdt^2: X \rightarrow Y$ fulfill the following properties:
	\begin{enumerate}
		\item \textbf{Existence of a unique solution to $\lopd^2$} \\
		$\exists ! \,\usol \in X$ solution of $\lopd^2$ such that $\lopd^2(\usol)=\uvec{0}_Y$;
		\item \textbf{Coercivity-like property of $\lopd^1$} \\
		$\exists \,\alpha_1 \geq 0$ independent of $\Delta$ such that
		\begin{equation}
			\norm{\lopd^1(\underline{\uvec{v}})-\lopd^1(\underline{\uvec{w}})}_Y\geq \alpha_1\norm{\underline{\uvec{v}}-\underline{\uvec{w}}}_X, \quad \forall \underline{\uvec{v}},\underline{\uvec{w}}\in X;
			\label{eq:DeC_coercivity}
		\end{equation}
		\item \textbf{Lipschitz-continuity-like property of $\lopd^1-\lopd^2$} \\
		$\exists\, \alpha_2 \geq 0$ independent of $\Delta$ such that
		\begin{equation}
			\norm{\left[\lopd^1(\underline{\uvec{v}})\!-\!\lopd^2(\underline{\uvec{v}})\right]\!-\!\left[\lopd^1(\underline{\uvec{w}})\!-\!\lopd^2(\underline{\uvec{w}})\right]}_Y\!\leq \!\alpha_2 \Delta \!\norm{\underline{\uvec{v}}-\underline{\uvec{w}}}_X, \quad\forall \underline{\uvec{v}},\underline{\uvec{w}}\in X.
			\label{eq:DeC_lipschitz}
		\end{equation}
	\end{enumerate}
	For a given $\underline{\uvec{u}}^{(0)}\in X$, let us consider the sequence of vectors $\left\lbrace\underline{\uvec{u}}^{(p)}\right\rbrace_{p\in\mathbb{N}}\subseteq X$ obtained by solving the following iterative procedure
	\begin{equation}
		\label{eq:DeC_iteration}
		\lopd^1(\underline{\uvec{u}}^{(p)}):=\lopd^1(\underline{\uvec{u}}^{(p-1)})-\lopd^2(\underline{\uvec{u}}^{(p-1)}), \quad p\geq 1.
	\end{equation}
	Then, the following error estimate holds
	\begin{equation}
		\label{eq:DeC_accuracy}
		\norm{\underline{\uvec{u}}^{(p)}-\usol}_X \leq \left( \Delta \frac{\alpha_2}{\alpha_1} \right)^p\norm{\underline{\uvec{u}}^{(0)}-\usol}_X, \quad\forall p\in \mathbb{N}. 
	\end{equation}
\end{theorem}
\begin{proof}
	Making direct use of the previous assumptions we have that
	\begin{align}
			\norm{\underline{\uvec{u}}^{(p)}-\usol}_X & \leq \frac{1}{\alpha_1}  \norm{\lopd^1(\underline{\uvec{u}}^{(p)})-\lopd^1(\usol)}_Y \label{eq:proof_dec_1}\\
			&= \frac{1}{\alpha_1}  \norm{\lopd^1(\underline{\uvec{u}}^{(p-1)})-\lopd^2(\underline{\uvec{u}}^{(p-1)})-\lopd^1(\usol)}_Y \label{eq:proof_dec_2}\\
			&= \frac{1}{\alpha_1}  \norm{ \left[ \lopd^1(\underline{\uvec{u}}^{(p-1)})-\lopd^2(\underline{\uvec{u}}^{(p-1)})\right]-\left[\lopd^1(\usol)-\lopd^2(\usol)\right]}_Y\label{eq:proof_dec_3}\\
			&\leq \Delta \frac{\alpha_2}{\alpha_1}  \norm{\underline{\uvec{u}}^{(p-1)}-\usol}_X.\label{eq:proof_dec_4}
	\end{align}
	In particular, in~\eqref{eq:proof_dec_1} we have used the property \eqref{eq:DeC_coercivity} of the operator $\lopd^1$, 
	in~\eqref{eq:proof_dec_2} we have used the definition~\eqref{eq:DeC_iteration} of the DeC iteration, 
	in~\eqref{eq:proof_dec_3} we have used the fact that $\usol$ is the solution of the operator $\lopd^2$, 
	and finally in~\eqref{eq:proof_dec_4} we have used the property~\eqref{eq:DeC_lipschitz} of $\lopd^1-\lopd^2$.
	The same arguments can be applied recursively to obtain~\eqref{eq:DeC_accuracy}.
\end{proof}

Let us notice that solving the generic iteration~\eqref{eq:DeC_iteration} with respect to $\underline{\uvec{u}}^{(p)}$ is easy due to our assumptions on the operator $\lopd^1$, as the right-hand side can be explicitly computed.
Moreover, we have that, for any given $\underline{\uvec{u}}^{(0)}$, the sequence of $\left\lbrace\underline{\uvec{u}}^{(p)}\right\rbrace_{p\in\mathbb{N}}$ converges to $\usol$ provided that $\Delta$ is small enough.
Let us also observe that a direct consequence of~\eqref{eq:DeC_accuracy} is that the sequence gains one order of accuracy at each iteration with respect to $\usol$.
Since we are actually interested in approximating the exact solution $\uex$, and not strictly in approximating the solution $\usol$ of the operator $\lopdt^2$, it does not make sense to perform the iterative procedure until convergence up to machine precision.
More specifically, if $\usol$ is $P$-th order accurate with respect to $\uex$ and $\underline{\uvec{u}}^{(0)}$ is an $O(\Delta)$-approximation of $\undu_{ex}$, the optimal number of iterations is $P$. After that, even though the accuracy with respect to $\usol$ will increase, the accuracy with respect to $\undu_{ex}$ will stay $P$.
Namely, the accuracy of $\underline{\uvec{u}}^{(p)}$ with respect to $\uex$ is $\min{(p,P)}$.
Therefore, here, we always consider a number of iterations equal to the formal order of accuracy of the method.

Further developments of the described framework allow to indefinitely increase the order of accuracy without saturation, but they will not be considered here. The interested reader is referred to~\cite{micalizzi2023efficient,micalizzi2023new,veiga2024improving}. 


\subsection{bDeC for ordinary differential equations}\label{sec:bDeC}
Let us consider the initial value problem
\begin{equation}
	\label{eq:ODE}
	\begin{cases}
		\frac{d}{dt}\uvec{u}(t) = \uvec{G}(t,\uvec{u}(t)),\quad t\in[0,T_f], \\
		\uvec{u}(0)=\uvec{z},
	\end{cases}
\end{equation}
with unknown solution $\uvec{u}:[0,T_f] \rightarrow \mathbb{R}^{N_c}$, $N_c\in \mathbb{N}^+$ being the number of components of the system of ODEs, final time $T_f\in \mathbb{R}^+$, initial condition $\uvec{z} \in \R^{N_c}$, and right-hand side function $\uvec{G}: [0,T_f] \times \R^{N_c} \to \R^{N_c}$.
We assume $\uvec{G}$ to satisfy the classical smoothness constraints guaranteeing the existence of a unique solution, i.e., $\uvec{G}$ is required to be Lipschitz-continuous with respect to $\uvec{u}$ uniformly with respect to $t$ with a Lipschitz constant $C_{Lip}$.

As customary in the context of one--step methods, we focus on the generic time step $[t_n,t_{n+1}]$, with $\Delta t:=t_{n+1}-t_n$, and we try to get $\uvec{u}_{n+1}\approx \uvec{u}(t_{n+1})$ by knowing $\uvec{u}_{n}\approx \uvec{u}(t_{n})$. This is done by introducing $M+1$ subtimenodes $t^m$ $m=0,\dots,M$ in the considered time interval such that  
\begin{equation}
	t_n=:t^0<t^1<\dots<t^M:=t_{n+1}.
\end{equation}
The order of the method is proportional to the number of subtimenodes. In the context of this work, we assume Gauss--Lobatto subtimenodes which lead to higher accuracy with respect to equispaced ones for fixed number. More specifically, $M+1$ equispaced subtimenodes are sufficient to achieve accuracy $M+1$, while, $M+1$ Gauss--Lobatto subtimenodes can achieve accuracy $2M$.
We will adopt the following notation:
$\uvec{u}(t^m)$ represents the exact solution to the ODE in $t^m$, while, $\uvec{u}^m$ is an approximation of the same quantity in the same subtimenode.
For the first subtimenode only, we set $\uvec{u}^0:=\uvec{u}_n$.

The scheme is constructed upon the integral version of the original ODE system~\eqref{eq:ODE} in each interval $[t^0,t^m]$
\begin{equation}
	\label{exint}
	\uvec{u}(t^m)-\uvec{u}(t_n)-\int_{t^0}^{t^m}\uvec{G}(t,\uvec{u}(t))dt=\uvec{0}, \quad m=1,\dots,M.
\end{equation}
In such a context, the DeC operators are nonlinear endomorphisms defined over $X=Y:=\R^{(M\times N_c)}$, which represents the space of the values of the solution in all subtimenodes but the first one, $t^0:=t_n$, for which the solution is known.

The implicit high order operator $\lopdt^2:\R^{(M\times N_c)}\to\R^{(M\times N_c)}$ is defined by approximating the integral over each interval through the high order quadrature formula associated to the subtimenodes
\begin{equation}
	\label{l2ODE}
	\lopdt^2(\underline{\uvec{u}}) = \begin{pmatrix}
		\uvec{u}^1-\uvec{u}_n-\Delta t \sum_{\ell=0}^{M} \theta^1_\ell \uvec{G}(t^\ell,\uvec{u}^\ell)\\
		\vdots\\
		\uvec{u}^M-\uvec{u}_n-\Delta t \sum_{\ell=0}^{M} \theta^M_\ell \uvec{G}(t^\ell,\uvec{u}^\ell)\\
	\end{pmatrix},
	\quad
	\underline{\uvec{u}}=\left(
	\begin{array}{ccc}
		\uvec{u}^1\\
		\vdots\\
		\uvec{u}^M
	\end{array}
	\right),
\end{equation}
where $\theta^m_\ell$ are the normalized quadrature weights.
The explicit low order operator $\lopdt^1:\R^{(M\times N_c)}\to\R^{(M\times N_c)}$ is instead obtained through a simple Euler approximation
\begin{equation}
	\label{l1ODE}
	\lopdt^1(\underline{\uvec{u}}) = \begin{pmatrix}
		
		\uvec{u}^1-\uvec{u}_n-\Delta t \beta^1 \uvec{G}(t_n,\uvec{u}_n)\\
		\vdots\\
		\uvec{u}^M-\uvec{u}_n-\Delta t \beta^M \uvec{G}(t_n,\uvec{u}_n)\\
	\end{pmatrix},
\end{equation}
where $\beta^m:=\frac{t^m-t_n}{\Delta t}$.

One can observe that solving the problem $\lopdt^2(\usol)=\uvec{0}$ is rather difficult, as it consists of a nonlinear system in all the unknown approximated values of the solution. Actually, the problem is equivalent to a fully implicit RK scheme and, in particular, for Gauss--Lobatto subtimendodes, one obtains the well--known LobattoIIIA methods~\cite{hairer1987solving}.
On the other hand, solving $\lopdt^1(\bbu)=\undr$ with given $\undr \in \R^{(M\times N_c)}$ is straightforward and, for $\undr=\uvec{0}$, its solution is first order accurate.

One can prove~\cite{lore_phd_thesis,micalizzi2023new} that the previously defined operators~\eqref{l2ODE} and~\eqref{l1ODE} fulfill the hypotheses of Theorem~\ref{th:DeC}. 
Then, introducing the vector $\underline{\uvec{u}}^{(p)} \in \R^{(M\times N_c)}:=\left(\uvec{u}^{1,(p)}, \dots, \uvec{u}^{M,(p)} \right)^T$, made by $M$ components $\uvec{u}^{m,(p)}\in \R^{N_c}$, associated to the subtimenodes $t^m$ $m=1,\dots,M$, the generic DeC iteration~\eqref{eq:DeC_iteration} reduces to
\begin{align}\label{eq:DeCODE_Remi}
	\uvec{u}^{m,(p)} = \uvec{u}_n+\Delta t \sum_{\ell=0}^{M} \theta^m_\ell \uvec{G}(t^\ell,\uvec{u}^{\ell,(p-1)}), \quad m=1,\dots,M.
\end{align}
We set $\uvec{u}^{m,(p)}:=\uvec{u}_n$ whenever we have $m=0$ or $p=0$, so to define the initial vector with all components equal to $\uvec{u}_n$, and we perform $P$ iterations, equal to the desired order of accuracy, after which we set $\uvec{u}_{n+1}:= \uvec{u}^{M,(P)}$.
In particular, in order to have accuracy $P$, we consider $M+1$ subtimenodes with $M=\left \lceil \frac{P}{2}\right \rceil $, as Gauss--Lobatto subtimenodes are assumed.
As shown in~\cite{micalizzi2023new,lore_phd_thesis}, bDeC methods can be regarded as explicit RK methods with $M(P - 1) + 1$ stages.

In the following, in order to lighten the notation especially in labels, we will refer to the described bDeC method simply as ``DeC''.

\section{Numerical results}\label{sec:numerical_results}
In this section, we test WENO--DeC with the numerical fluxes reported in Section~\ref{sec:fluxes} on several benchmarks, involving smooth and non--smooth solutions.
In particular, we provide results for the one--dimensional and two--dimensional Euler equations in Sections~\ref{sec:Euler_1d} and~\ref{sec:Euler_2d} respectively. 
The problems are carefully selected to put in evidence the differences among the investigates numerical fluxes.
Before going to the results, let us give a brief description of some elements adopted in the context of the numerical tests.

The time step is here computed as 
\begin{equation}
	\Delta t:= C_{CFL} \min{\left(\frac{\Delta x}{\max{(s^x)}},\frac{\Delta y}{\max{(s^y)}}\right)},
	\label{eq:CFL}
\end{equation}
where $s^x$ and $s^y$ represent estimates of the maximum local wave speeds in absolute value along the $x$- and $y$-direction respectively, and $C_{CFL}$ is a constant.
In particular, here we consider a direct estimate with $s^x:=\abs{u}+\sqrt{\gamma \frac{p}{\rho}}$ and $s^y:=\abs{v}+\sqrt{\gamma \frac{p}{\rho}}$, where $\rho,u,v,p$ are obtained from the cell averages at time $t_n$.
As a general rule, we tried to run each test starting by $C_{CFL}=0.95$ and $C_{CFL}=0.45$ in one and two dimensions respectively, lowering $C_{CFL}$ when simulation crashes occurred. Let us remark that the associated (linear) stability overbounds are given by $C_{CFL}^{max}=1$ and $C_{CFL}^{max}=0.5$.

\begin{remark}[On the time step computation]
	Note that the simple wave speed estimates for $s^x$ and $s^y$ based on the naive evaluation of the required quantities from the cell averages at time $t_n$, adopted here as in many other references, may fail to capture the correct wave speeds and may therefore lead to violation of the Courant--Friedrichs--Lewy (CFL) stability condition even for~$C_{CFL}<C_{CFL}^{max}$ triggering instabilities, for example in Riemann problems with zero initial velocity.
	If one adopts rigorous wave speed estimates bounding the real wave speeds, the linear CFL stability constraints mentioned, $C_{CFL}^{max}=1$ and $C_{CFL}^{max}=0.5$ in one and two space dimensions, are strictly valid for first order FV methods obtained with all investigated numerical fluxes, with two exceptions: 
	in the two--dimensional case, first order methods obtained with both LxF and FORCE are linearly unstable in the setting of simultaneous update formulas~\cite{toro2000centred}, such as~\eqref{eq:FV_semidiscretization}.
	Linearly stable versions of both LxF and FORCE in multiple space dimensions on unstructured meshes have been constructed in~\cite{toro2009force} and in~\cite{dumbser2010force}, in conservative and non--conservative setting respectively.
\end{remark}

In all reported simulations, the order of the time discretization is chosen in such a way to match the one of the space discretization, and we apply the WENO reconstruction to characteristic variables, as outlined in Section~\ref{sec:characteristic_WENO}.
Furthermore, let us recall that, as described in Remark~\ref{rmk:negative_weights}, the four--point Gauss--Legendre quadrature rule has been adopted for WENO5 in two space dimensions for the surface integrals. Instead, in all other two--dimensional cases, we considered the Gauss--Legendre quadrature formula with the minimal number of points needed in order to reach the desired accuracy.
%
In the presented tests, initialization consists in the simple approximation of cell averages via numerical integration at the beginning of the simulation. The quadrature formula adopted is the Gauss--Legendre quadrature formula with the minimal number of points sufficient to achieve the desired spatial accuracy.
%
Finally, for what follows, it is useful to introduce the concept of ``efficiency'', representing the error with respect to the computational time.
We say that a scheme is more efficient than a competitor if it is able to achieve smaller error for the same computation time, conversely, we say that a scheme is less efficient than a competitor if it realizes bigger error for the same computation time.

\subsection{One--dimensional Euler equations}\label{sec:Euler_1d}
In this section, we perform tests on the one--dimensional Euler equations. 
We start by the simple advection of a smooth density, in Section~\ref{sec:Euler_1d_sin4}, in order to verify the order of accuracy of the discretizations, and to compare the performance of the numerical fluxes on smooth problems.
Then, we continue with tests involving discontinuities.
In particular, we consider selected problems, which allow to better put in evidence the impact of the numerical flux.
In Sections~\ref{sec:Euler_1d_RP1}, \ref{sec:Euler_1d_RP5}, \ref{sec:Euler_1d_RP6} and~\ref{sec:Euler_1d_RP7}, we compare the schemes on four challenging Riemann problems taken from~\cite[Section 10.8]{ToroBook}.
%
%
These are very tough tests with different nature, specifically designed to evaluate various features of numerical schemes. 
A high level of robustness from numerical approaches is required in order to be able to handle them, and many classical methods experience simulation crashes in this context.
As mentioned in the reference, even when no crashes occur, spurious oscillations are not uncommon in numerical simulations on these tests.
For convenience, we report the related test informations in Table~\ref{tab:Euler_1d_RP}, and we refer to the tests through their original test numbers from~\cite[Section 10.8]{ToroBook}.
For all of them, the domain is $[0,1]$, transmissive boundary conditions are used and $x_d$ represents the initial location of the discontinuity.

\begin{table}
	\centering
	\begin{tabular}{|c||c|c|c||c|c|c||c||c|}
		\hline
		Test& $\rho_L$ & $u_L$ & $p_L$& $\rho_R$& $u_R$ & $p_R$ & $x_d$ & $T_f$\\\hline\hline
		1   &      1.0 &  0.75 &   1.0&    0.125&   0.0 &  0.1  &  0.3  & 0.2 \\
   		\hline
		5   &     1.0  &-19.59745&1000.0& 1.0 & -19.59745 & 0.01 &  0.8 &  0.012\\
   		\hline
		6   &     1.4  & 0.0    & 1.0 & 1.0 & 0.0 & 1.0 &  0.5 &  2.0\\
   		\hline
		7   &     1.4  & 0.1    & 1.0 & 1.0 & 0.1 & 1.0 &  0.5 &  2.0\\
		\hline
	\end{tabular}
	\caption{Test informations for Riemann problems 1, 5, 6 and 7 from~\cite[Section 10.8]{ToroBook}}\label{tab:Euler_1d_RP}
\end{table}

In Section~\ref{sec:Euler_1d_titarev_toro}, we report results for a shock--turbulence interaction problem. In particular, we consider the tough modification proposed in~\cite{titarev2004finite} of the original test introduced by Shu and Osher in~\cite{shu1989efficient}.
Finally, in Section~\ref{sec:Euler_1d_omitted_results}, we report some comments concerning numerical results on tests that have been omitted in this work for the sake of compactness and clarity.

The exact reference solutions to the problems discussed in Sections~\ref{sec:Euler_1d_RP1}, \ref{sec:Euler_1d_RP5}, \ref{sec:Euler_1d_RP6} and \ref{sec:Euler_1d_RP7} have been computed using the library NUMERICA~\cite{toro1999numerica}.
For the test reported in Section~\ref{sec:Euler_1d_titarev_toro}, no exact solution is available. The associated reference solution has been obtained by adopting a FV method with second order accurate van Leer's minmod spatial discretization~\cite{AbgrallMishranotes} and SSPRK2 on a very refined mesh of 200,000 cells,
$C_{CFL}:=0.5$, reconstruction of characteristic variables and exact Riemann solver as numerical flux.

\subsubsection{Advection of smooth density}\label{sec:Euler_1d_sin4}
This test is taken from~\cite{toro2005tvd} and is meant to assess the order of accuracy of the discretizations under investigation and the performance of different numerical fluxes for regular solutions. 
On the computational domain $\Omega:=[-1,1]$ with periodic boundary conditions, we consider the advection of a smooth density profile. 
Namely, we consider the initial condition (given in terms of primitive variables)
\begin{align}
	\begin{cases}
		\rho_0(x)&:=2+\sin^4{\left(\pi x\right)},\\
		u_0(x)&:=u_\infty, \\
		p_0(x)&:=p_\infty,
	\end{cases}
\end{align}
with $u_\infty:=1$ and $p_\infty:=1.$ The corresponding exact solution is given by $\uvec{u}(x,t):=\uvec{u}_0(x-u_{\infty}t)$, which concretely speaking consists in an advection of the density.
Even though the test may look ``simple'', it does not lack challenging aspects:
let us, in fact, remark that ENO schemes are well known to exhibit weak performance on this problem~\cite{rogerson1990numerical,shu1990numerical}.

As in~\cite{toro2005tvd}, we consider here a final time $T_f:=2$, corresponding to one period of the solution, and $C_{CFL}:=0.95.$
The errors reported here are related to the density, but analogous results have been obtained for all other variables and are hence omitted to save space.

We ran convergence analyses with all the 8 numerical fluxes for orders 3, 5 and 7.
Convergence tables displaying errors in $L^1$-, $L^2$- and $L^{\infty}$-norms along with computational times, for meshes of $N$ elements, are reported in Tables~\ref{tab:Euler_1d_sin4_convergence_table_WENO3_DeC3},~\ref{tab:Euler_1d_sin4_convergence_table_WENO5_DeC5} and~\ref{tab:Euler_1d_sin4_convergence_table_WENO7_DeC7}.
The expected order of accuracy has been obtained in all cases in all the three considered norms.
Let us remark that such a desirable feature is not achieved by other approaches, especially for what concerns the $L^{\infty}$-norm.
Let us notice that WENO3--DeC3 asymptotically shows a slightly superconvergent character.

\begin{table}[htbp]
	\centering
	\caption{One--dimensional Euler equations, Advection of smooth density: convergence tables for WENO3--DeC3}
	\label{tab:Euler_1d_sin4_convergence_table_WENO3_DeC3}
	\scalebox{0.65}{ 
		\begin{tabular}{c c c c c c c c}
			\toprule
			\multirow{2}{*}{$N$} & \multicolumn{2}{c}{$L^1$ error $\rho$} & \multicolumn{2}{c}{$L^2$ error $\rho$} & \multicolumn{2}{c}{$L^{\infty}$ error $\rho$} & \multirow{2}{*}{CPU Time} \\
			\cmidrule(lr){2-3} \cmidrule(lr){4-5} \cmidrule(lr){6-7}
			& Error & Order & Error & Order & Error & Order & \\
			\midrule
			
			\multicolumn{8}{c}{\textbf{LxF}} \\ 
			\midrule
			160  &   3.251e-02  &  $-$  &   3.515e-02  &  $-$  &   7.146e-02  &  $-$  &   1.418e-01 \\ 
			320  &   6.488e-03  &  2.325  &   9.232e-03  &  1.929  &   2.390e-02  &  1.580  &   5.219e-01 \\ 
			640  &   9.022e-04  &  2.846  &   1.640e-03  &  2.493  &   5.711e-03  &  2.065  &   1.983e+00 \\ 
			1280  &   7.545e-05  &  3.580  &   1.452e-04  &  3.498  &   6.825e-04  &  3.065  &   7.919e+00 \\ 
			2560  &   4.631e-06  &  4.026  &   7.191e-06  &  4.336  &   3.092e-05  &  4.464  &   3.174e+01 \\ 
			5120  &   2.716e-07  &  4.092  &   3.291e-07  &  4.450  &   1.039e-06  &  4.895  &   1.254e+02 \\ 
			\midrule

			\multicolumn{8}{c}{\textbf{FORCE}} \\ 
			\midrule
			160  &   2.263e-02  &  $-$  &   2.582e-02  &  $-$  &   5.550e-02  &  $-$  &   1.466e-01 \\ 
			320  &   4.415e-03  &  2.358  &   6.668e-03  &  1.953  &   1.844e-02  &  1.590  &   5.481e-01 \\ 
			640  &   5.967e-04  &  2.887  &   1.143e-03  &  2.545  &   4.257e-03  &  2.115  &   2.077e+00 \\ 
			1280  &   4.764e-05  &  3.647  &   9.412e-05  &  3.602  &   4.650e-04  &  3.194  &   8.323e+00 \\ 
			2560  &   2.932e-06  &  4.022  &   4.561e-06  &  4.367  &   1.971e-05  &  4.561  &   3.358e+01 \\ 
			5120  &   1.723e-07  &  4.089  &   2.088e-07  &  4.449  &   6.595e-07  &  4.901  &   1.316e+02 \\ 
			\midrule

			\multicolumn{8}{c}{\textbf{Rus}} \\ 
			\midrule
			160  &   2.948e-02  &  $-$  &   3.206e-02  &  $-$  &   6.613e-02  &  $-$  &   1.412e-01 \\ 
			320  &   5.847e-03  &  2.334  &   8.373e-03  &  1.937  &   2.210e-02  &  1.581  &   5.132e-01 \\ 
			640  &   8.023e-04  &  2.865  &   1.469e-03  &  2.511  &   5.221e-03  &  2.082  &   1.970e+00 \\ 
			1280  &   6.646e-05  &  3.594  &   1.273e-04  &  3.528  &   6.078e-04  &  3.103  &   8.121e+00 \\ 
			2560  &   4.104e-06  &  4.017  &   6.272e-06  &  4.344  &   2.696e-05  &  4.495  &   3.208e+01 \\ 
			5120  &   2.422e-07  &  4.083  &   2.883e-07  &  4.443  &   9.049e-07  &  4.897  &   1.272e+02 \\ 
			\midrule

			\multicolumn{8}{c}{\textbf{HLL}} \\ 
			\midrule
			160  &   1.908e-02  &  $-$  &   2.233e-02  &  $-$  &   4.937e-02  &  $-$  &   1.583e-01 \\ 
			320  &   3.706e-03  &  2.364  &   5.755e-03  &  1.957  &   1.641e-02  &  1.590  &   5.629e-01 \\ 
			640  &   4.969e-04  &  2.899  &   9.711e-04  &  2.567  &   3.729e-03  &  2.137  &   2.187e+00 \\ 
			1280  &   3.881e-05  &  3.678  &   7.751e-05  &  3.647  &   3.909e-04  &  3.254  &   8.729e+00 \\ 
			2560  &   2.391e-06  &  4.021  &   3.723e-06  &  4.380  &   1.611e-05  &  4.601  &   3.464e+01 \\ 
			5120  &   1.407e-07  &  4.087  &   1.706e-07  &  4.448  &   5.387e-07  &  4.902  &   1.385e+02 \\ 
			\midrule

			\multicolumn{8}{c}{\textbf{CU}} \\ 
			\midrule
			160  &   1.908e-02  &  $-$  &   2.233e-02  &  $-$  &   4.937e-02  &  $-$  &   1.414e-01 \\ 
			320  &   3.706e-03  &  2.364  &   5.755e-03  &  1.957  &   1.641e-02  &  1.590  &   5.338e-01 \\ 
			640  &   4.969e-04  &  2.899  &   9.711e-04  &  2.567  &   3.729e-03  &  2.137  &   2.033e+00 \\ 
			1280  &   3.881e-05  &  3.678  &   7.751e-05  &  3.647  &   3.909e-04  &  3.254  &   8.068e+00 \\ 
			2560  &   2.391e-06  &  4.021  &   3.723e-06  &  4.380  &   1.611e-05  &  4.601  &   3.182e+01 \\ 
			5120  &   1.407e-07  &  4.087  &   1.706e-07  &  4.448  &   5.387e-07  &  4.902  &   1.284e+02 \\ 
			\midrule

			\multicolumn{8}{c}{\textbf{LDCU}} \\ 
			\midrule
			160  &   1.908e-02  &  $-$  &   2.233e-02  &  $-$  &   4.937e-02  &  $-$  &   1.525e-01 \\ 
			320  &   3.706e-03  &  2.364  &   5.755e-03  &  1.957  &   1.641e-02  &  1.590  &   5.619e-01 \\ 
			640  &   4.969e-04  &  2.899  &   9.711e-04  &  2.567  &   3.729e-03  &  2.137  &   2.158e+00 \\ 
			1280  &   3.881e-05  &  3.678  &   7.751e-05  &  3.647  &   3.909e-04  &  3.254  &   8.737e+00 \\ 
			2560  &   2.391e-06  &  4.021  &   3.723e-06  &  4.380  &   1.611e-05  &  4.601  &   3.401e+01 \\ 
			5120  &   1.407e-07  &  4.087  &   1.706e-07  &  4.448  &   5.387e-07  &  4.902  &   1.389e+02 \\ 
			\midrule

			\multicolumn{8}{c}{\textbf{HLLC}} \\ 
			\midrule
			160  &   1.908e-02  &  $-$  &   2.233e-02  &  $-$  &   4.937e-02  &  $-$  &   1.670e-01 \\ 
			320  &   3.706e-03  &  2.364  &   5.755e-03  &  1.957  &   1.641e-02  &  1.590  &   6.246e-01 \\ 
			640  &   4.969e-04  &  2.899  &   9.711e-04  &  2.567  &   3.729e-03  &  2.137  &   2.455e+00 \\ 
			1280  &   3.881e-05  &  3.678  &   7.751e-05  &  3.647  &   3.909e-04  &  3.254  &   9.511e+00 \\ 
			2560  &   2.391e-06  &  4.021  &   3.723e-06  &  4.380  &   1.611e-05  &  4.601  &   3.866e+01 \\ 
			5120  &   1.407e-07  &  4.087  &   1.706e-07  &  4.448  &   5.387e-07  &  4.902  &   1.530e+02 \\ 
			\midrule

			\multicolumn{8}{c}{\textbf{Ex.RS}} \\ 
			\midrule
			160  &   1.908e-02  &  $-$  &   2.233e-02  &  $-$  &   4.937e-02  &  $-$  &   1.685e-01 \\ 
			320  &   3.706e-03  &  2.364  &   5.755e-03  &  1.957  &   1.641e-02  &  1.590  &   6.262e-01 \\ 
			640  &   4.969e-04  &  2.899  &   9.711e-04  &  2.567  &   3.729e-03  &  2.137  &   2.440e+00 \\ 
			1280  &   3.881e-05  &  3.678  &   7.751e-05  &  3.647  &   3.909e-04  &  3.254  &   9.535e+00 \\ 
			2560  &   2.391e-06  &  4.021  &   3.723e-06  &  4.380  &   1.611e-05  &  4.601  &   3.807e+01 \\ 
			5120  &   1.407e-07  &  4.087  &   1.706e-07  &  4.448  &   5.387e-07  &  4.902  &   1.527e+02 \\ 
			\midrule

			\bottomrule
	\end{tabular}}
\end{table}

\begin{table}[htbp]
	\centering
	\caption{One--dimensional Euler equations, Advection of smooth density: convergence tables for WENO5--DeC5}
	\label{tab:Euler_1d_sin4_convergence_table_WENO5_DeC5}
	\scalebox{0.65}{ 
		\begin{tabular}{c c c c c c c c}
			\toprule
			\multirow{2}{*}{$N$} & \multicolumn{2}{c}{$L^1$ error $\rho$} & \multicolumn{2}{c}{$L^2$ error $\rho$} & \multicolumn{2}{c}{$L^{\infty}$ error $\rho$} & \multirow{2}{*}{CPU Time} \\
			\cmidrule(lr){2-3} \cmidrule(lr){4-5} \cmidrule(lr){6-7}
			& Error & Order & Error & Order & Error & Order & \\
			\midrule
			
			\multicolumn{8}{c}{\textbf{LxF}} \\ 
			\midrule
			80  &   2.501e-03  &  $-$  &   2.065e-03  &  $-$  &   2.753e-03  &  $-$  &   1.129e-01 \\ 
			160  &   1.097e-04  &  4.512  &   1.083e-04  &  4.253  &   2.123e-04  &  3.697  &   3.674e-01 \\ 
			320  &   2.995e-06  &  5.195  &   2.874e-06  &  5.236  &   6.349e-06  &  5.063  &   1.396e+00 \\ 
			640  &   7.039e-08  &  5.411  &   6.082e-08  &  5.562  &   1.106e-07  &  5.843  &   5.475e+00 \\ 
			1280  &   1.610e-09  &  5.450  &   1.319e-09  &  5.527  &   1.786e-09  &  5.953  &   2.193e+01 \\ 
			2560  &   3.816e-11  &  5.399  &   3.341e-11  &  5.304  &   4.982e-11  &  5.163  &   8.775e+01 \\ 
			\midrule

			\multicolumn{8}{c}{\textbf{FORCE}} \\ 
			\midrule
			80  &   1.830e-03  &  $-$  &   1.548e-03  &  $-$  &   2.517e-03  &  $-$  &   1.175e-01 \\ 
			160  &   7.063e-05  &  4.695  &   7.010e-05  &  4.464  &   1.442e-04  &  4.126  &   3.823e-01 \\ 
			320  &   1.894e-06  &  5.221  &   1.833e-06  &  5.258  &   4.133e-06  &  5.125  &   1.439e+00 \\ 
			640  &   4.461e-08  &  5.408  &   3.856e-08  &  5.571  &   7.011e-08  &  5.881  &   5.709e+00 \\ 
			1280  &   1.020e-09  &  5.451  &   8.360e-10  &  5.527  &   1.131e-09  &  5.953  &   2.261e+01 \\ 
			2560  &   2.418e-11  &  5.399  &   2.117e-11  &  5.304  &   3.161e-11  &  5.162  &   9.041e+01 \\ 
			\midrule

			\multicolumn{8}{c}{\textbf{Rus}} \\ 
			\midrule
			80  &   2.347e-03  &  $-$  &   1.954e-03  &  $-$  &   2.660e-03  &  $-$  &   1.166e-01 \\ 
			160  &   1.013e-04  &  4.534  &   1.015e-04  &  4.267  &   2.013e-04  &  3.724  &   3.734e-01 \\ 
			320  &   2.769e-06  &  5.193  &   2.697e-06  &  5.234  &   6.033e-06  &  5.061  &   1.414e+00 \\ 
			640  &   6.457e-08  &  5.422  &   5.632e-08  &  5.582  &   1.048e-07  &  5.847  &   5.508e+00 \\ 
			1280  &   1.463e-09  &  5.464  &   1.194e-09  &  5.559  &   1.556e-09  &  6.073  &   2.225e+01 \\ 
			2560  &   3.433e-11  &  5.413  &   2.979e-11  &  5.325  &   4.358e-11  &  5.158  &   8.896e+01 \\ 
			\midrule

			\multicolumn{8}{c}{\textbf{HLL}} \\ 
			\midrule
			80  &   1.586e-03  &  $-$  &   1.371e-03  &  $-$  &   2.391e-03  &  $-$  &   1.234e-01 \\ 
			160  &   5.795e-05  &  4.775  &   5.782e-05  &  4.567  &   1.214e-04  &  4.299  &   4.002e-01 \\ 
			320  &   1.545e-06  &  5.230  &   1.500e-06  &  5.268  &   3.435e-06  &  5.144  &   1.519e+00 \\ 
			640  &   3.636e-08  &  5.409  &   3.144e-08  &  5.576  &   5.727e-08  &  5.906  &   5.976e+00 \\ 
			1280  &   8.323e-10  &  5.449  &   6.821e-10  &  5.527  &   9.233e-10  &  5.955  &   2.371e+01 \\ 
			2560  &   1.972e-11  &  5.399  &   1.727e-11  &  5.304  &   2.579e-11  &  5.162  &   9.545e+01 \\ 
			\midrule

			\multicolumn{8}{c}{\textbf{CU}} \\ 
			\midrule
			80  &   1.586e-03  &  $-$  &   1.371e-03  &  $-$  &   2.391e-03  &  $-$  &   1.156e-01 \\ 
			160  &   5.795e-05  &  4.775  &   5.782e-05  &  4.567  &   1.214e-04  &  4.299  &   3.763e-01 \\ 
			320  &   1.545e-06  &  5.230  &   1.500e-06  &  5.268  &   3.435e-06  &  5.144  &   1.410e+00 \\ 
			640  &   3.636e-08  &  5.409  &   3.144e-08  &  5.576  &   5.727e-08  &  5.906  &   5.708e+00 \\ 
			1280  &   8.323e-10  &  5.449  &   6.821e-10  &  5.527  &   9.233e-10  &  5.955  &   2.219e+01 \\ 
			2560  &   1.972e-11  &  5.399  &   1.727e-11  &  5.304  &   2.577e-11  &  5.163  &   8.861e+01 \\ 
			\midrule

			\multicolumn{8}{c}{\textbf{LDCU}} \\ 
			\midrule
			80  &   1.586e-03  &  $-$  &   1.371e-03  &  $-$  &   2.391e-03  &  $-$  &   1.217e-01 \\ 
			160  &   5.795e-05  &  4.775  &   5.782e-05  &  4.567  &   1.214e-04  &  4.299  &   3.945e-01 \\ 
			320  &   1.545e-06  &  5.230  &   1.500e-06  &  5.268  &   3.435e-06  &  5.144  &   1.606e+00 \\ 
			640  &   3.636e-08  &  5.409  &   3.144e-08  &  5.576  &   5.727e-08  &  5.906  &   5.884e+00 \\ 
			1280  &   8.323e-10  &  5.449  &   6.821e-10  &  5.527  &   9.233e-10  &  5.955  &   2.333e+01 \\ 
			2560  &   1.972e-11  &  5.399  &   1.727e-11  &  5.304  &   2.577e-11  &  5.163  &   9.270e+01 \\ 
			\midrule

			\multicolumn{8}{c}{\textbf{HLLC}} \\ 
			\midrule
			80  &   1.586e-03  &  $-$  &   1.371e-03  &  $-$  &   2.391e-03  &  $-$  &   1.317e-01 \\ 
			160  &   5.795e-05  &  4.775  &   5.782e-05  &  4.567  &   1.214e-04  &  4.299  &   4.361e-01 \\ 
			320  &   1.545e-06  &  5.230  &   1.500e-06  &  5.268  &   3.435e-06  &  5.144  &   1.667e+00 \\ 
			640  &   3.636e-08  &  5.409  &   3.144e-08  &  5.576  &   5.727e-08  &  5.906  &   6.532e+00 \\ 
			1280  &   8.323e-10  &  5.449  &   6.821e-10  &  5.527  &   9.233e-10  &  5.955  &   2.609e+01 \\ 
			2560  &   1.972e-11  &  5.399  &   1.727e-11  &  5.304  &   2.576e-11  &  5.164  &   1.012e+02 \\ 
			\midrule

			\multicolumn{8}{c}{\textbf{Ex.RS}} \\ 
			\midrule
			80  &   1.586e-03  &  $-$  &   1.371e-03  &  $-$  &   2.391e-03  &  $-$  &   1.313e-01 \\ 
			160  &   5.795e-05  &  4.775  &   5.782e-05  &  4.567  &   1.214e-04  &  4.299  &   4.377e-01 \\ 
			320  &   1.545e-06  &  5.230  &   1.500e-06  &  5.268  &   3.435e-06  &  5.144  &   1.648e+00 \\ 
			640  &   3.636e-08  &  5.409  &   3.144e-08  &  5.576  &   5.727e-08  &  5.906  &   6.490e+00 \\ 
			1280  &   8.323e-10  &  5.449  &   6.821e-10  &  5.527  &   9.233e-10  &  5.955  &   2.654e+01 \\ 
			2560  &   1.972e-11  &  5.399  &   1.727e-11  &  5.304  &   2.579e-11  &  5.162  &   1.019e+02 \\ 
			\midrule

			\bottomrule
	\end{tabular}}
\end{table}

\begin{table}[htbp]
	\centering
	\caption{One--dimensional Euler equations, Advection of smooth density: convergence tables for WENO7--DeC7}
	\label{tab:Euler_1d_sin4_convergence_table_WENO7_DeC7}
	\scalebox{0.65}{ 
		\begin{tabular}{c c c c c c c c}
			\toprule
			\multirow{2}{*}{$N$} & \multicolumn{2}{c}{$L^1$ error $\rho$} & \multicolumn{2}{c}{$L^2$ error $\rho$} & \multicolumn{2}{c}{$L^{\infty}$ error $\rho$} & \multirow{2}{*}{CPU Time} \\
			\cmidrule(lr){2-3} \cmidrule(lr){4-5} \cmidrule(lr){6-7}
			& Error & Order & Error & Order & Error & Order & \\
			\midrule
			
			\multicolumn{8}{c}{\textbf{LxF}} \\ 
			\midrule
			80  &   2.963e-04  &  $-$  &   3.204e-04  &  $-$  &   5.555e-04  &  $-$  &   3.807e-01 \\ 
			160  &   1.133e-06  &  8.031  &   1.304e-06  &  7.941  &   2.845e-06  &  7.609  &   1.281e+00 \\ 
			320  &   8.000e-09  &  7.146  &   1.437e-08  &  6.503  &   7.014e-08  &  5.342  &   4.862e+00 \\ 
			640  &   5.527e-11  &  7.177  &   1.353e-10  &  6.731  &   8.850e-10  &  6.308  &   1.953e+01 \\ 
			1280  &   2.484e-13  &  7.798  &   4.772e-13  &  8.147  &   3.010e-12  &  8.200  &   7.648e+01 \\ 
			\midrule

			\multicolumn{8}{c}{\textbf{FORCE}} \\ 
			\midrule
			80  &   2.048e-04  &  $-$  &   2.230e-04  &  $-$  &   3.944e-04  &  $-$  &   3.906e-01 \\ 
			160  &   7.435e-07  &  8.106  &   8.809e-07  &  7.984  &   2.014e-06  &  7.613  &   1.295e+00 \\ 
			320  &   5.247e-09  &  7.147  &   9.661e-09  &  6.511  &   4.914e-08  &  5.357  &   4.864e+00 \\ 
			640  &   3.537e-11  &  7.213  &   8.797e-11  &  6.779  &   5.938e-10  &  6.371  &   1.966e+01 \\ 
			1280  &   1.601e-13  &  7.787  &   3.014e-13  &  8.189  &   1.908e-12  &  8.282  &   7.783e+01 \\ 
			\midrule

			\multicolumn{8}{c}{\textbf{Rus}} \\ 
			\midrule
			80  &   2.800e-04  &  $-$  &   3.045e-04  &  $-$  &   5.314e-04  &  $-$  &   4.021e-01 \\ 
			160  &   1.042e-06  &  8.070  &   1.205e-06  &  7.981  &   2.542e-06  &  7.707  &   1.274e+00 \\ 
			320  &   7.211e-09  &  7.175  &   1.275e-08  &  6.563  &   6.275e-08  &  5.340  &   5.044e+00 \\ 
			640  &   4.890e-11  &  7.204  &   1.184e-10  &  6.751  &   7.818e-10  &  6.327  &   1.931e+01 \\ 
			1280  &   2.171e-13  &  7.816  &   4.164e-13  &  8.152  &   2.488e-12  &  8.296  &   7.844e+01 \\ 
			\midrule

			\multicolumn{8}{c}{\textbf{HLL}} \\ 
			\midrule
			80  &   1.751e-04  &  $-$  &   1.919e-04  &  $-$  &   3.411e-04  &  $-$  &   3.989e-01 \\ 
			160  &   6.198e-07  &  8.142  &   7.444e-07  &  8.010  &   1.736e-06  &  7.618  &   1.336e+00 \\ 
			320  &   4.359e-09  &  7.152  &   8.156e-09  &  6.512  &   4.220e-08  &  5.362  &   5.239e+00 \\ 
			640  &   2.897e-11  &  7.233  &   7.285e-11  &  6.807  &   4.985e-10  &  6.404  &   2.056e+01 \\ 
			1280  &   1.333e-13  &  7.763  &   2.501e-13  &  8.186  &   1.541e-12  &  8.338  &   8.056e+01 \\ 
			\midrule

			\multicolumn{8}{c}{\textbf{CU}} \\ 
			\midrule
			80  &   1.751e-04  &  $-$  &   1.919e-04  &  $-$  &   3.411e-04  &  $-$  &   3.854e-01 \\ 
			160  &   6.198e-07  &  8.142  &   7.444e-07  &  8.010  &   1.736e-06  &  7.618  &   1.287e+00 \\ 
			320  &   4.359e-09  &  7.152  &   8.156e-09  &  6.512  &   4.220e-08  &  5.362  &   4.927e+00 \\ 
			640  &   2.897e-11  &  7.233  &   7.285e-11  &  6.807  &   4.985e-10  &  6.404  &   1.994e+01 \\ 
			1280  &   1.679e-13  &  7.431  &   2.637e-13  &  8.110  &   1.551e-12  &  8.328  &   7.838e+01 \\ 
			\midrule

			\multicolumn{8}{c}{\textbf{LDCU}} \\ 
			\midrule
			80  &   1.751e-04  &  $-$  &   1.919e-04  &  $-$  &   3.411e-04  &  $-$  &   3.952e-01 \\ 
			160  &   6.198e-07  &  8.142  &   7.444e-07  &  8.010  &   1.736e-06  &  7.618  &   1.314e+00 \\ 
			320  &   4.359e-09  &  7.152  &   8.156e-09  &  6.512  &   4.220e-08  &  5.362  &   5.069e+00 \\ 
			640  &   2.897e-11  &  7.233  &   7.285e-11  &  6.807  &   4.985e-10  &  6.404  &   2.011e+01 \\ 
			1280  &   1.679e-13  &  7.431  &   2.637e-13  &  8.110  &   1.551e-12  &  8.328  &   8.010e+01 \\ 
			\midrule

			\multicolumn{8}{c}{\textbf{HLLC}} \\ 
			\midrule
			80  &   1.751e-04  &  $-$  &   1.919e-04  &  $-$  &   3.411e-04  &  $-$  &   4.183e-01 \\ 
			160  &   6.198e-07  &  8.142  &   7.444e-07  &  8.010  &   1.736e-06  &  7.618  &   1.376e+00 \\ 
			320  &   4.359e-09  &  7.152  &   8.156e-09  &  6.512  &   4.220e-08  &  5.362  &   5.494e+00 \\ 
			640  &   2.897e-11  &  7.233  &   7.285e-11  &  6.807  &   4.985e-10  &  6.404  &   2.116e+01 \\ 
			1280  &   1.354e-13  &  7.741  &   2.485e-13  &  8.195  &   1.512e-12  &  8.365  &   8.462e+01 \\ 
			\midrule

			\multicolumn{8}{c}{\textbf{Ex.RS}} \\ 
			\midrule
			80  &   1.751e-04  &  $-$  &   1.919e-04  &  $-$  &   3.411e-04  &  $-$  &   4.175e-01 \\ 
			160  &   6.198e-07  &  8.142  &   7.444e-07  &  8.010  &   1.736e-06  &  7.618  &   1.387e+00 \\ 
			320  &   4.359e-09  &  7.152  &   8.156e-09  &  6.512  &   4.220e-08  &  5.362  &   5.419e+00 \\ 
			640  &   2.897e-11  &  7.233  &   7.285e-11  &  6.807  &   4.985e-10  &  6.404  &   2.115e+01 \\ 
			1280  &   1.333e-13  &  7.763  &   2.501e-13  &  8.186  &   1.541e-12  &  8.338  &   8.507e+01 \\ 
			\midrule

			\bottomrule
	\end{tabular}}
\end{table}

The results of the convergence analysis and of the efficiency analysis in the $L^1$-norm, in terms of error versus time, are reported for all orders and numerical fluxes in Figure~\ref{fig:Euler_1d_sin4_WENODeC}, on the left and on the right side respectively.
We do not display graphical results for the other norms in order to save space, since in fact they are qualitatively analogous, as can be noticed from the convergence tables.
Only slight differences can be appreciated for different numerical fluxes on this test, which tend to diminish as the order of accuracy increases. Generally speaking,  upwind fluxes are more accurate than centred fluxes (LxF and FORCE), except for Rus which is less accurate than FORCE for this problem.
In fact, from the convergence analysis plot on the left, we see, as expected, LxF and Rus being more diffusive than the other numerical fluxes and producing higher errors for fixed mesh refinement; FORCE performs better than them but slightly worse than all remaining numerical fluxes.
From the efficiency analysis plot on the right, we see that LxF is the worst option.
It must be noticed that, for fixed order, lines corresponding to different numerical fluxes are hard to distinguish, especially for order 7.
This aspect will be put in evidence also in the next tests: the impact of the numerical flux is higher in lower order settings, as increasing the order tends to compensate the choice of a more diffusive numerical flux and to minimize the differences with respect to this choice.
Let us remark that, for this test, increasing the order of accuracy of the discretization results in huge computational advantages for all numerical fluxes, as can be observed from the efficiency analysis plot on the right in Figure~\ref{fig:Euler_1d_sin4_WENODeC}. Investigations of higher order discretizations within the WENO--DeC framework are planned for future works.

\begin{figure}[htbp]
	\centering
	\begin{subfigure}[t]{0.53\textwidth}
		\centering
		\includegraphics[width=\textwidth]{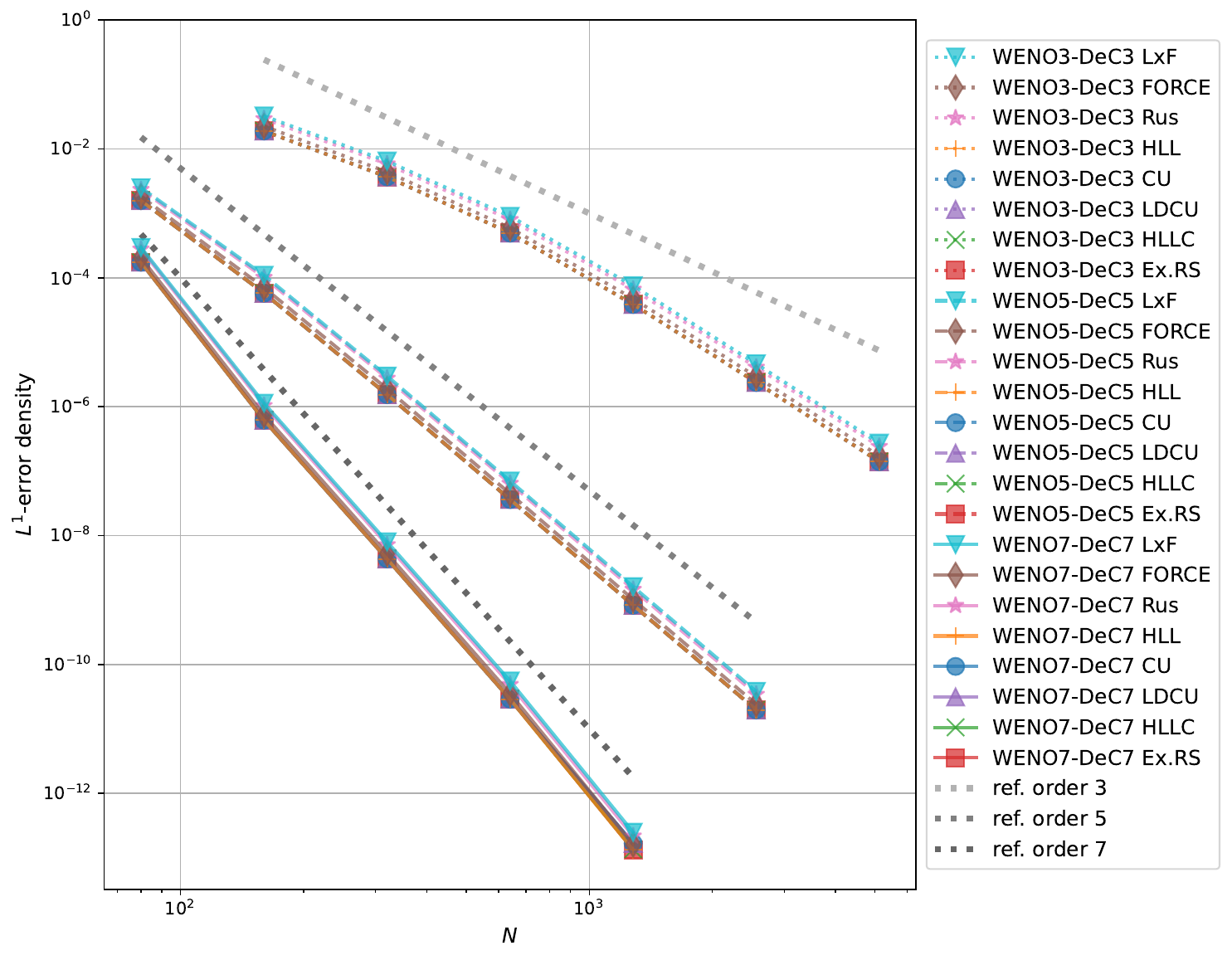}
		\caption{Convergence analysis on the density}
	\end{subfigure}
	\quad
	\begin{subfigure}[t]{0.43\textwidth}
		\centering
		\includegraphics[width=\textwidth]{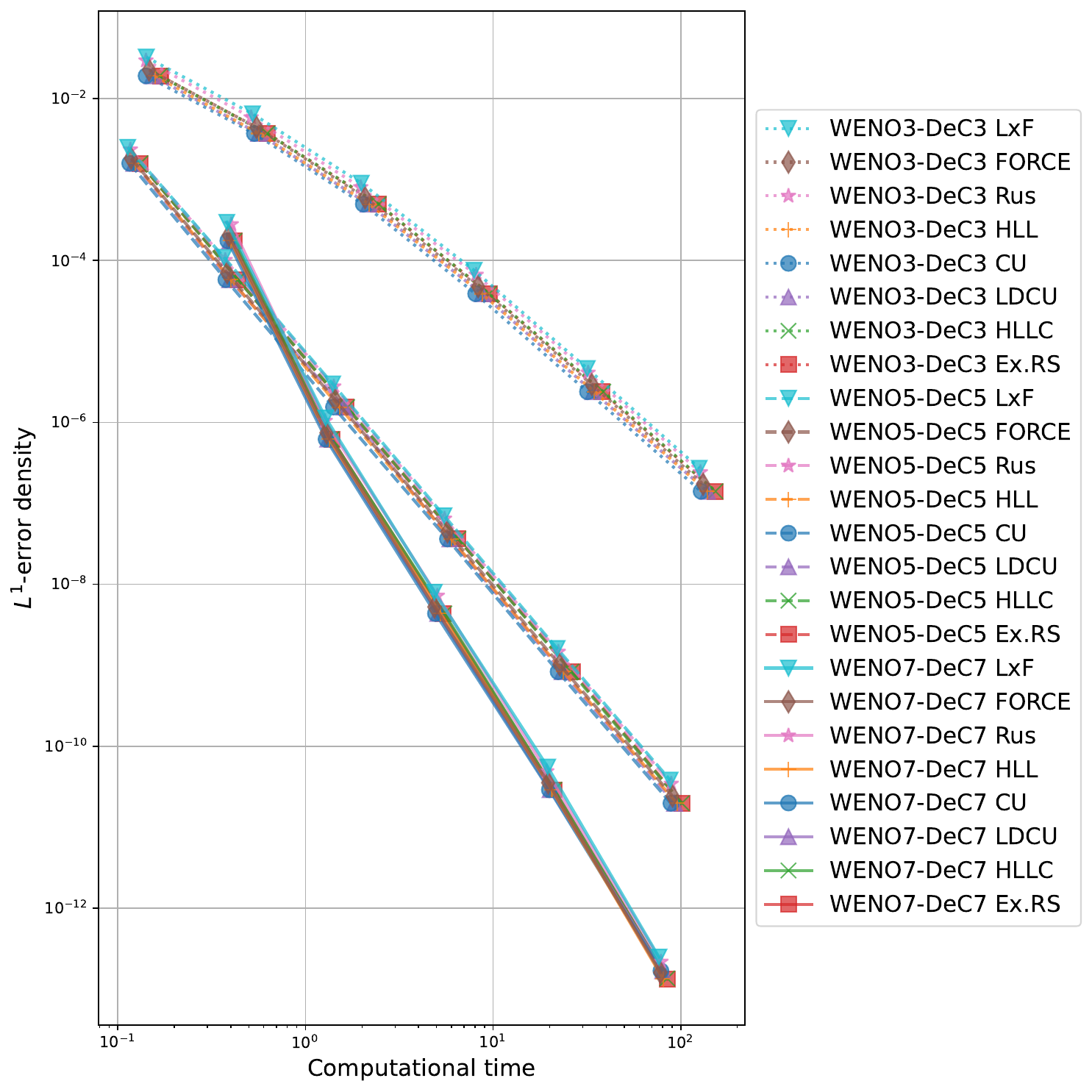}
		\caption{Error on the density versus time}
	\end{subfigure}
	\caption{One--dimensional Euler equations, Advection of smooth density: Convergence analysis and efficiency analysis for all numerical fluxes and orders}
	\label{fig:Euler_1d_sin4_WENODeC}
\end{figure}

\subsubsection{Riemann problem 1}\label{sec:Euler_1d_RP1}
This test consists in a challenging modification, introduced in~\cite{ToroBook}, of the well--known Sod shock tube~\cite{sod1978survey}.
In fact, while in the original test the initial velocity is uniformly equal to zero, in such a modification the left state moves towards right with speed $0.75$.
This determines the existence of a sonic point in the solution, whose correct handling constitutes a challenge for many numerical schemes~\cite{ToroBook}.

We could run the test without any problem with $C_{CFL}:=0.95$ for all numerical fluxes and orders for $100$ elements.
The obtained density profiles are reported in Figure~\ref{fig:Euler_1d_RP1_zoom_density}.
Before starting the comparison, let us notice that all the discretizations under investigation provide a correct handling of the sonic point, which is not trivial, as already stated.

The most evident differences among the different numerical fluxes can be observed in the description of the rarefaction wave, in panels A and B.
Ex.RS provides the sharpest description of such a feature followed by LDCU. LxF, FORCE and Rus have the worst performance. HLL, CU and HLLC give similar results of intermediate quality between LDCU and FORCE.
Indeed, also in this case, increasing the order of accuracy improves the quality of the results, however, it must be noticed how this has more impact on the more diffusive numerical fluxes.
The difference between the numerical fluxes in panels A and B tends to become smaller as the order increases, still being evident even for order 7.
In particular, the improvements for LxF, FORCE and Rus are bigger than the ones registered for Ex.RS, whose results are already very good for order 3.

Qualitatively similar results have been obtained for the Sod shock tube test, however, they have been omitted, due to the more challenging nature of this test, in order to save space.

\begin{figure}[htbp]
	\centering
	\begin{subfigure}[b]{1.\textwidth}
		\centering
		\includegraphics[width=\textwidth]{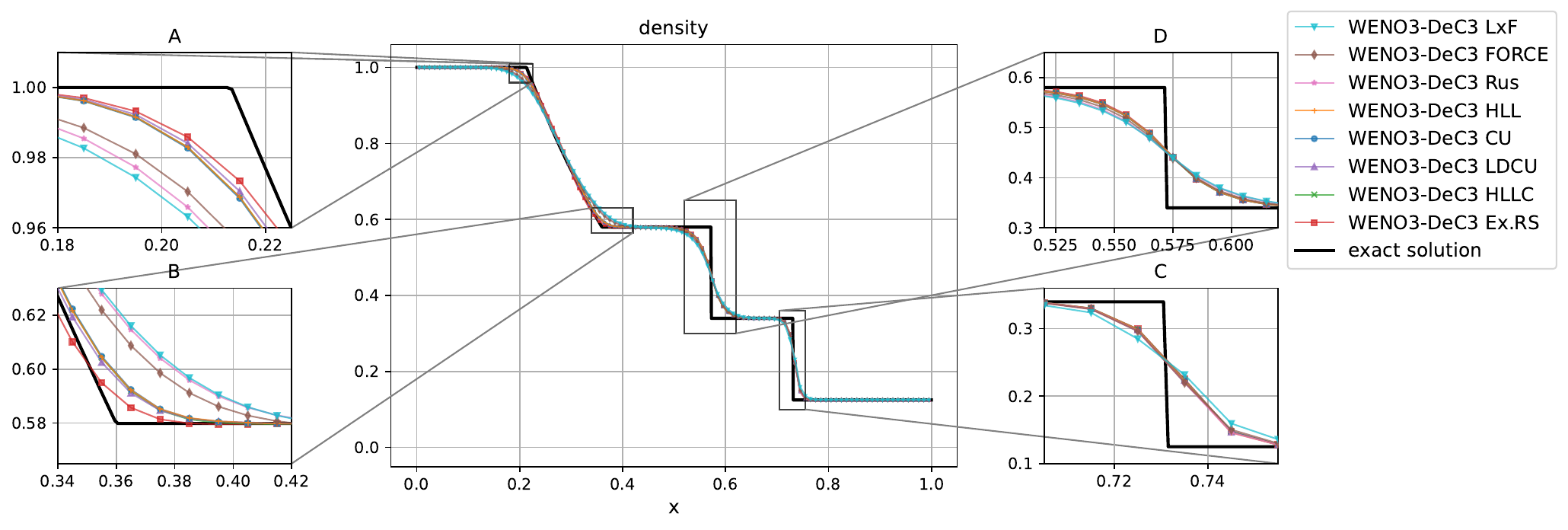}
		\caption{Order 3}
	\end{subfigure}\\
	\begin{subfigure}[b]{1.\textwidth}
		\centering
		\includegraphics[width=\textwidth]{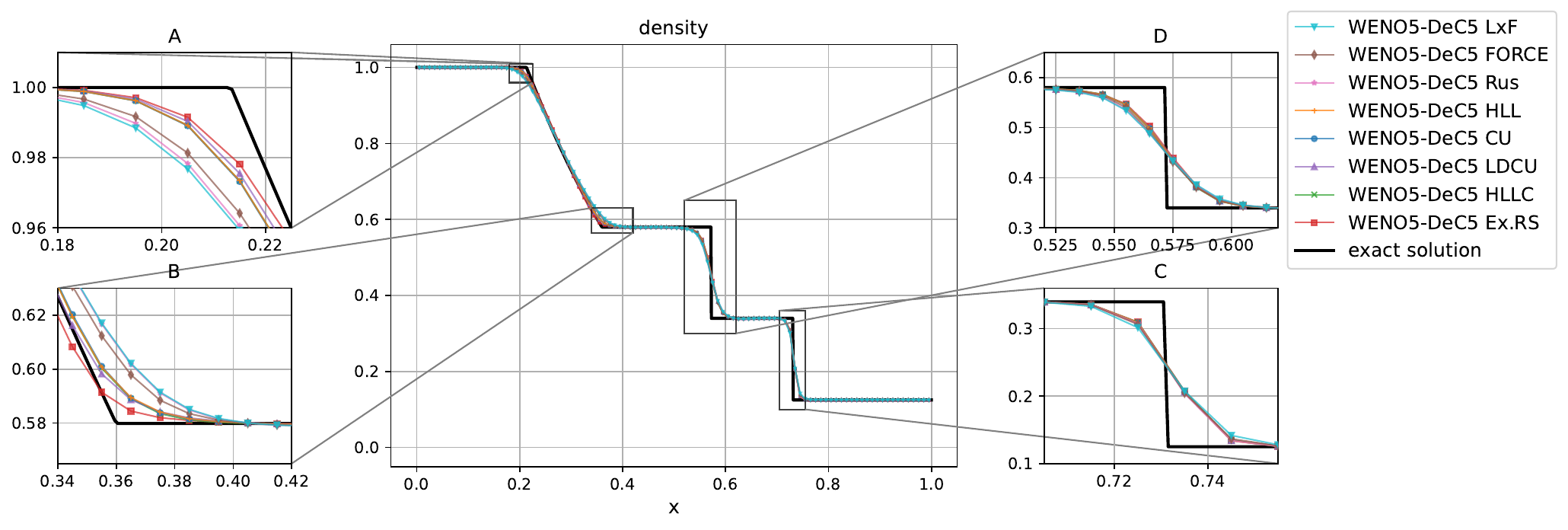}
		\caption{Order 5}
	\end{subfigure}
	\\
	\begin{subfigure}[b]{1.\textwidth}
		\centering
		\includegraphics[width=\textwidth]{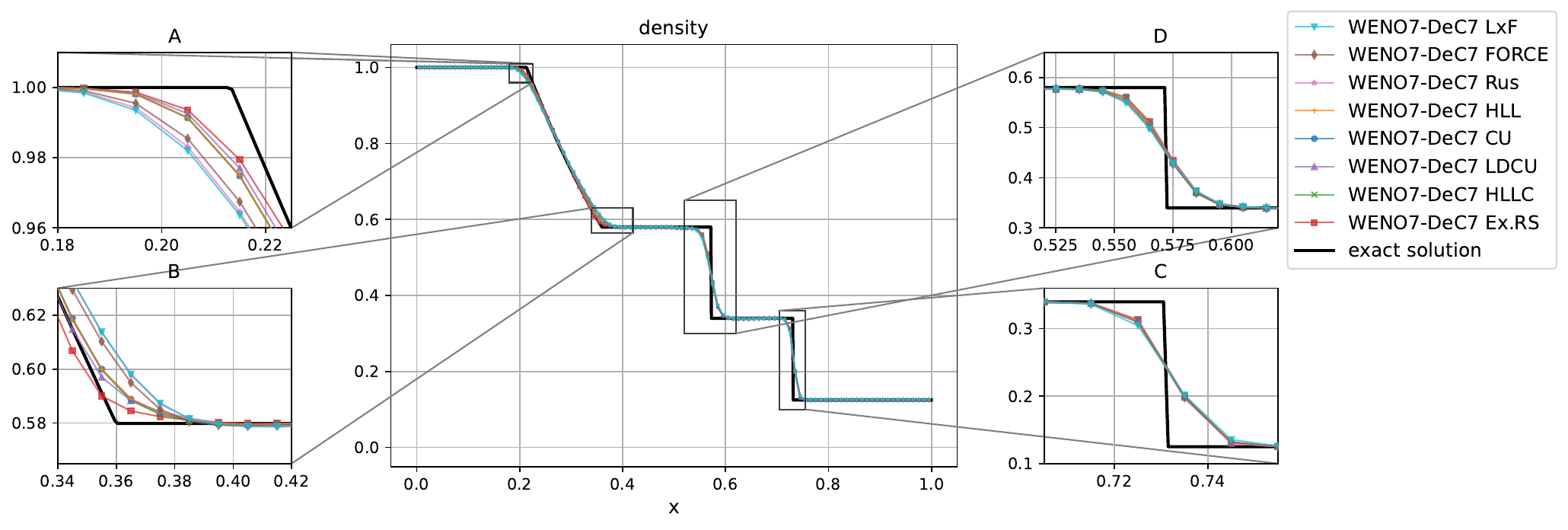}
		\caption{Order 7}
	\end{subfigure}
	\caption{One--dimensional Euler equations, Riemann problem 1: Density profile obtained with all numerical fluxes for different orders with $C_{CFL}:=0.95$}
	\label{fig:Euler_1d_RP1_zoom_density}
\end{figure}

\subsubsection{Riemann problem 5}\label{sec:Euler_1d_RP5}
%

This Riemann problem results from choosing the left and middle states of the well--known Woodward--Colella problem~\cite{woodward1984numerical}, which involves three states and does not have an exact solution. 
Instead, the resulting Riemann problem has exact solution, and it constitutes a very severe test, with initial pressure jump of $10^{5}$ in the order of magnitude and an emerging shock wave of Mach number 198.
Also in this case, we have considered $100$ cells for the numerical experiments.
%
%
The behavior of the methods on this test was quite irregular.
\begin{itemize}
	\item Concerning LxF, 
	with order 3 there are no simulation crashes for $C_{CFL}\leq 0.35$; 
	with order 5 for $C_{CFL}\leq 0.3$;
	with order 7 for $C_{CFL}\leq 0.25$.

	\item Concerning FORCE, 
	no simulation crashes occur for $C_{CFL}=0.95$ with any order.
	
	\item Concerning Rus, 
	with order 3 there are no simulation crashes for $C_{CFL}\leq 0.95$; 
	with order 5 for $C_{CFL}\leq 0.85$;
	with order 7 for $C_{CFL}\leq 0.75$.
	
	\item Concerning HLL, 
	with order 3 there are no simulation crashes for $C_{CFL}\leq 0.15$; 
	with order 5 for $C_{CFL}\leq 0.2$;
	with order 7 for $C_{CFL}\leq 0.25$.
	
	\item Concerning CU, 
	with order 3 there are no simulation crashes for $C_{CFL}\leq 0.4$; 
	with order 5 for $C_{CFL}\leq 0.6$;
	with order 7 for $C_{CFL}\leq 0.55$.

	\item Concerning LDCU, 
	with order 3 there are no simulation crashes for $C_{CFL}\leq 0.45$; 
	with order 5 for $C_{CFL}\leq 0.5$;
	with order 7 for $C_{CFL}\leq 0.8$.

	\item Concerning HLLC, 
	with order 3 there are no simulation crashes for $C_{CFL}\leq 0.1$; 
	with order 5 and 5 for $C_{CFL}\leq 0.15$.

	\item Concerning Ex.RS, 
	with order 3 there are no simulation crashes for $C_{CFL}\leq 0.15$; 
	with order 5 for $C_{CFL}\leq 0.2$;
	with order 7 for $C_{CFL}\leq 0.3$.

\end{itemize}
It is interesting to notice how, in this test, adopting higher order discretizations allows to use higher values of $C_{CFL}$ for HLL, CU, LDCU, HLLC, and Ex.RS.

Let us comment the density results obtained for $C_{CFL}:=0.1$, for which all settings are available.
They are reported in Figure~\ref{fig:Euler_1d_RP5_zoom_density}.

In this test, we can appreciate huge differences depending on the choice of the numerical flux.
The main ones concern the description of the peak of the density in panel D.
The complete upwind numerical fluxes, HLLC and Ex.RS, provide a nice and detailed capturing of such a solution feature even for order 3.
The other numerical fluxes perform much worse.
Interesting differences can be appreciated also in panel E with similar observations: the complete upwind numerical fluxes outperform the other fluxes by far.
In general, the centred fluxes (LxF and FORCE) achieve the worst performance; the complete upwind ones (HLLC and Ex.RS) have the best one; the other numerical fluxes, i.e., the incomplete upwind numerical fluxes, give similar results amongst themselves.

Let us notice that increasing the order of the discretization determines strong improvements for all numerical fluxes but for HLLC and Ex.RS, for which the quality of the results is very high already for order 3.
Differences among all numerical fluxes become smaller but still rather evident passing from order 3 to order 7.
As a matter of fact, the choice of the ``correct'' numerical flux, and in particular the adoption of a complete upwind flux, for this test, is equivalent to the adoption of a higher order discretization. 
Increasing the order of accuracy, within the explored range, helps but it is not sufficient to nullify the differences between the different numerical fluxes.
As a further evidence for this fact, we report in Figure~\ref{fig:Euler_1d_RP5_zoom_density_o1} the results obtained with all numerical fluxes for a standard first order FV discretization involving a piecewise-constant (PWC) reconstruction in space and Euler in time. We consider $100$ cells and $C_{CFL}:=0.95$.
As one can clearly observe, for this test, the quality of the results obtained through the complete upwind fluxes, HLLC and Ex.RS, is comparable with the one of the results obtained for much higher order reported in Figure~\ref{fig:Euler_1d_RP5_zoom_density}.
The results obtained for LxF are in line with what reported in~\cite{toro2024computational}.
Let us notice that the first order setting is not subjected to the tedious time step restrictions, which characterize higher order discretizations, and all simulations could be run with $C_{CFL}:=0.95$ without crashes.

\begin{figure}[htbp]
	\centering
	\begin{subfigure}[b]{0.6\textwidth}
		\centering
		\includegraphics[width=\textwidth]{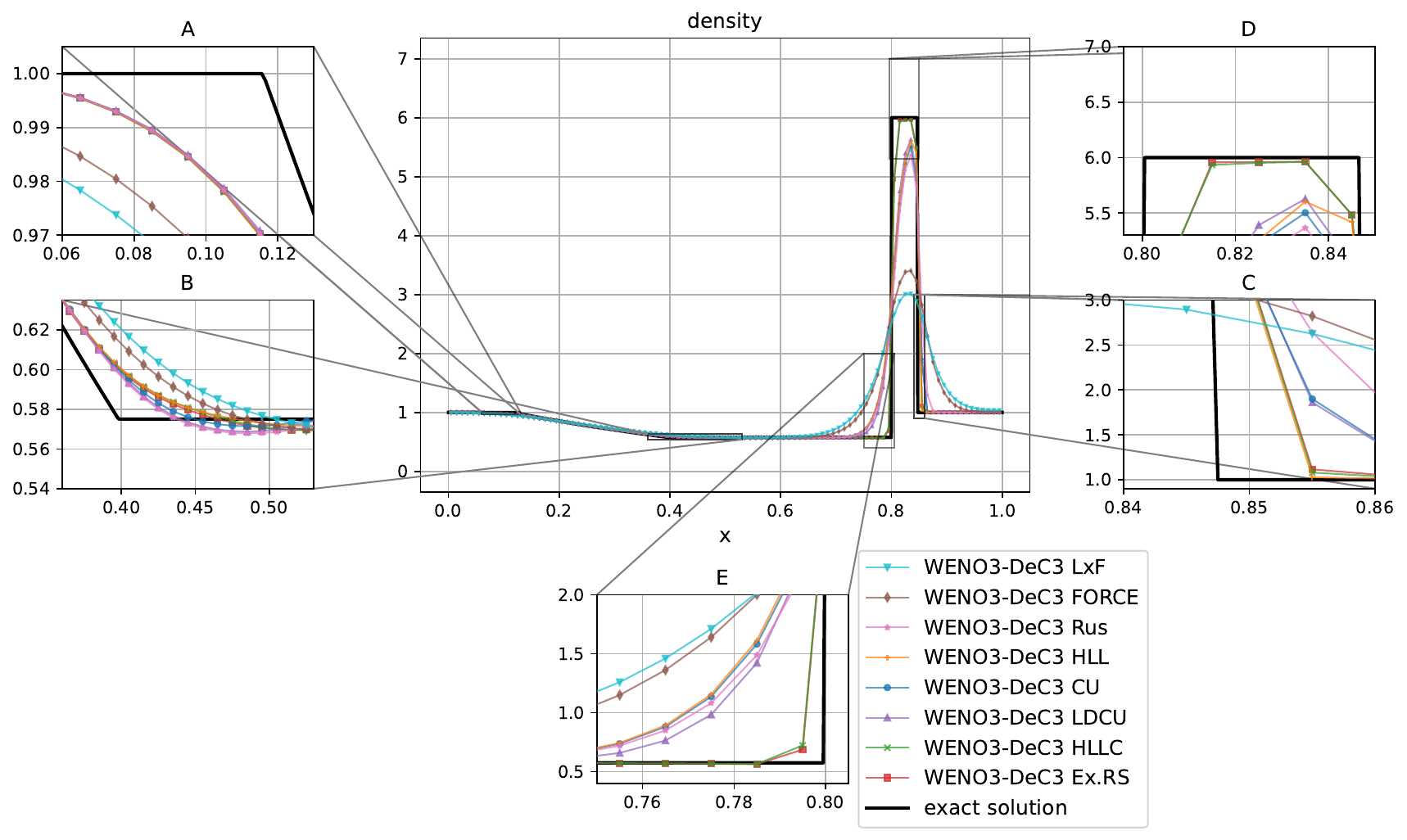}
		\caption{Order 3}
	\end{subfigure}\\
	\begin{subfigure}[b]{0.6\textwidth}
		\centering
		\includegraphics[width=\textwidth]{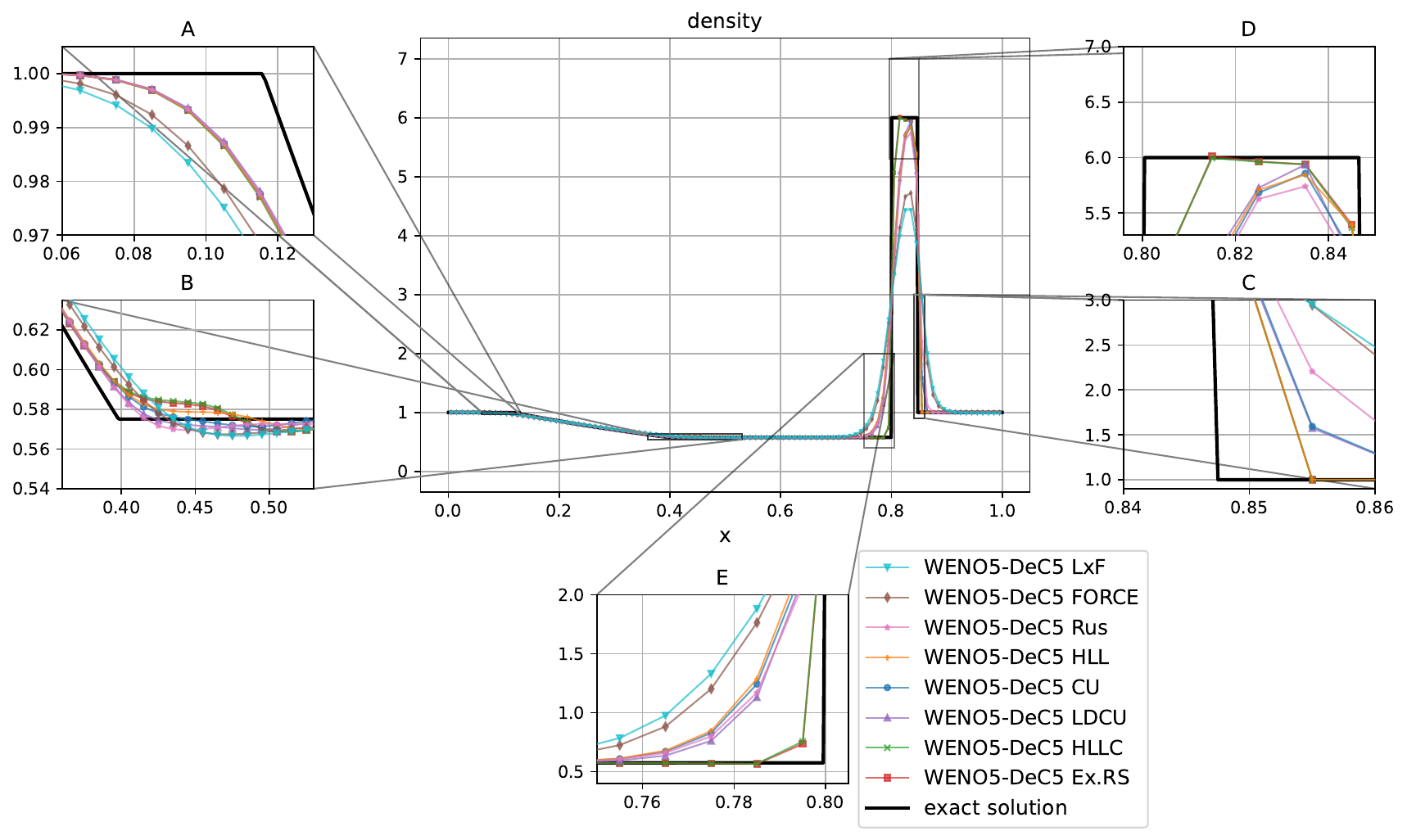}
		\caption{Order 5}
	\end{subfigure}
	\\
	\begin{subfigure}[b]{0.6\textwidth}
		\centering
		\includegraphics[width=\textwidth]{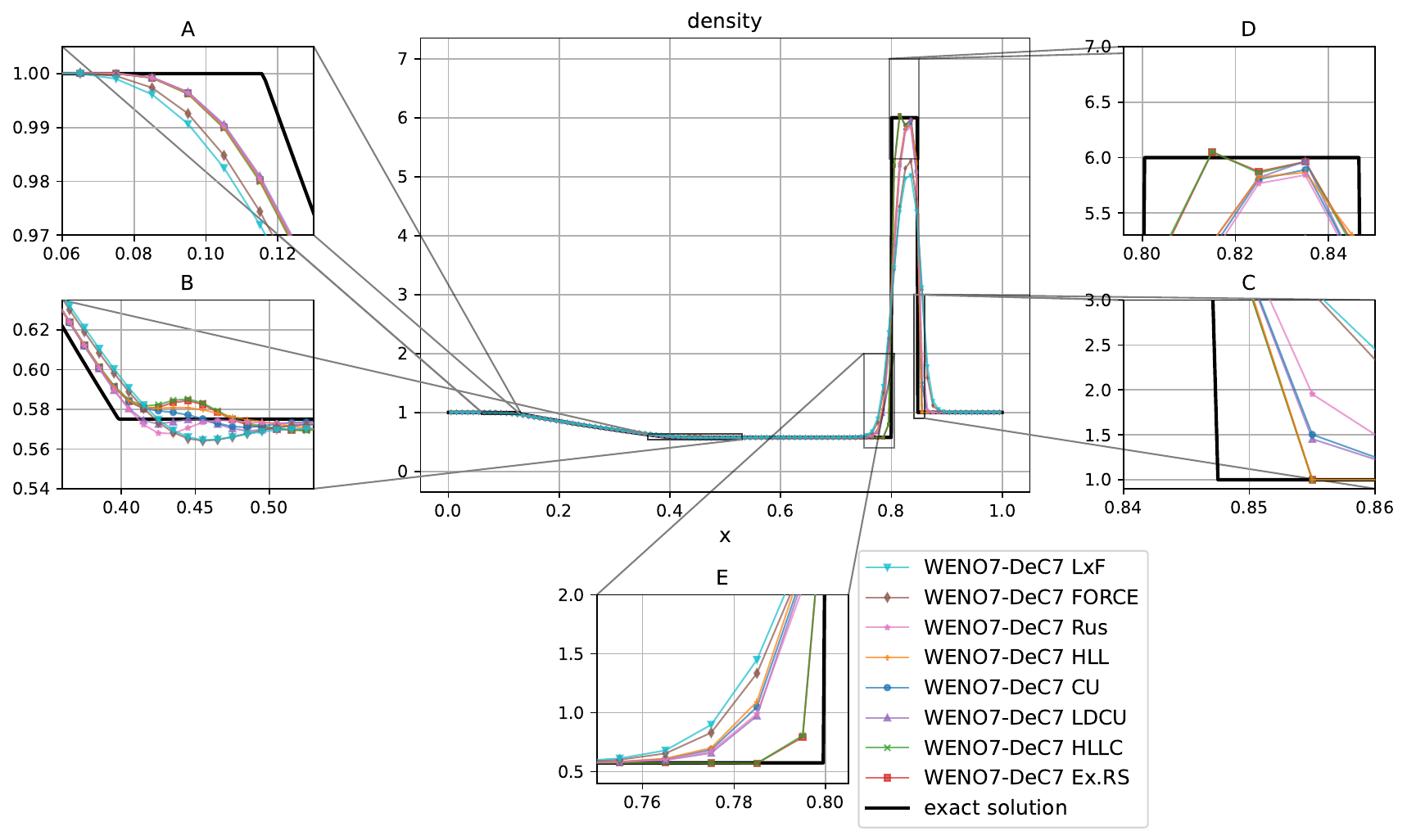}
		\caption{Order 7}
	\end{subfigure}
	\caption{One--dimensional Euler equations, Riemann problem 5: Density profile obtained with all numerical fluxes for different orders with $C_{CFL}:=0.1$}
	\label{fig:Euler_1d_RP5_zoom_density}
\end{figure}

\begin{figure}[ht]
	\centering
	\includegraphics[width=\textwidth]{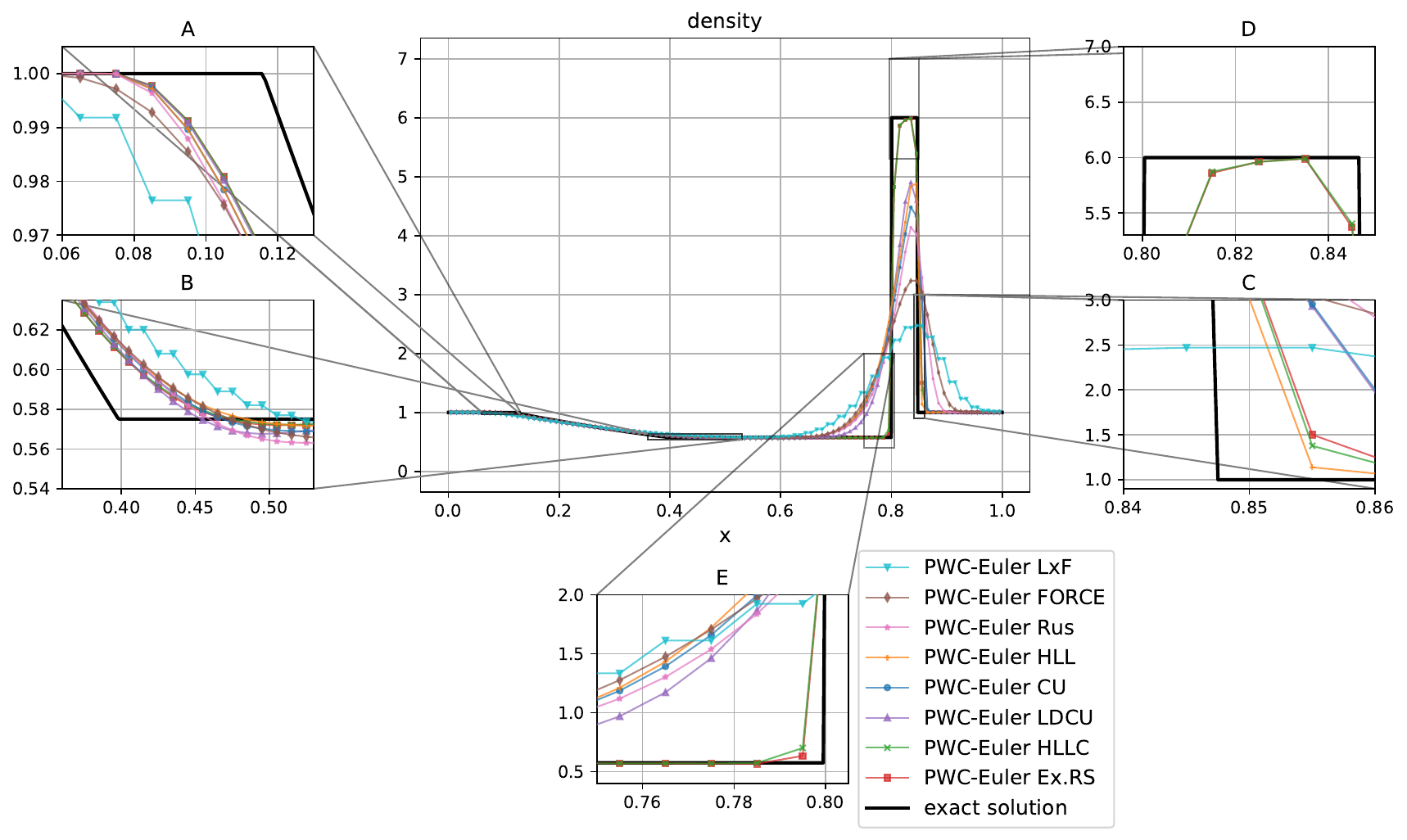} 
	\caption{One--dimensional Euler equations, Riemann problem 5: Density profile obtained with all numerical fluxes for a first order FV setting with $C_{CFL}:=0.95$}
	\label{fig:Euler_1d_RP5_zoom_density_o1}
\end{figure}

\subsubsection{Riemann problem 6}\label{sec:Euler_1d_RP6}
This test consists in a stationary contact discontinuity.
This is a typical problem in which complete upwind fluxes are expected to outperform all the other numerical fluxes.
In fact, the main motivation for the introduction of HLLC lied precisely in the bad performance of the HLL numerical flux (and of other incomplete upwind or centred numerical fluxes) in capturing contact surfaces,
shear waves and material interfaces.

We ran all simulations with $C_{CFL}:=0.95$ and $100$ elements, and the density results are reported, for final times $T_f=2.0$ and $T_f=5000$, in Figures~\ref{fig:Euler_1d_stationary_contact_density_T2} and~\ref{fig:Euler_1d_stationary_contact_density_T5000} respectively.
Looking at Figure~\ref{fig:Euler_1d_stationary_contact_density_T2}, we can see how the exact solution is preserved by the complete upwind fluxes, HLLC and Ex.RS. All other numerical fluxes provide a similar smearing of the contact discontinuity.
This is much more evident for longer final time, see in Figure~\ref{fig:Euler_1d_stationary_contact_density_T5000}.
The stationary solution is still exactly preserved by HLLC and Ex.RS. 
All other numerical fluxes yield unacceptable results, even though the accuracy improves somehow with higher order,  but still not getting comparable with the sharpness realized by the complete upwind fluxes HLLC and Ex.RS. 
This test problem belongs to a class of flow situations in which the use of complete upwind fluxes is essential and in which the deficiencies of centred and incomplete upwind fluxes cannot be compensated by increasing the order of accuracy of the approach.

\begin{figure}[htbp]
	\centering
	\begin{subfigure}[b]{0.6\textwidth}
		\centering
		\includegraphics[width=\textwidth]{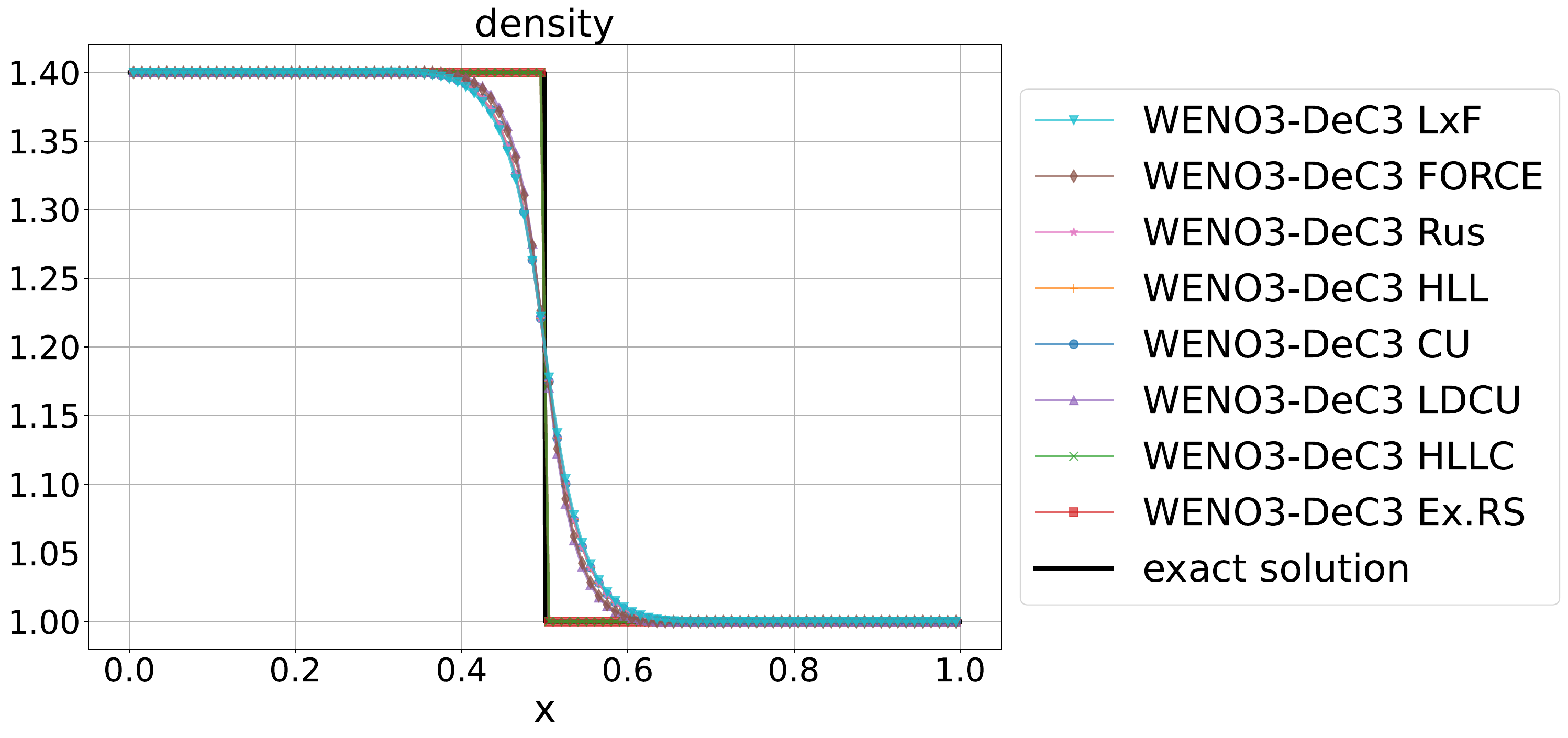}
		\caption{Order 3}
	\end{subfigure}\\
	\begin{subfigure}[b]{0.6\textwidth}
		\centering
		\includegraphics[width=\textwidth]{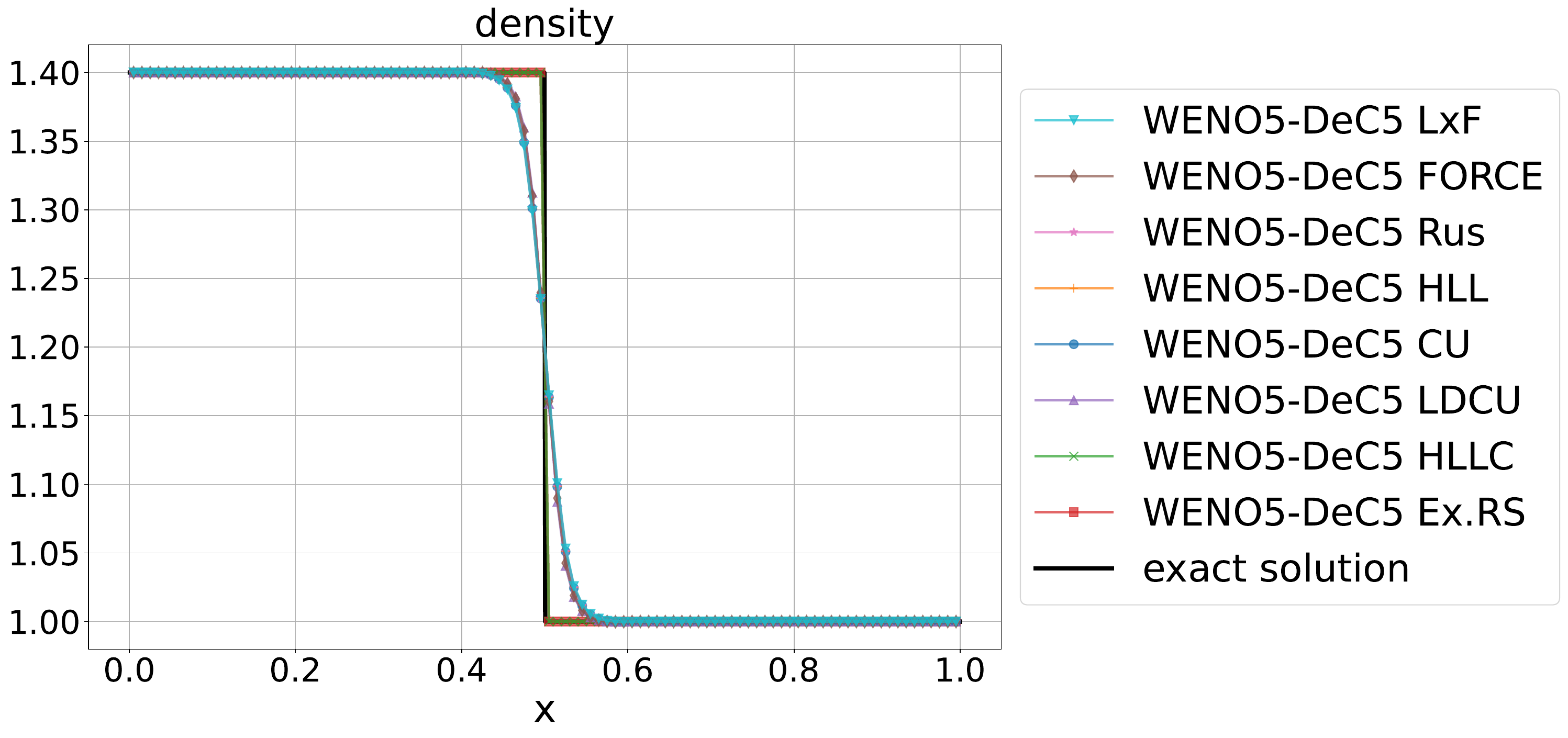}
		\caption{Order 5}
	\end{subfigure}
	\\
	\begin{subfigure}[b]{0.6\textwidth}
		\centering
		\includegraphics[width=\textwidth]{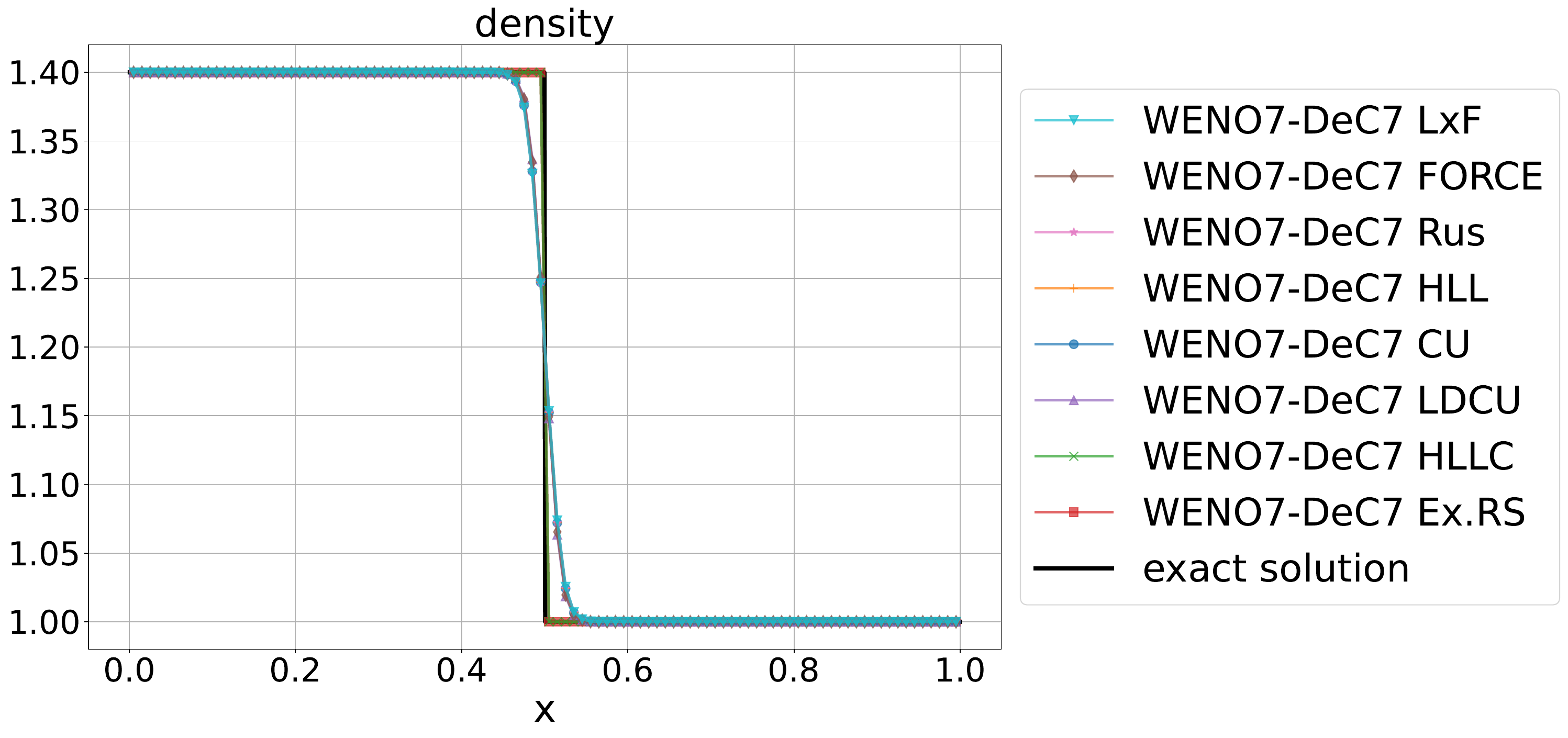}
		\caption{Order 7}
	\end{subfigure}
	\caption{One--dimensional Euler equations, Riemann problem 6: Density profile obtained with all numerical fluxes for different orders with $C_{CFL}:=0.95$ at final time $T_f=2$}
	\label{fig:Euler_1d_stationary_contact_density_T2}
\end{figure}

\begin{figure}[htbp]
	\centering
	\begin{subfigure}[b]{0.6\textwidth}
		\centering
		\includegraphics[width=\textwidth]{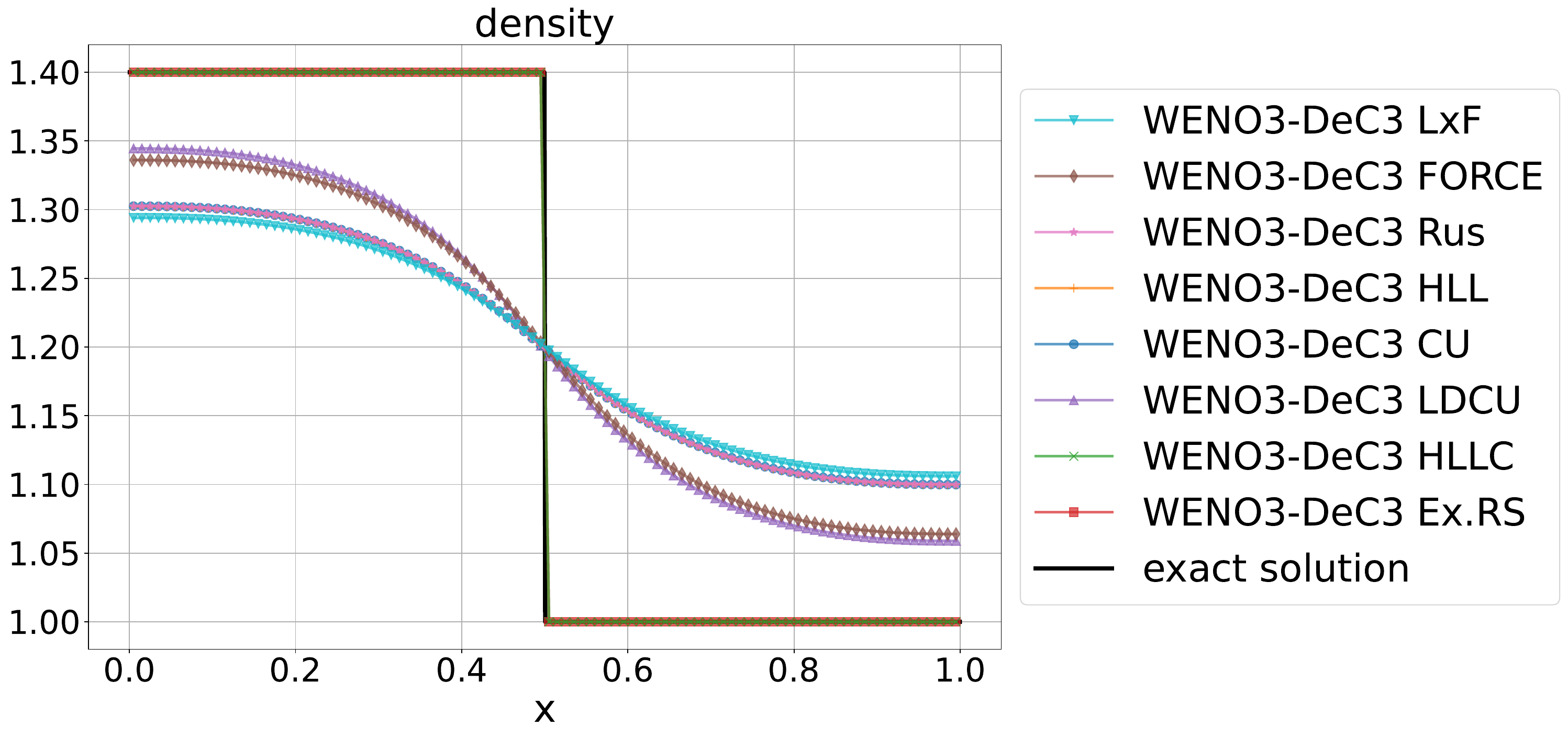}
		\caption{Order 3}
	\end{subfigure}\\
	\begin{subfigure}[b]{0.6\textwidth}
		\centering
		\includegraphics[width=\textwidth]{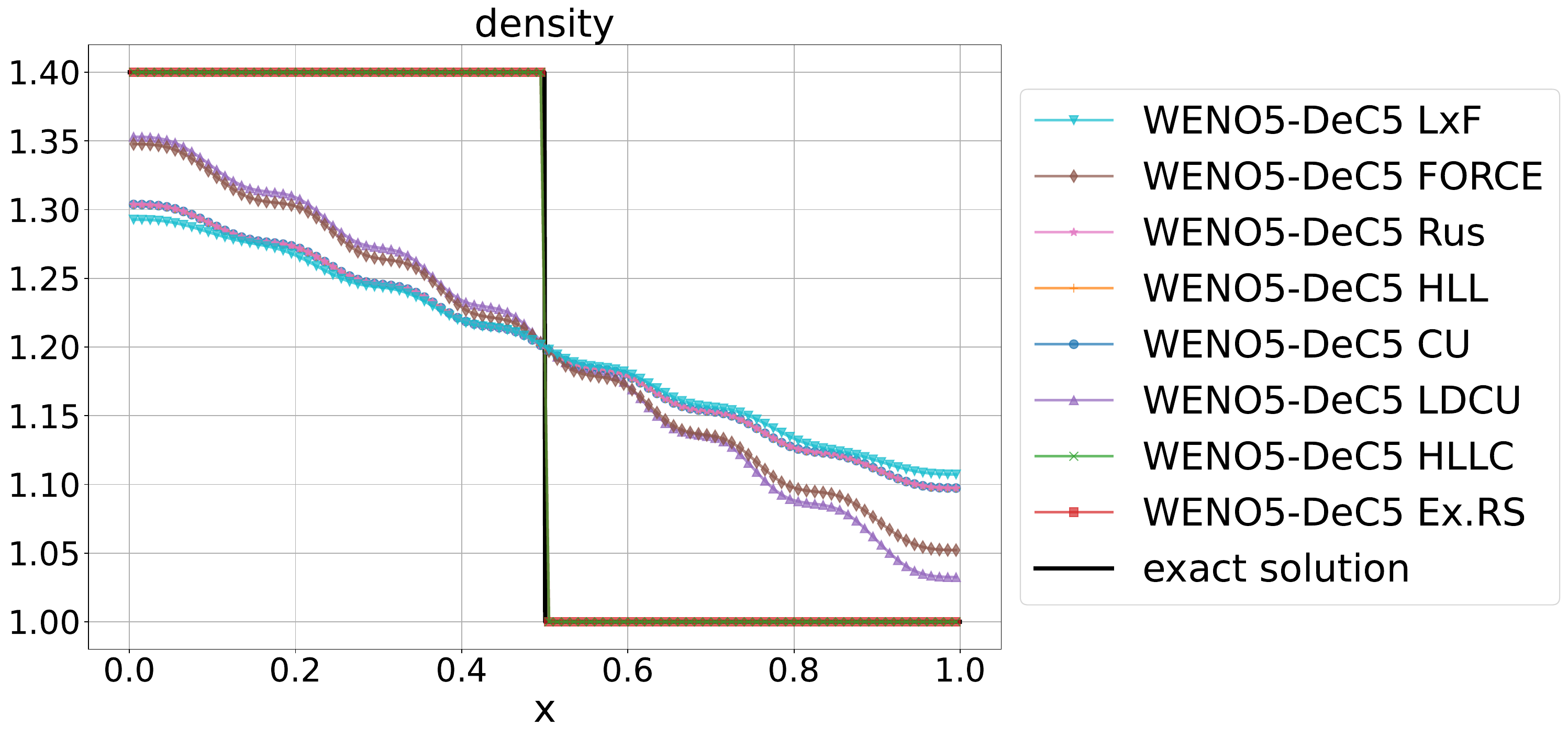}
		\caption{Order 5}
	\end{subfigure}
	\\
	\begin{subfigure}[b]{0.6\textwidth}
		\centering
		\includegraphics[width=\textwidth]{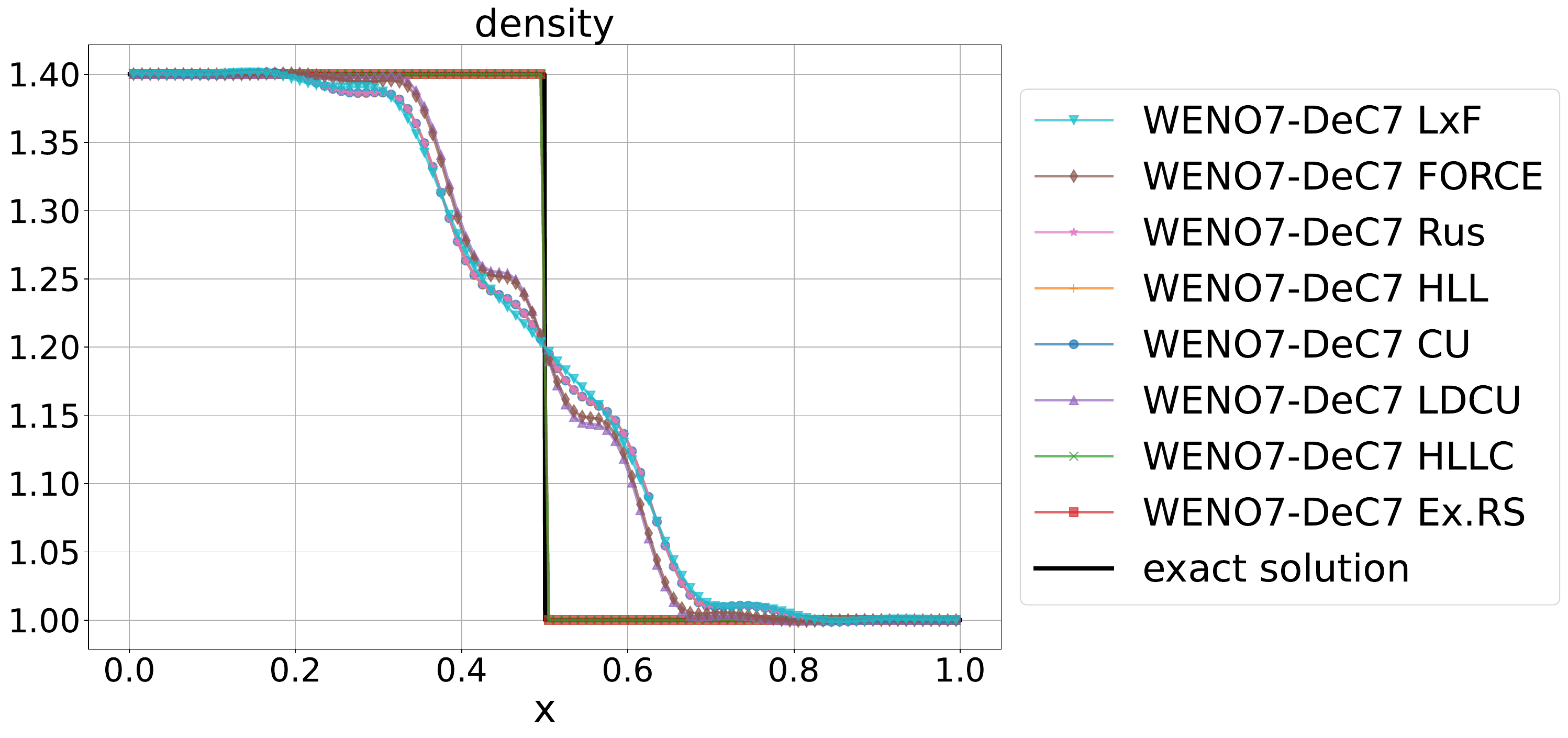}
		\caption{Order 7}
	\end{subfigure}
	\caption{One--dimensional Euler equations, Riemann problem 6: Density profile obtained with all numerical fluxes for different orders with $C_{CFL}:=0.95$ at final time $T_f=5000$}
	\label{fig:Euler_1d_stationary_contact_density_T5000}
\end{figure}

\subsubsection{Riemann problem 7}\label{sec:Euler_1d_RP7}
This test consists in a slight modification of the previous one. 
Here, the two states of the Riemann problem have velocity $0.1$. 
By Galilean invariance, this determines the stationary solution of the previous test to move towards right with the same speed.
Being the contact wave not stationary anymore, the methods obtained through complete upwind fluxes will not capture the exact solution, unlike in the previous case. Still, differences among the numerical fluxes are to be expected and are confirmed by our numerical results.

Again, we considered $C_{CFL}:=0.95$ and $100$ elements.
The density results are displayed in Figure~\ref{fig:Euler_1d_moving_contact_density}.
As expected, due to the nonzero velocity, the complete upwind fluxes, HLLC and Ex.RS, no longer exactly capture the analytical solution.
Nonetheless, we see that they still outperform all other numerical fluxes for orders 3 and 5.
Let us notice that, for order 7, the results obtained through all numerical fluxes are analogous.
As already remarked, the adoption of a higher order discretization may make up for the choice of wrong numerical fluxes. 

\begin{figure}[htbp]
	\centering
	\begin{subfigure}[b]{0.6\textwidth}
		\centering
		\includegraphics[width=\textwidth]{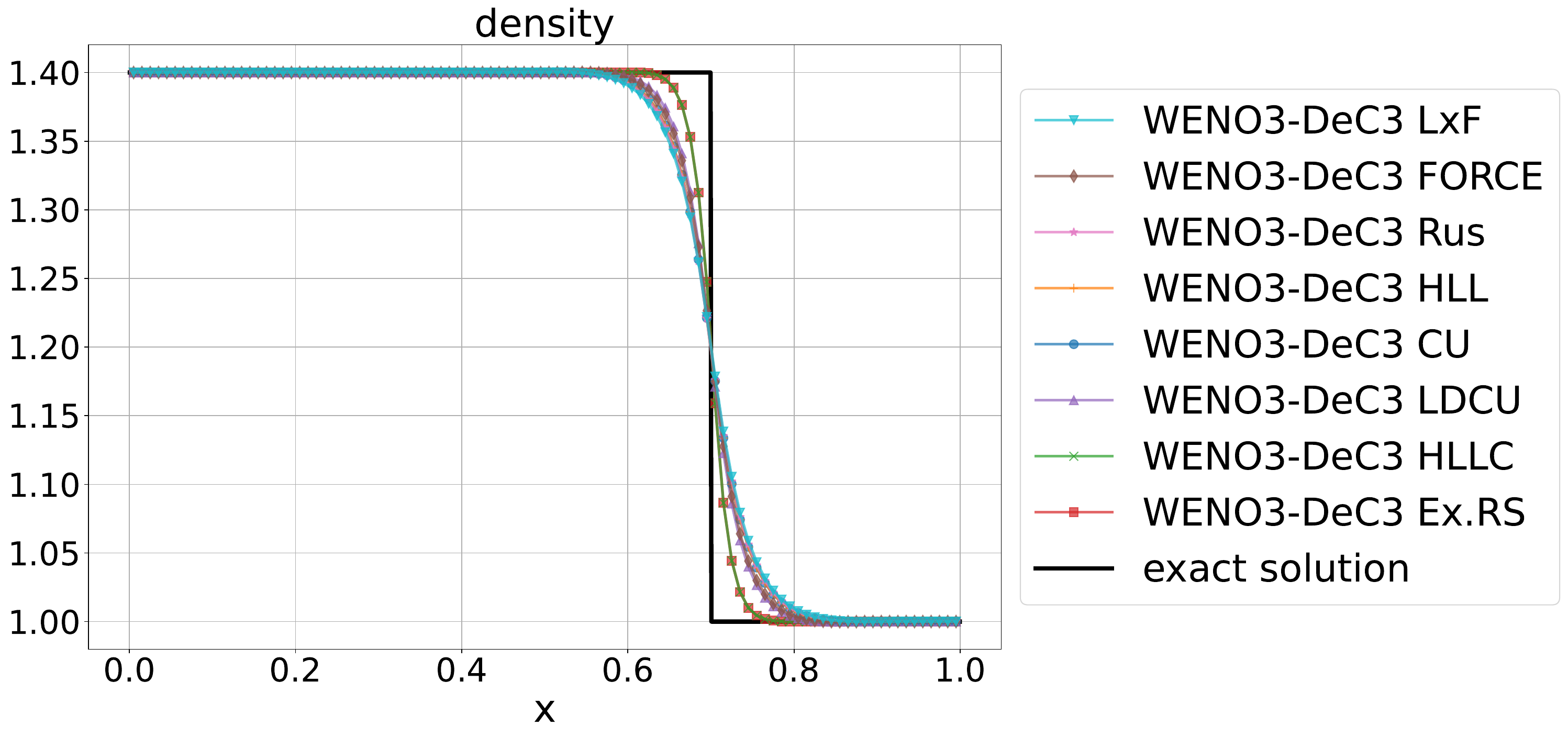}
		\caption{Order 3}
	\end{subfigure}\\
	\begin{subfigure}[b]{0.6\textwidth}
		\centering
		\includegraphics[width=\textwidth]{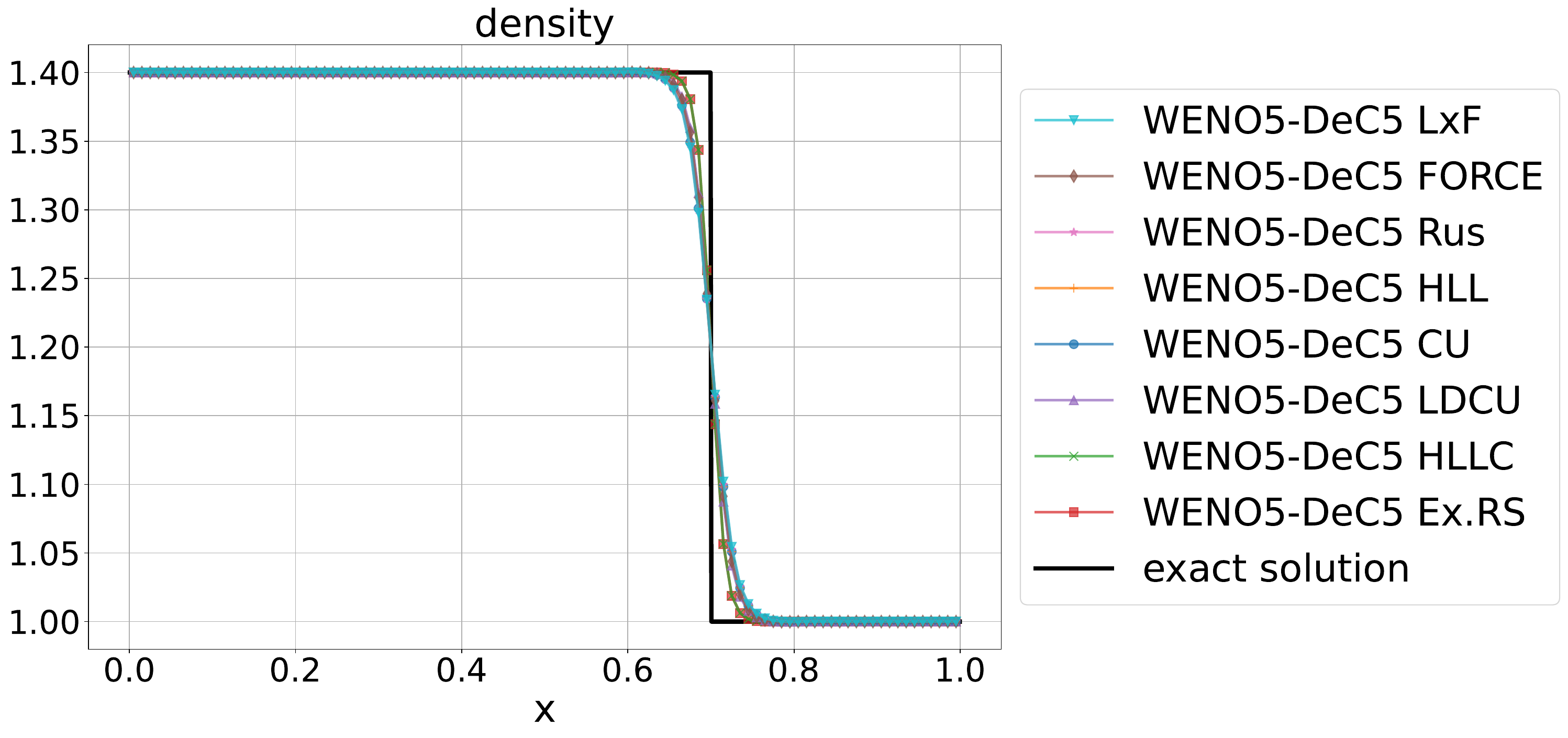}
		\caption{Order 5}
	\end{subfigure}
	\\
	\begin{subfigure}[b]{0.6\textwidth}
		\centering
		\includegraphics[width=\textwidth]{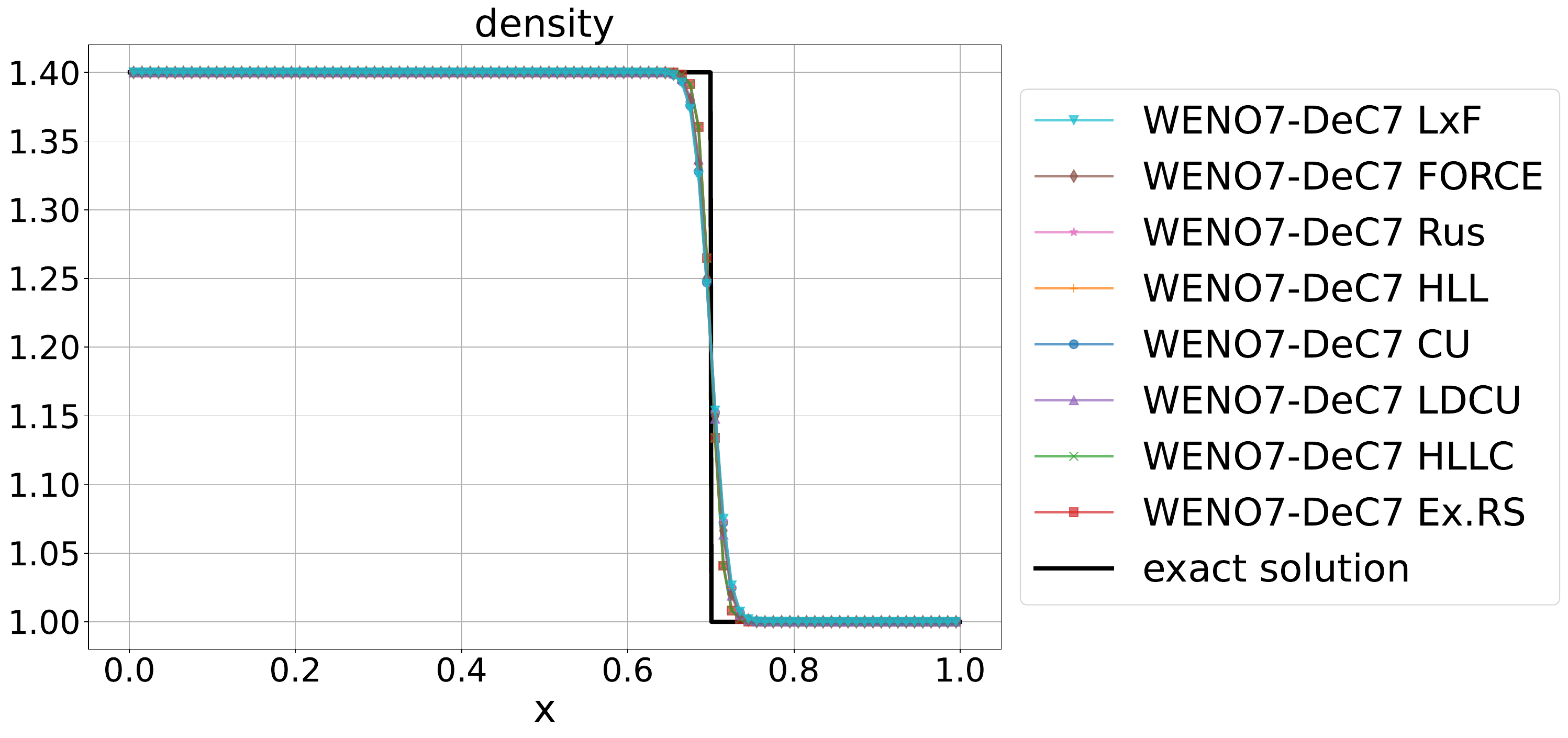}
		\caption{Order 7}
	\end{subfigure}
	\caption{One--dimensional Euler equations, Riemann problem 7: Density profile obtained with all numerical fluxes for different orders with $C_{CFL}:=0.95$}
	\label{fig:Euler_1d_moving_contact_density}
\end{figure}

\subsubsection{Modified shock--turbulence interaction}\label{sec:Euler_1d_titarev_toro}
This last one--dimensional test, presented in~\cite{titarev2004finite}, consists in a challenging modification of the famous shock--turbolence interaction problem introduced in~\cite{shu1989efficient} by Shu and Osher, and it consists in a shock interacting with a turbulent flow characterized by high frequency oscillations.
On the computational domain $\Omega:=[-5,5]$, the initial conditions read 
\begin{align}
	\begin{pmatrix}
		\rho\\
		u\\
		p
	\end{pmatrix}(x,0):=\begin{cases}
		(1.515695,0.523346,1.80500)^T, \quad  &\text{if}~x< -4.5,\\
		(1.0+0.1\sin{(20 \pi x)},0,1)^T, \quad &\text{otherwise}.
	\end{cases} 
	\label{eq:Euler_1d_titarev_toro_IC}
\end{align}
Inflow boundary condition is assumed at the left boundary, while, transmissive boundary condition is prescribed at the right one. The final time is $T_f:=5.$
The resulting solution presents several smooth structures to be captured, as well as a shock.
Problems of this type are well--known to be challenging from a numerical point of view, as schemes should be able at the same time to provide a sharp capturing of the solution features in smooth areas while sharply resolving discontinuities.
More in detail, the challenges posed by this modification are the same as the ones of the original test but they are made much more extreme by the fact that, in this case, the density profile in the turbulent part of the flow has a frequency which is more than 12 times higher than the one of the original problem, determining the formation of many more structures in the solution.
We ran the test with 1000 cells and $C_{CFL}:=0.95$ without experiencing simulation crashes.
The density results for all numerical fluxes and orders are reported in Figure~\ref{fig:Euler_1d_titarev_toro_zoom_density}.

We can observe differences among the numerical fluxes, see panels A and B.
The complete upwind numerical fluxes, HLLC and Ex.RS, give the best results for all orders. In particular, even for order 3, the results obtained through these two numerical fluxes are extremely good, especially in panel A.
Results obtained through all other numerical fluxes are simply unacceptable for order 3. Their quality improves as the order of accuracy increases without ever reaching the one of HLLC and Ex.RS.
Also in this test, we have that the adoption of a good (namely, complete upwind) numerical flux is equivalent (actually better for the considered orders) to the adoption of a much higher order discretization.

Let us observe that, for the considered mesh refinement, even for order 7, even for complete upwind fluxes, the results are actually far from the reference solution in most of the turbulent region of the flow.
We found that in this test case increasing the order of accuracy plays a crucial role for obtaining a good approximation of the solution in the whole domain, however, deeper investigations in this direction are left for future works.

\begin{figure}[htbp]
	\centering
	\begin{subfigure}[b]{0.9\textwidth}
		\centering
		\includegraphics[width=\textwidth]{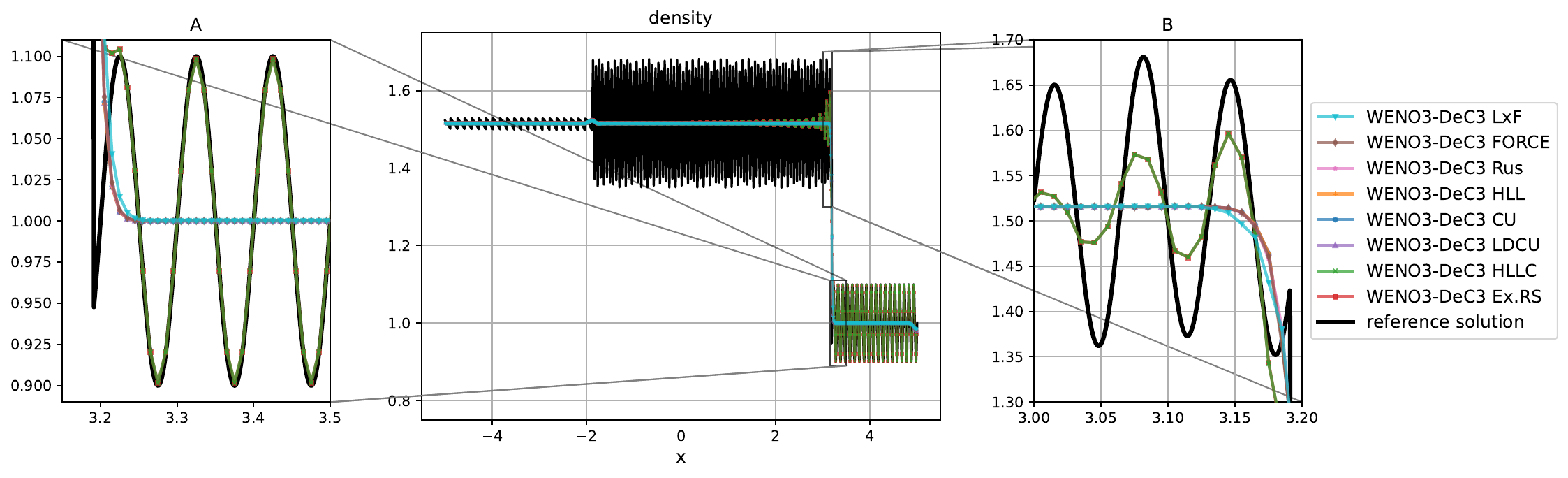}
		\caption{Order 3}
	\end{subfigure}
	\\
	\begin{subfigure}[b]{0.9\textwidth}
		\centering
		\includegraphics[width=\textwidth]{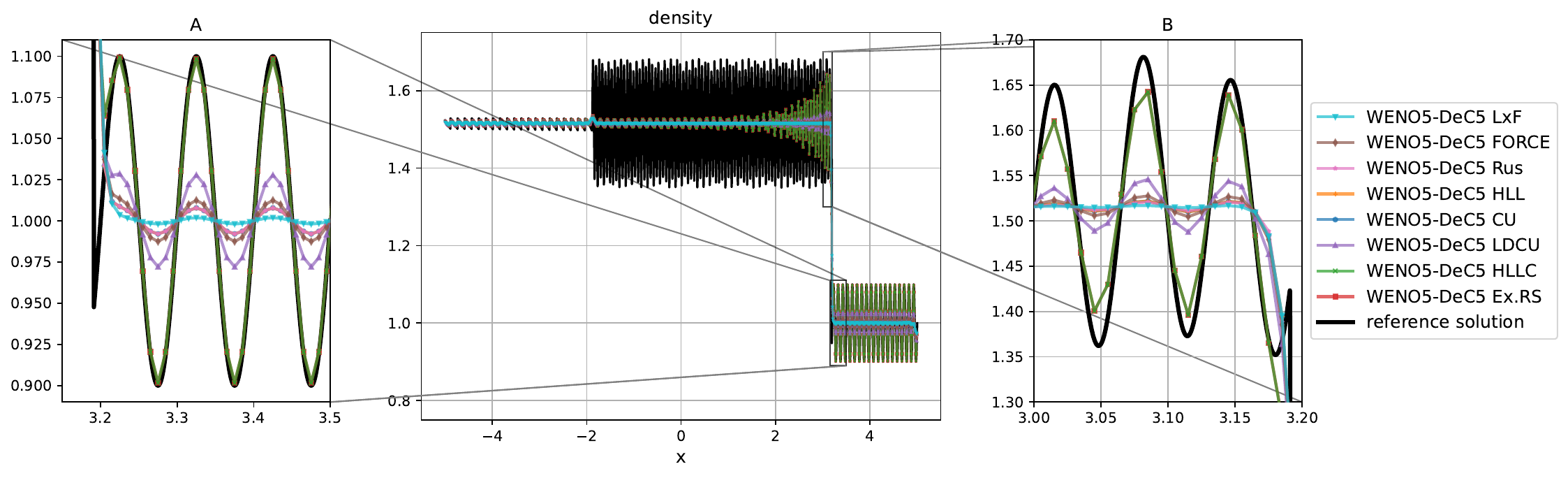}
		\caption{Order 5}
	\end{subfigure}
	\\
	\begin{subfigure}[b]{0.9\textwidth}
		\centering
		\includegraphics[width=\textwidth]{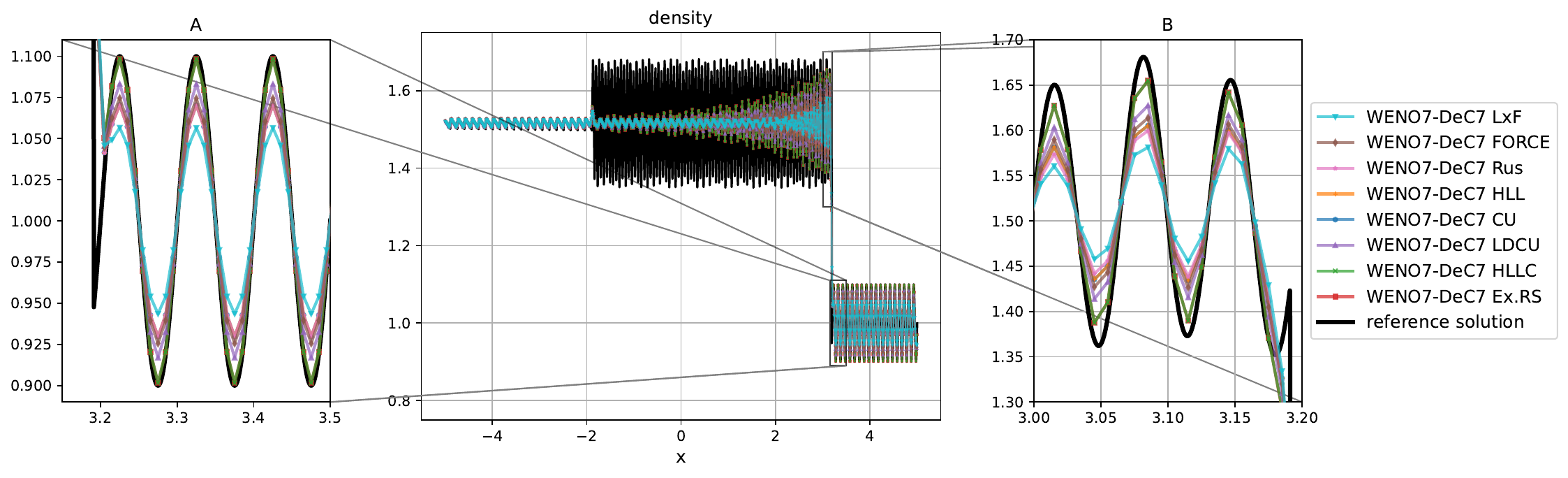}
		\caption{Order 7}
	\end{subfigure}
	\caption{One--dimensional Euler equations, Modified shock--turbulence interaction: Density profile obtained with all numerical fluxes for different orders with $C_{CFL}:=0.95$}
	\label{fig:Euler_1d_titarev_toro_zoom_density}
\end{figure}

\subsubsection{Comments on omitted results}\label{sec:Euler_1d_omitted_results}
Here, we provide some further comments on results which have been omitted for the sake of compactness.
If it is true that the choice of the correct numerical flux is essential in the context of the simulations that we have previously presented, it is also true that this is not the case for other test cases that we have considered in our investigation.
In particular, on Riemann problems 2, 3 and 4 from~\cite[Section 10.8]{ToroBook}, on the original shock--turbulence interaction problem from~\cite{shu1989efficient}, on Lax problem~\cite{liu2015finite,xuan4663427filtered}, and on the long--time (with $T_f:=2000$) advection of a density profile shaped as the composite wave from~\cite{jiang1996efficient}, we found that the choice of the numerical flux is less crucial.
Actually, the real difference in these tests is passing from centred fluxes (LxF and FORCE) or Rus, to any of the other reported numerical fluxes.
Generally speaking, the results obtained using LxF, FORCE, or Rus are consistently worse than those produced by the other numerical fluxes, whose results are instead very similar.
More specifically, LxF always yields the worst results, while the preference between Rus and FORCE depends on the specific test.

We also observed that, increasing the order of the discretization minimizes the differences among the results obtained through different numerical fluxes. For order 7, the performance of the investigated centred fluxes and Rus on the mentioned tests is comparable to the one of the other numerical fluxes.
This is why we plan future investigations in this direction, pushing the order of accuracy to values higher than 7.

Last but not least, we performed all the simulations also applying the WENO reconstruction to conserved variables to check wether the choice of the correct numerical flux could prevent from the oscillations, which are typically associated to such a choice as described in~\cite{qiu2002construction,miyoshi2020short,peng2019adaptive,ghosh2012compact}. We found out that this is not the case. In presence of discontinuities, when reconstructing conserved variables, simulation crashes are more likely to happen. Moreover, even when simulations crashes do not occur, spurious oscillations with amplitude proportional to the considered order arise.
The higher viscosity of more diffusive numerical fluxes, e.g., Rus, determines in some tests a little smearing of such oscillations.
Overall, according to our results, applying the reconstruction to conserved variables does not appear to be a good practice when discontinuities are expected for any choice of the numerical flux.

\subsection{Two--dimensional Euler equations}\label{sec:Euler_2d}
In this section, we deal with numerical simulations involving the two--dimensional Euler equations.
In Section~\ref{sec:Euler_2d_unsteady_vortex}, we test the accuracy of the methods and the performance of the different numerical fluxes on a smooth unsteady isentropic vortex;
in Section~\ref{sec:Euler_2d_unsteady_vortex_longer_time}, we asses the ability of the different numerical fluxes to deal with long--time evolution problems on the same vortex;
in Section~\ref{sec:Euler_2d_explosion_problem}, we compare the numerical fluxes and their shock--capturing properties on an explosion problem~\cite{titarev2004finite}.
The reference solution in the last test has been computed through a FV method with second order accurate van Leer's minmod spatial discretization~\cite{AbgrallMishranotes}, SSPRK2 time discretization, reconstruction of characteristic variables and exact Riemann solver as numerical flux, on a very refined mesh of $3000\times3000$ cells, adopting
$C_{CFL}:=0.25$.

\subsubsection{Smooth isentropic unsteady vortex}\label{sec:Euler_2d_unsteady_vortex}
This test is classically used in many references~\cite{shu1998essentially,micalizzi2023efficient,lore_phd_thesis,boscheri2022continuous,boscheri2015direct,Lagrange2D,boscheri2019high,ArepoTN,boscheri2017arbitrary} to assess the accuracy of very high order numerical schemes.
The computational domain is $\Omega := [-10,10]\times [-10,10]$, and we assume periodic boundary conditions. 
The vortex is initially centered in the origin, $\uvec{x}_c:=(x_c,y_c)^T=(0,0)^T$, and translates with background speed $\uvec{v}_\infty:=(u_\infty,v_\infty)^T=(1,1)^T$. 
The initial condition of the vortex in terms of primitive variables, in the generic point $\uvec{x}=(x,y)^T$ of the space domain, can be described using the radial coordinate $r(\uvec{x}):=\norm{\uvec{x}-\uvec{x}_c}_2$, expressing the distance from the center of the vortex, and reads
\begin{equation}\label{eq:vortex}
	\begin{cases}
		\rho(\uvec{x},0) := (1+\delta \Temp)^{\frac{1}{\gamma-1}},\\
		\uvec{v}(\uvec{x},0):=\uvec{v}_\infty+\frac{\beta}{2\pi}e^{\frac{1-r^2}{2}} \begin{pmatrix}
			-(y-y_c)\\(x-x_c)
		\end{pmatrix},	\\
		p(\uvec{x},0) := (1+\delta \Temp)^{\frac{\gamma}{\gamma-1}},
	\end{cases}
	\quad \delta \Temp := -\frac{(\gamma-1)\beta^2}{8\gamma\pi^2}e^{1-r^2},
\end{equation}
with $\beta:=5$ and $\Temp=\frac{p}{\rho}$ temperature of the fluid. The exact solution at a generic point in space and time is given by $\uvec{u}(\uvec{x},t)=\uvec{u}(\uvec{x}-\uvec{v}_\infty t,0)$. 
We consider a final time $T_f:=0.1$. Let us notice that the vortex, despite being $C^{\infty}$, has not compact support. Therefore, the domain must be chosen big enough to avoid boundary effects, when periodic boundary conditions are assumed, depending on the expected or desired accuracy.
We ran convergence analyses over meshes with $N$ elements both in the $x$- and $y$-direction, and $C_{CFL}:=0.45$.
Also in this case, we will consider only the error on the density, but analogous results have been obtained for all other variables.

The expected convergence rate in the three considered norms has been obtained for all numerical fluxes and all orders, with a single exception: LxF with WENO3--DeC3 experienced simulation crashes in some refinements. This is actually in line with the theoretical analysis presented in~\cite{toro2000centred}. 
The von Neumann linear stability analysis carried in that reference shows that LxF and FORCE are unstable in a first order setting in multiple space dimensions due to their excessive numerical viscosity, however, nothing seems to be known regarding higher order extensions of these methods. 
Surprisingly, in our numerical experiments, we found that increasing the order of accuracy has a stabilizing effect: LxF seems to be stable for orders 5 and 7, while, FORCE seems to be stable for all investigated orders.
Let us notice that in~\cite{toro2009force,dumbser2010force}, linearly stable versions of these two fluxes have been proposed, whose investigation within the WENO--DeC framework is planned for future works.

\begin{table}[htbp]
	\centering
	\caption{Two--dimensional Euler equations, Smooth isentropic unsteady vortex: convergence tables for WENO3--DeC3}
	\label{tab:Euler_2d_unsteady_vortex_convergence_table_WENO3_DeC3}
	\scalebox{0.65}{ 
		\begin{tabular}{c c c c c c c c}
			\toprule
			\multirow{2}{*}{$N$} & \multicolumn{2}{c}{$L^1$ error $\rho$} & \multicolumn{2}{c}{$L^2$ error $\rho$} & \multicolumn{2}{c}{$L^{\infty}$ error $\rho$} & \multirow{2}{*}{CPU Time} \\
			\cmidrule(lr){2-3} \cmidrule(lr){4-5} \cmidrule(lr){6-7}
			& Error & Order & Error & Order & Error & Order & \\
			\midrule
			
			\multicolumn{8}{c}{\textbf{LxF}} \\ 
			\midrule
			160  &   6.886e-01  &  $-$  &   2.598e-01  &  $-$  &   2.786e-01  &  $-$  &   1.309e+01 \\ 
			320  &   $-$  &  $-$  &   $-$  &  $-$  &   $-$  &  $-$  &   $-$ \\ 
			640  &   $-$  &  $-$  &   $-$  &  $-$  &   $-$  &  $-$  &   $-$ \\ 
			1280  &   8.154e-05  &    $-$  &   4.170e-05  &    $-$  &   9.915e-05  &  $-$  &   4.185e+03 \\ 
			2560  &   1.484e-05  &  2.458  &   2.234e-06  &  4.223  &   3.917e-06  &  4.662  &   3.090e+04 \\ 
			5120  &   $-$  &  $-$  &   $-$  &  $-$  &   $-$  &  $-$  &   $-$ \\ 
			\midrule

			\multicolumn{8}{c}{\textbf{FORCE}} \\ 
			\midrule
			160  &   2.846e-02  &  $-$  &   1.248e-02  &  $-$  &   2.337e-02  &  $-$  &   1.252e+01 \\ 
			320  &   5.314e-03  &  2.421  &   2.603e-03  &  2.261  &   6.491e-03  &  1.848  &   7.209e+01 \\ 
			640  &   6.753e-04  &  2.976  &   3.643e-04  &  2.837  &   9.611e-04  &  2.756  &   5.546e+02 \\ 
			1280  &   4.199e-05  &  4.007  &   2.158e-05  &  4.077  &   5.082e-05  &  4.241  &   3.996e+03 \\ 
			2560  &   2.217e-06  &  4.243  &   1.013e-06  &  4.412  &   2.045e-06  &  4.635  &   3.034e+04 \\ 
			5120  &   1.123e-07  &  4.303  &   4.743e-08  &  4.417  &   7.525e-08  &  4.764  &   2.452e+05 \\ 
			\midrule

			\multicolumn{8}{c}{\textbf{Rus}} \\ 
			\midrule
			160  &   1.329e-02  &  $-$  &   5.679e-03  &  $-$  &   1.047e-02  &  $-$  &   1.706e+01 \\ 
			320  &   2.999e-03  &  2.148  &   1.494e-03  &  1.927  &   3.512e-03  &  1.576  &   1.373e+02 \\ 
			640  &   3.725e-04  &  3.009  &   2.057e-04  &  2.860  &   4.517e-04  &  2.959  &   8.407e+02 \\ 
			1280  &   2.506e-05  &  3.894  &   1.301e-05  &  3.983  &   2.785e-05  &  4.020  &   5.598e+03 \\ 
			2560  &   1.387e-06  &  4.175  &   6.376e-07  &  4.351  &   1.207e-06  &  4.529  &   5.376e+04 \\ 
			5120  &   7.068e-08  &  4.295  &   2.984e-08  &  4.417  &   4.474e-08  &  4.753  &   3.725e+05 \\ 
			\midrule

			\multicolumn{8}{c}{\textbf{HLL}} \\ 
			\midrule
			160  &   8.997e-03  &  $-$  &   3.641e-03  &  $-$  &   6.570e-03  &  $-$  &   1.280e+01 \\ 
			320  &   1.843e-03  &  2.288  &   9.329e-04  &  1.965  &   2.320e-03  &  1.502  &   7.456e+01 \\ 
			640  &   2.109e-04  &  3.127  &   1.197e-04  &  2.962  &   3.586e-04  &  2.694  &   5.771e+02 \\ 
			1280  &   1.375e-05  &  3.939  &   6.977e-06  &  4.101  &   1.907e-05  &  4.233  &   4.156e+03 \\ 
			2560  &   7.709e-07  &  4.157  &   3.355e-07  &  4.378  &   6.809e-07  &  4.808  &   3.142e+04 \\ 
			5120  &   4.265e-08  &  4.176  &   1.595e-08  &  4.394  &   2.389e-08  &  4.833  &   2.493e+05 \\ 
			\midrule

			\multicolumn{8}{c}{\textbf{CU}} \\ 
			\midrule
			160  &   8.997e-03  &  $-$  &   3.640e-03  &  $-$  &   6.567e-03  &  $-$  &   1.258e+01 \\ 
			320  &   1.842e-03  &  2.288  &   9.329e-04  &  1.964  &   2.320e-03  &  1.501  &   7.280e+01 \\ 
			640  &   2.109e-04  &  3.127  &   1.197e-04  &  2.962  &   3.586e-04  &  2.694  &   5.626e+02 \\ 
			1280  &   1.375e-05  &  3.939  &   6.977e-06  &  4.101  &   1.907e-05  &  4.233  &   4.121e+03 \\ 
			2560  &   7.709e-07  &  4.157  &   3.355e-07  &  4.378  &   6.809e-07  &  4.808  &   3.086e+04 \\ 
			5120  &   4.265e-08  &  4.176  &   1.595e-08  &  4.394  &   2.389e-08  &  4.833  &   2.440e+05 \\ 
			\midrule

			\multicolumn{8}{c}{\textbf{LDCU}} \\ 
			\midrule
			160  &   8.691e-03  &  $-$  &   3.535e-03  &  $-$  &   6.387e-03  &  $-$  &   1.277e+01 \\ 
			320  &   1.793e-03  &  2.277  &   9.158e-04  &  1.948  &   2.304e-03  &  1.471  &   7.469e+01 \\ 
			640  &   2.045e-04  &  3.132  &   1.176e-04  &  2.961  &   3.560e-04  &  2.694  &   5.765e+02 \\ 
			1280  &   1.318e-05  &  3.955  &   6.734e-06  &  4.126  &   1.874e-05  &  4.248  &   4.171e+03 \\ 
			2560  &   7.364e-07  &  4.162  &   3.196e-07  &  4.397  &   6.568e-07  &  4.835  &   3.150e+04 \\ 
			5120  &   4.070e-08  &  4.177  &   1.504e-08  &  4.409  &   2.320e-08  &  4.823  &   2.500e+05 \\ 
			\midrule

			\multicolumn{8}{c}{\textbf{HLLC}} \\ 
			\midrule
			160  &   8.927e-03  &  $-$  &   3.636e-03  &  $-$  &   6.558e-03  &  $-$  &   1.282e+01 \\ 
			320  &   1.826e-03  &  2.289  &   9.322e-04  &  1.964  &   2.319e-03  &  1.500  &   7.600e+01 \\ 
			640  &   2.068e-04  &  3.142  &   1.194e-04  &  2.965  &   3.586e-04  &  2.693  &   5.907e+02 \\ 
			1280  &   1.343e-05  &  3.945  &   6.940e-06  &  4.105  &   1.908e-05  &  4.232  &   4.274e+03 \\ 
			2560  &   7.559e-07  &  4.151  &   3.342e-07  &  4.376  &   6.814e-07  &  4.807  &   3.227e+04 \\ 
			5120  &   4.226e-08  &  4.161  &   1.603e-08  &  4.382  &   2.396e-08  &  4.830  &   2.554e+05 \\ 
			\midrule

			\multicolumn{8}{c}{\textbf{Ex.RS}} \\ 
			\midrule
			160  &   8.934e-03  &  $-$  &   3.640e-03  &  $-$  &   6.569e-03  &  $-$  &   1.776e+01 \\ 
			320  &   1.826e-03  &  2.291  &   9.321e-04  &  1.965  &   2.319e-03  &  1.502  &   1.406e+02 \\ 
			640  &   2.068e-04  &  3.142  &   1.194e-04  &  2.965  &   3.586e-04  &  2.693  &   8.609e+02 \\ 
			1280  &   1.343e-05  &  3.945  &   6.940e-06  &  4.105  &   1.908e-05  &  4.232  &   5.737e+03 \\ 
			2560  &   7.559e-07  &  4.151  &   3.342e-07  &  4.376  &   6.814e-07  &  4.807  &   5.618e+04 \\ 
			5120  &   4.226e-08  &  4.161  &   1.603e-08  &  4.382  &   2.395e-08  &  4.830  &   3.750e+05 \\ 
			\midrule

			\bottomrule
	\end{tabular}}
\end{table}

\begin{table}[htbp]
	\centering
	\caption{Two--dimensional Euler equations, Smooth isentropic unsteady vortex: convergence tables for WENO5--DeC5}
	\label{tab:Euler_2d_unsteady_vortex_convergence_table_WENO5_DeC5}
	\scalebox{0.65}{ 
		\begin{tabular}{c c c c c c c c}
			\toprule
			\multirow{2}{*}{$N$} & \multicolumn{2}{c}{$L^1$ error $\rho$} & \multicolumn{2}{c}{$L^2$ error $\rho$} & \multicolumn{2}{c}{$L^{\infty}$ error $\rho$} & \multirow{2}{*}{CPU Time} \\
			\cmidrule(lr){2-3} \cmidrule(lr){4-5} \cmidrule(lr){6-7}
			& Error & Order & Error & Order & Error & Order & \\
			\midrule
			
			\multicolumn{8}{c}{\textbf{LxF}} \\ 
			\midrule
			160  &   9.904e-04  &  $-$  &   4.217e-04  &  $-$  &   4.310e-04  &  $-$  &   5.802e+01 \\ 
			320  &   1.615e-05  &  5.938  &   6.439e-06  &  6.033  &   6.618e-06  &  6.025  &   3.487e+02 \\ 
			640  &   3.973e-07  &  5.346  &   1.489e-07  &  5.435  &   1.274e-07  &  5.699  &   2.801e+03 \\ 
			1280  &   9.058e-09  &  5.455  &   3.461e-09  &  5.426  &   3.204e-09  &  5.314  &   2.011e+04 \\ 
			2560  &   2.073e-10  &  5.449  &   7.816e-11  &  5.469  &   7.067e-11  &  5.503  &   1.519e+05 \\ 
			\midrule

			\multicolumn{8}{c}{\textbf{FORCE}} \\ 
			\midrule
			160  &   5.322e-04  &  $-$  &   2.259e-04  &  $-$  &   2.492e-04  &  $-$  &   5.808e+01 \\ 
			320  &   8.532e-06  &  5.963  &   3.391e-06  &  6.057  &   3.523e-06  &  6.144  &   3.422e+02 \\ 
			640  &   2.094e-07  &  5.349  &   7.854e-08  &  5.432  &   6.838e-08  &  5.687  &   2.711e+03 \\ 
			1280  &   4.790e-09  &  5.450  &   1.831e-09  &  5.422  &   1.751e-09  &  5.288  &   1.977e+04 \\ 
			2560  &   1.099e-10  &  5.446  &   4.142e-11  &  5.467  &   3.847e-11  &  5.508  &   1.516e+05 \\ 
			\midrule

			\multicolumn{8}{c}{\textbf{Rus}} \\ 
			\midrule
			160  &   2.374e-04  &  $-$  &   1.008e-04  &  $-$  &   1.429e-04  &  $-$  &   8.473e+01 \\ 
			320  &   4.614e-06  &  5.685  &   1.829e-06  &  5.783  &   2.009e-06  &  6.152  &   6.917e+02 \\ 
			640  &   1.119e-07  &  5.365  &   4.196e-08  &  5.446  &   4.115e-08  &  5.609  &   4.176e+03 \\ 
			1280  &   2.793e-09  &  5.325  &   1.068e-09  &  5.296  &   1.146e-09  &  5.167  &   2.775e+04 \\ 
			2560  &   6.746e-11  &  5.371  &   2.535e-11  &  5.396  &   2.658e-11  &  5.430  &   2.493e+05 \\ 
			\midrule

			\multicolumn{8}{c}{\textbf{HLL}} \\ 
			\midrule
			160  &   1.748e-04  &  $-$  &   7.990e-05  &  $-$  &   1.366e-04  &  $-$  &   5.984e+01 \\ 
			320  &   3.339e-06  &  5.710  &   1.363e-06  &  5.873  &   2.120e-06  &  6.010  &   3.677e+02 \\ 
			640  &   7.905e-08  &  5.401  &   2.982e-08  &  5.515  &   3.594e-08  &  5.882  &   2.912e+03 \\ 
			1280  &   1.912e-09  &  5.369  &   7.124e-10  &  5.387  &   8.342e-10  &  5.429  &   2.131e+04 \\ 
			2560  &   4.461e-11  &  5.422  &   1.608e-11  &  5.470  &   1.744e-11  &  5.580  &   1.611e+05 \\ 
			\midrule

			\multicolumn{8}{c}{\textbf{CU}} \\ 
			\midrule
			160  &   1.748e-04  &  $-$  &   7.990e-05  &  $-$  &   1.366e-04  &  $-$  &   5.764e+01 \\ 
			320  &   3.339e-06  &  5.710  &   1.363e-06  &  5.873  &   2.120e-06  &  6.010  &   3.510e+02 \\ 
			640  &   7.905e-08  &  5.401  &   2.982e-08  &  5.515  &   3.594e-08  &  5.882  &   2.802e+03 \\ 
			1280  &   1.912e-09  &  5.369  &   7.124e-10  &  5.387  &   8.342e-10  &  5.429  &   2.021e+04 \\ 
			2560  &   4.461e-11  &  5.422  &   1.608e-11  &  5.470  &   1.744e-11  &  5.580  &   1.533e+05 \\ 
			\midrule

			\multicolumn{8}{c}{\textbf{LDCU}} \\ 
			\midrule
			160  &   1.696e-04  &  $-$  &   7.871e-05  &  $-$  &   1.386e-04  &  $-$  &   6.005e+01 \\ 
			320  &   3.227e-06  &  5.715  &   1.327e-06  &  5.890  &   2.120e-06  &  6.031  &   3.671e+02 \\ 
			640  &   7.613e-08  &  5.406  &   2.883e-08  &  5.525  &   3.594e-08  &  5.882  &   2.895e+03 \\ 
			1280  &   1.833e-09  &  5.376  &   6.846e-10  &  5.396  &   8.342e-10  &  5.429  &   2.084e+04 \\ 
			2560  &   4.266e-11  &  5.425  &   1.541e-11  &  5.473  &   1.744e-11  &  5.580  &   1.569e+05 \\ 
			\midrule

			\multicolumn{8}{c}{\textbf{HLLC}} \\ 
			\midrule
			160  &   1.711e-04  &  $-$  &   7.877e-05  &  $-$  &   1.363e-04  &  $-$  &   6.781e+01 \\ 
			320  &   3.351e-06  &  5.674  &   1.362e-06  &  5.854  &   2.120e-06  &  6.006  &   3.753e+02 \\ 
			640  &   7.991e-08  &  5.390  &   3.014e-08  &  5.498  &   3.594e-08  &  5.882  &   2.981e+03 \\ 
			1280  &   1.942e-09  &  5.363  &   7.257e-10  &  5.376  &   8.342e-10  &  5.429  &   2.142e+04 \\ 
			2560  &   4.563e-11  &  5.411  &   1.657e-11  &  5.453  &   1.744e-11  &  5.580  &   1.631e+05 \\ 
			\midrule

			\multicolumn{8}{c}{\textbf{Ex.RS}} \\ 
			\midrule
			160  &   1.711e-04  &  $-$  &   7.876e-05  &  $-$  &   1.363e-04  &  $-$  &   8.918e+01 \\ 
			320  &   3.351e-06  &  5.674  &   1.362e-06  &  5.854  &   2.120e-06  &  6.006  &   7.257e+02 \\ 
			640  &   7.991e-08  &  5.390  &   3.014e-08  &  5.498  &   3.594e-08  &  5.882  &   4.428e+03 \\ 
			1280  &   1.942e-09  &  5.363  &   7.257e-10  &  5.376  &   8.342e-10  &  5.429  &   2.905e+04 \\ 
			2560  &   4.569e-11  &  5.409  &   1.657e-11  &  5.453  &   1.744e-11  &  5.580  &   2.642e+05 \\ 
			\midrule

			\bottomrule
	\end{tabular}}
\end{table}

\begin{table}[htbp]
	\centering
	\caption{Two--dimensional Euler equations, Smooth isentropic unsteady vortex: convergence tables for WENO7--DeC7}
	\label{tab:Euler_2d_unsteady_vortex_convergence_table_WENO7_DeC7}
	\scalebox{0.65}{ 
		\begin{tabular}{c c c c c c c c}
			\toprule
			\multirow{2}{*}{$N$} & \multicolumn{2}{c}{$L^1$ error $\rho$} & \multicolumn{2}{c}{$L^2$ error $\rho$} & \multicolumn{2}{c}{$L^{\infty}$ error $\rho$} & \multirow{2}{*}{CPU Time} \\
			\cmidrule(lr){2-3} \cmidrule(lr){4-5} \cmidrule(lr){6-7}
			& Error & Order & Error & Order & Error & Order & \\
			\midrule
			
			\multicolumn{8}{c}{\textbf{LxF}} \\ 
			\midrule
			160  &   5.232e-05  &  $-$  &   1.893e-05  &  $-$  &   2.270e-05  &  $-$  &   3.067e+02 \\ 
			320  &   2.091e-07  &  7.967  &   8.631e-08  &  7.777  &   9.729e-08  &  7.866  &   1.861e+03 \\ 
			640  &   1.024e-09  &  7.674  &   4.250e-10  &  7.666  &   6.531e-10  &  7.219  &   1.488e+04 \\ 
			1280  &   5.310e-12  &  7.591  &   2.175e-12  &  7.610  &   3.208e-12  &  7.669  &   1.081e+05 \\ 
			\midrule

			\multicolumn{8}{c}{\textbf{FORCE}} \\ 
			\midrule
			160  &   2.761e-05  &  $-$  &   1.014e-05  &  $-$  &   1.257e-05  &  $-$  &   3.094e+02 \\ 
			320  &   1.093e-07  &  7.981  &   4.517e-08  &  7.810  &   5.057e-08  &  7.958  &   1.835e+03 \\ 
			640  &   5.364e-10  &  7.671  &   2.238e-10  &  7.657  &   3.566e-10  &  7.148  &   1.455e+04 \\ 
			1280  &   2.826e-12  &  7.568  &   1.135e-12  &  7.623  &   1.694e-12  &  7.718  &   1.057e+05 \\ 
			\midrule

			\multicolumn{8}{c}{\textbf{Rus}} \\ 
			\midrule
			160  &   1.232e-05  &  $-$  &   4.622e-06  &  $-$  &   6.184e-06  &  $-$  &   4.177e+02 \\ 
			320  &   6.129e-08  &  7.651  &   2.524e-08  &  7.517  &   3.064e-08  &  7.657  &   3.463e+03 \\ 
			640  &   2.999e-10  &  7.675  &   1.272e-10  &  7.633  &   2.375e-10  &  7.011  &   2.111e+04 \\ 
			1280  &   1.725e-12  &  7.442  &   6.918e-13  &  7.523  &   1.204e-12  &  7.623  &   1.369e+05 \\ 
			\midrule

			\multicolumn{8}{c}{\textbf{HLL}} \\ 
			\midrule
			160  &   8.704e-06  &  $-$  &   3.632e-06  &  $-$  &   5.512e-06  &  $-$  &   3.085e+02 \\ 
			320  &   4.165e-08  &  7.707  &   1.818e-08  &  7.642  &   2.653e-08  &  7.699  &   1.899e+03 \\ 
			640  &   2.038e-10  &  7.675  &   8.667e-11  &  7.712  &   1.574e-10  &  7.397  &   1.502e+04 \\ 
			1280  &   1.163e-12  &  7.453  &   4.423e-13  &  7.615  &   6.839e-13  &  7.846  &   1.085e+05 \\ 
			\midrule

			\multicolumn{8}{c}{\textbf{CU}} \\ 
			\midrule
			160  &   8.704e-06  &  $-$  &   3.632e-06  &  $-$  &   5.512e-06  &  $-$  &   3.716e+02 \\ 
			320  &   4.165e-08  &  7.707  &   1.818e-08  &  7.642  &   2.653e-08  &  7.699  &   2.247e+03 \\ 
			640  &   2.038e-10  &  7.675  &   8.667e-11  &  7.712  &   1.574e-10  &  7.397  &   1.787e+04 \\ 
			1280  &   1.163e-12  &  7.453  &   4.423e-13  &  7.615  &   6.835e-13  &  7.847  &   1.085e+05 \\ 
			\midrule

			\multicolumn{8}{c}{\textbf{LDCU}} \\ 
			\midrule
			160  &   8.373e-06  &  $-$  &   3.580e-06  &  $-$  &   5.480e-06  &  $-$  &   3.114e+02 \\ 
			320  &   3.983e-08  &  7.716  &   1.772e-08  &  7.659  &   2.653e-08  &  7.690  &   1.901e+03 \\ 
			640  &   1.948e-10  &  7.675  &   8.363e-11  &  7.727  &   1.574e-10  &  7.397  &   1.494e+04 \\ 
			1280  &   1.108e-12  &  7.458  &   4.222e-13  &  7.630  &   6.850e-13  &  7.844  &   1.099e+05 \\ 
			\midrule

			\multicolumn{8}{c}{\textbf{HLLC}} \\ 
			\midrule
			160  &   8.236e-06  &  $-$  &   3.560e-06  &  $-$  &   5.394e-06  &  $-$  &   3.179e+02 \\ 
			320  &   3.956e-08  &  7.702  &   1.765e-08  &  7.656  &   2.653e-08  &  7.668  &   1.914e+03 \\ 
			640  &   1.949e-10  &  7.665  &   8.401e-11  &  7.715  &   1.574e-10  &  7.397  &   1.524e+04 \\ 
			1280  &   1.155e-12  &  7.398  &   4.353e-13  &  7.592  &   6.825e-13  &  7.849  &   1.108e+05 \\ 
			\midrule

			\multicolumn{8}{c}{\textbf{Ex.RS}} \\ 
			\midrule
			160  &   8.236e-06  &  $-$  &   3.560e-06  &  $-$  &   5.394e-06  &  $-$  &   4.513e+02 \\ 
			320  &   3.956e-08  &  7.702  &   1.765e-08  &  7.656  &   2.653e-08  &  7.668  &   3.588e+03 \\ 
			640  &   1.950e-10  &  7.665  &   8.401e-11  &  7.715  &   1.574e-10  &  7.397  &   2.116e+04 \\ 
			1280  &   1.172e-12  &  7.378  &   4.354e-13  &  7.592  &   6.845e-13  &  7.845  &   1.393e+05 \\ 
			\midrule

			\bottomrule
	\end{tabular}}
\end{table}

Convergence tables reporting errors and computational times are reported in Tables~\ref{tab:Euler_2d_unsteady_vortex_convergence_table_WENO3_DeC3},~\ref{tab:Euler_2d_unsteady_vortex_convergence_table_WENO5_DeC5}, and~\ref{tab:Euler_2d_unsteady_vortex_convergence_table_WENO7_DeC7}.
The graphical representations of the convergence analysis and of the efficiency analysis in $L^1$-norm are reported in Figure~\ref{fig:Euler_2d_unsteady_vortex_WENODeC}, on the left and on the right side respectively.
As in the one--dimensional case, graphical representations of the errors in other norms are omitted because they are rather similar.
Looking at the left plot, we see how all schemes achieve the formal order of accuracy besides LxF with WENO3--DeC3, as already mentioned.
One can see that LxF, FORCE and Rus are more diffusive than the other numerical fluxes, and produce higher errors for fixed refinement for all orders.
Among these three numerical fluxes, LxF is the worst, Rus is the best.
All other numerical fluxes perform similarly.

Looking at the efficiency comparison on the right, we see more variety with respect to the one--dimensional case.
Such a difference may be due to the fact that, in this case, we deal with a more involved problem with respect to the simple advection of a density profile.
The best performance, in terms of error versus time, are achieved by HLL, CU, LDCU and HLLC, which give similar results. They are followed by Ex.RS, FORCE and Rus. LxF is the worst option also in this context.

As in the one--dimensional case, however, the choice with more impact on the errors is the one concerning the order of accuracy of the discretization. Higher order schemes result more convenient than lower order ones, allowing to achieve smaller errors within coarser discretizations and smaller computational times, for all numerical fluxes.

\begin{figure}[htbp]
	\centering
	\begin{subfigure}[t]{0.53\textwidth}
		\centering
		\includegraphics[width=\textwidth]{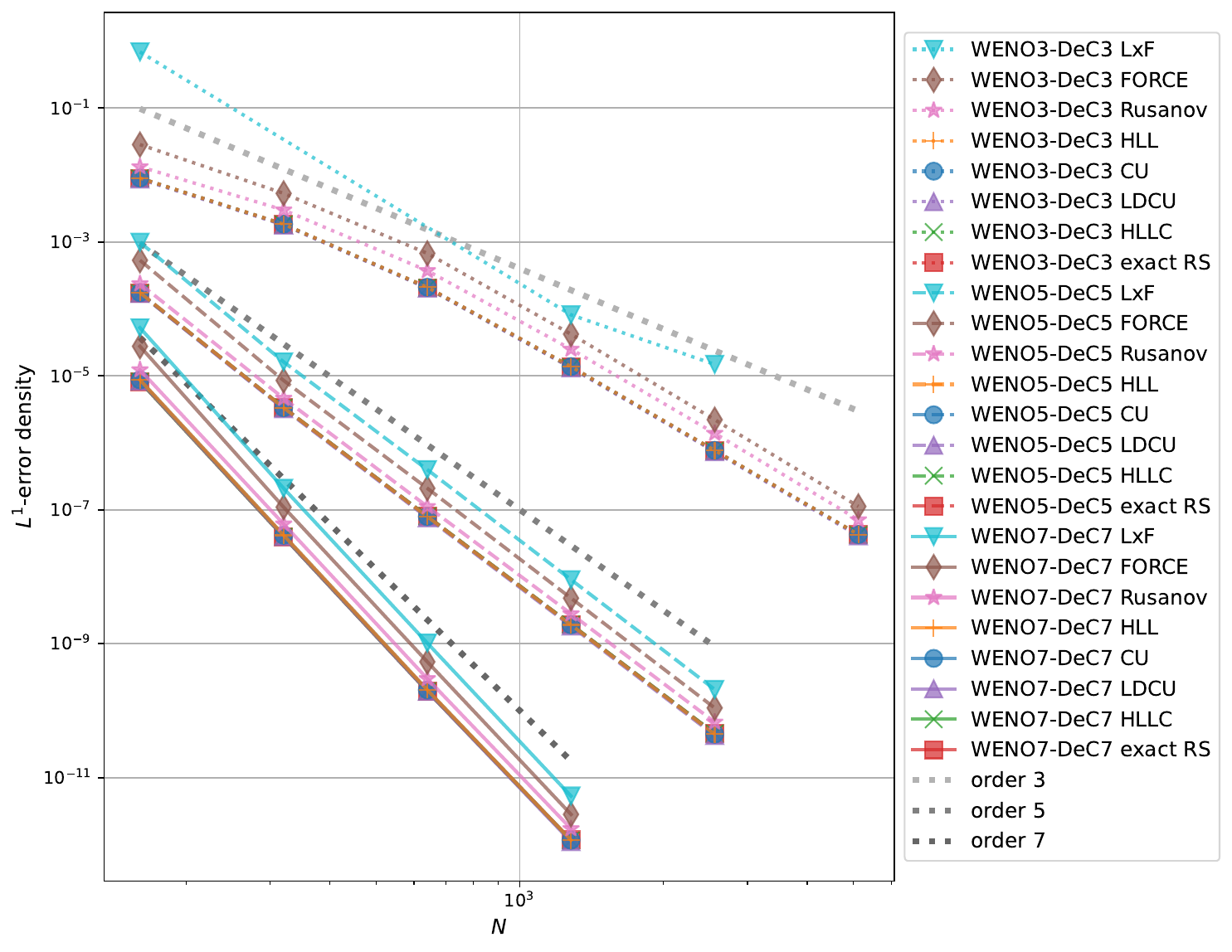}
		\caption{Convergence analysis on the density}
	\end{subfigure}
	\quad
	\begin{subfigure}[t]{0.43\textwidth}
		\centering
		\includegraphics[width=\textwidth]{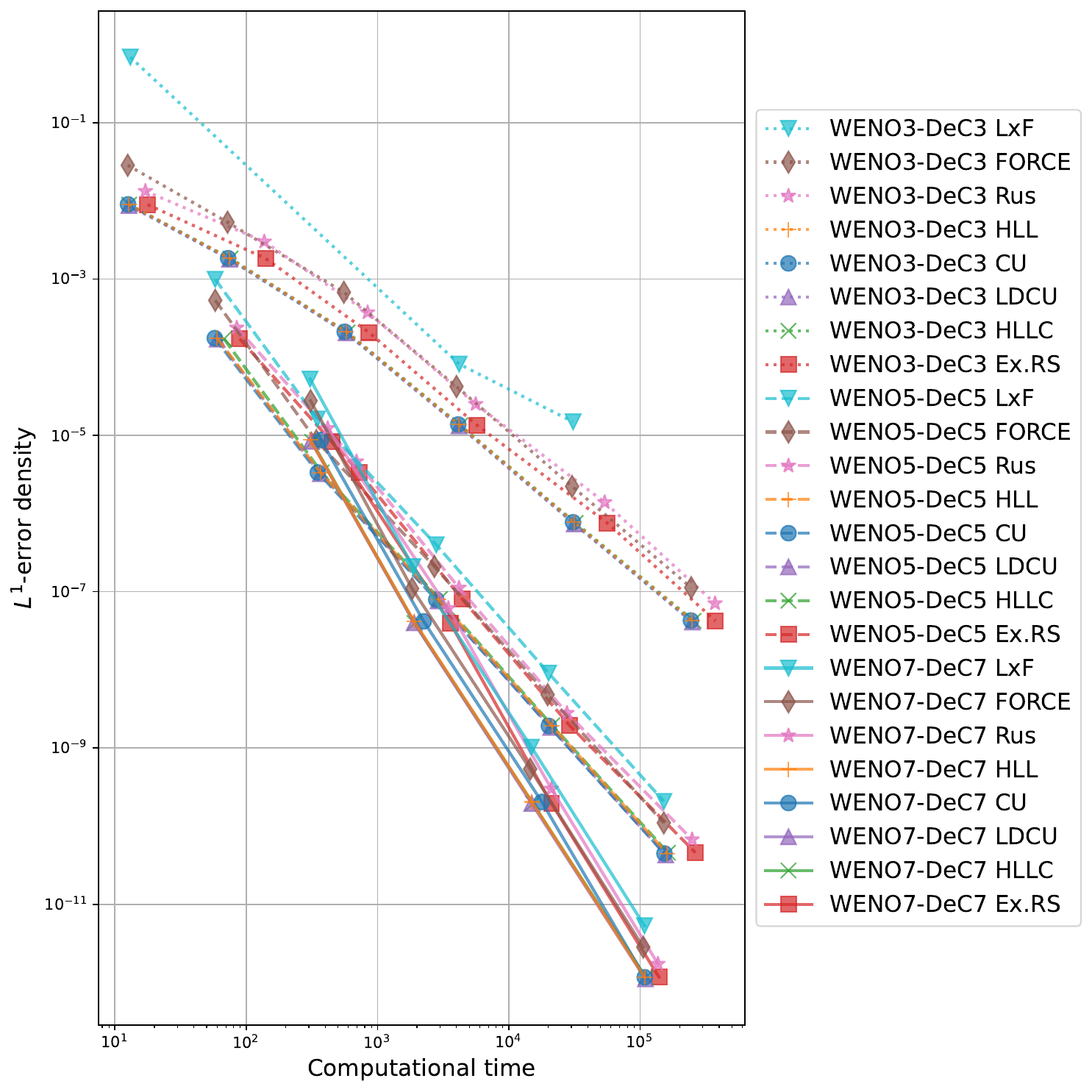}
		\caption{Error on the density versus computational time}
	\end{subfigure}
	\caption{Two--dimensional Euler equations, Smooth isentropic unsteady vortex:  Convergence analysis and efficiency analysis for all numerical fluxes and orders}
	\label{fig:Euler_2d_unsteady_vortex_WENODeC}
\end{figure}

\subsubsection{Long--time evolution of smooth isentropic unsteady vortex}\label{sec:Euler_2d_unsteady_vortex_longer_time}
In this section, we consider the same vortex presented in the previous one, but on the smaller computational domain $\Omega:=[-5,5]\times[-5,5]$, as we are not interested in machine precision effects.
We keep assuming periodic boundary conditions.
We want to asses the differences of the numerical fluxes under investigations on long--time evolution problems.
Numerical fluxes are well known to influence the diffusion of the resulting schemes, which is a crucial aspect in problems of this type. 
We expect more diffusive numerical fluxes, e.g., LxF, FORCE and Rus, to perform worse than less diffusive ones, e.g., HLLC and Ex.RS.
The density results obtained for all numerical fluxes for order 3, 5 and 7, at final times $T_f:=100$, $400$ and $1600$ respectively, are reported in Figure~\ref{fig:Euler_2d_unsteady_vortex_longer_time}. 
In particular, we report a one--dimensional slice of the two--dimensional density profile along the diagonal of the domain $y=x$.

For order 3, the line corresponding to LxF is absent, as the related simulation crashed. Again, we remark that LxF and FORCE are unstable in two--space dimensions for low order discretizations~\cite{toro2000centred}, however, in our numerical experiments they seem to be stable for high order ones: LxF from order 5 on, FORCE for all investigated orders.

For any order, we can see that, as expected, the worst results are the ones obtained with LxF, FORCE and Rus.
The best results are instead obtained with HLLC and Ex.RS for orders 3 and 5, with LDCU for order 7.
HLL and CU give the same results, as they are identical, even though here they are used with different wave speed  estimates.

Let us notice that, also in this case, the order of accuracy of the scheme plays a crucial role in the quality of the results. 
For order 3, results are rather diffusive, in fact, the density profile is smeared to around 20\% of its initial amplitude for the less diffusive numerical fluxes at final time $T_f:=100$.
For order 5, with final time four times longer, $T_f:=400$, results are much better, and the initial profile is only smeared to around 70\% of its initial amplitude for the less diffusive numerical fluxes.
For order 7, with final time four times longer, $T_f:=1600$, results are even better, with the profile only smeared to around 85\% of its initial amplitude for the less diffusive numerical fluxes.
Indeed also the results of more diffusive numerical fluxes improve increasing the order of accuracy of the discretization.

\begin{figure}[htbp]
	\centering
	\begin{subfigure}[b]{0.48\textwidth}
		\centering
		\includegraphics[width=\textwidth]{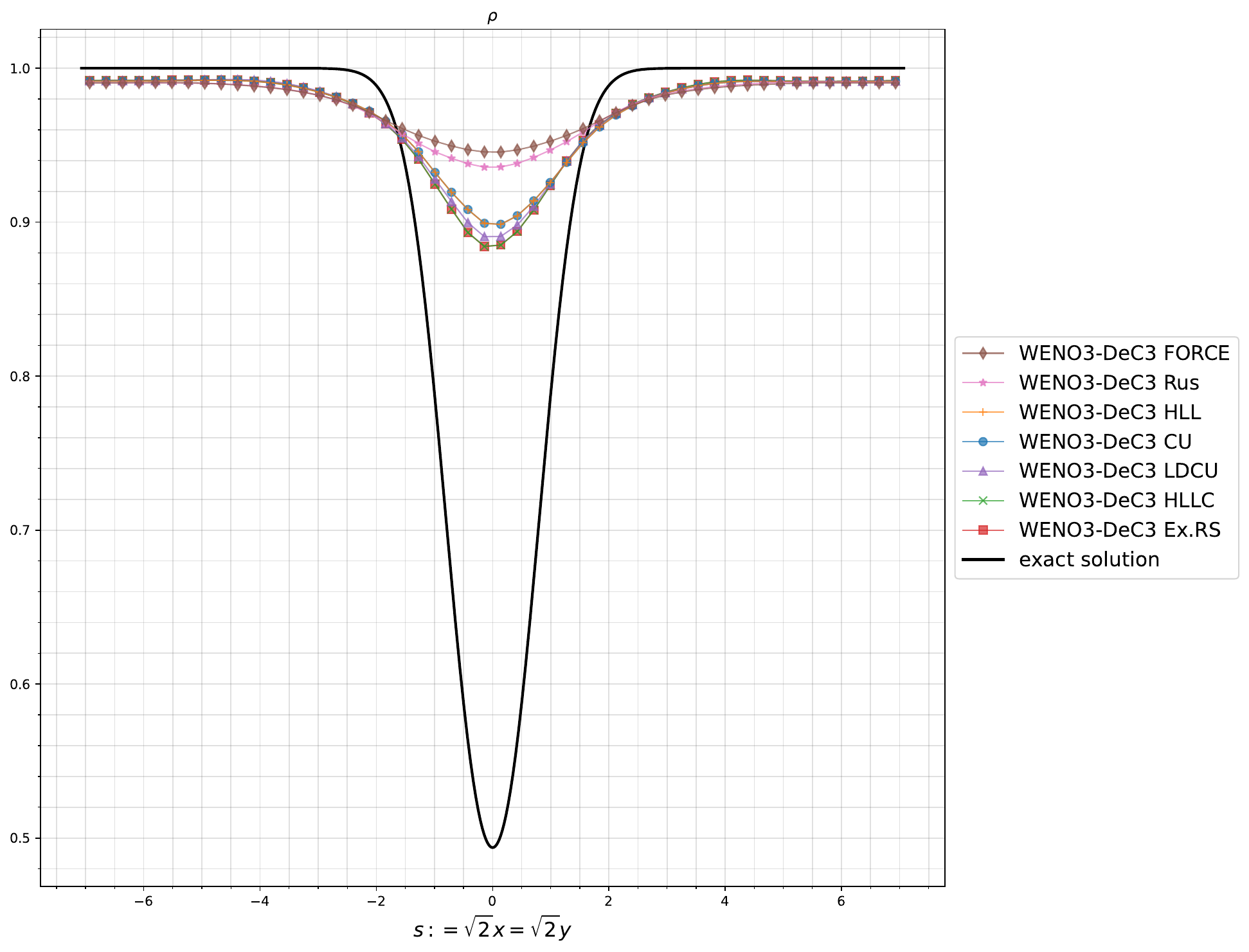}
		\caption{Order 3 with $T_f:=100$}
	\end{subfigure}\\
	\begin{subfigure}[b]{0.48\textwidth}
		\centering
		\includegraphics[width=\textwidth]{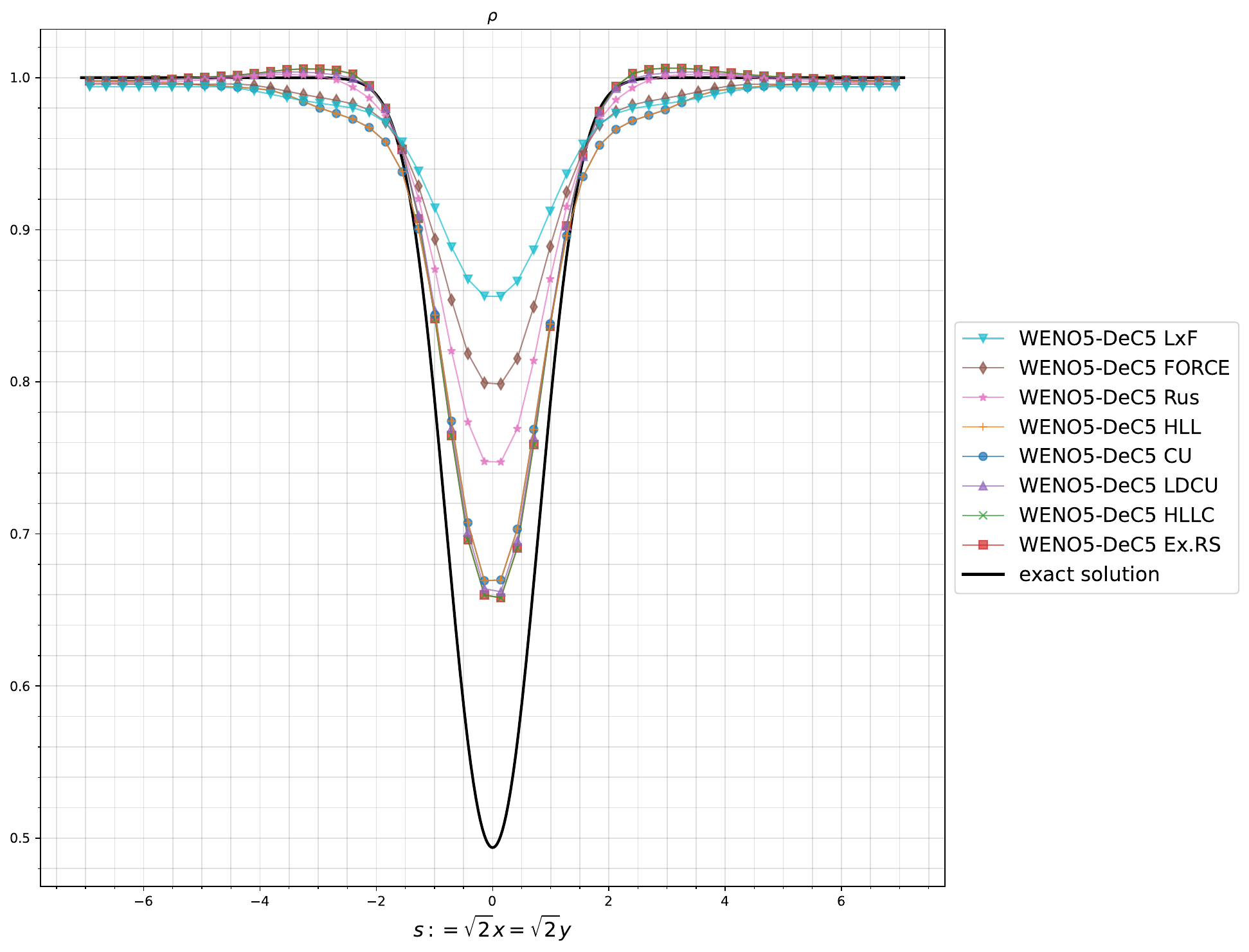}
		\caption{Order 5 with $T_f:=400$}
	\end{subfigure}
	\\
	\begin{subfigure}[b]{0.48\textwidth}
		\centering
		\includegraphics[width=\textwidth]{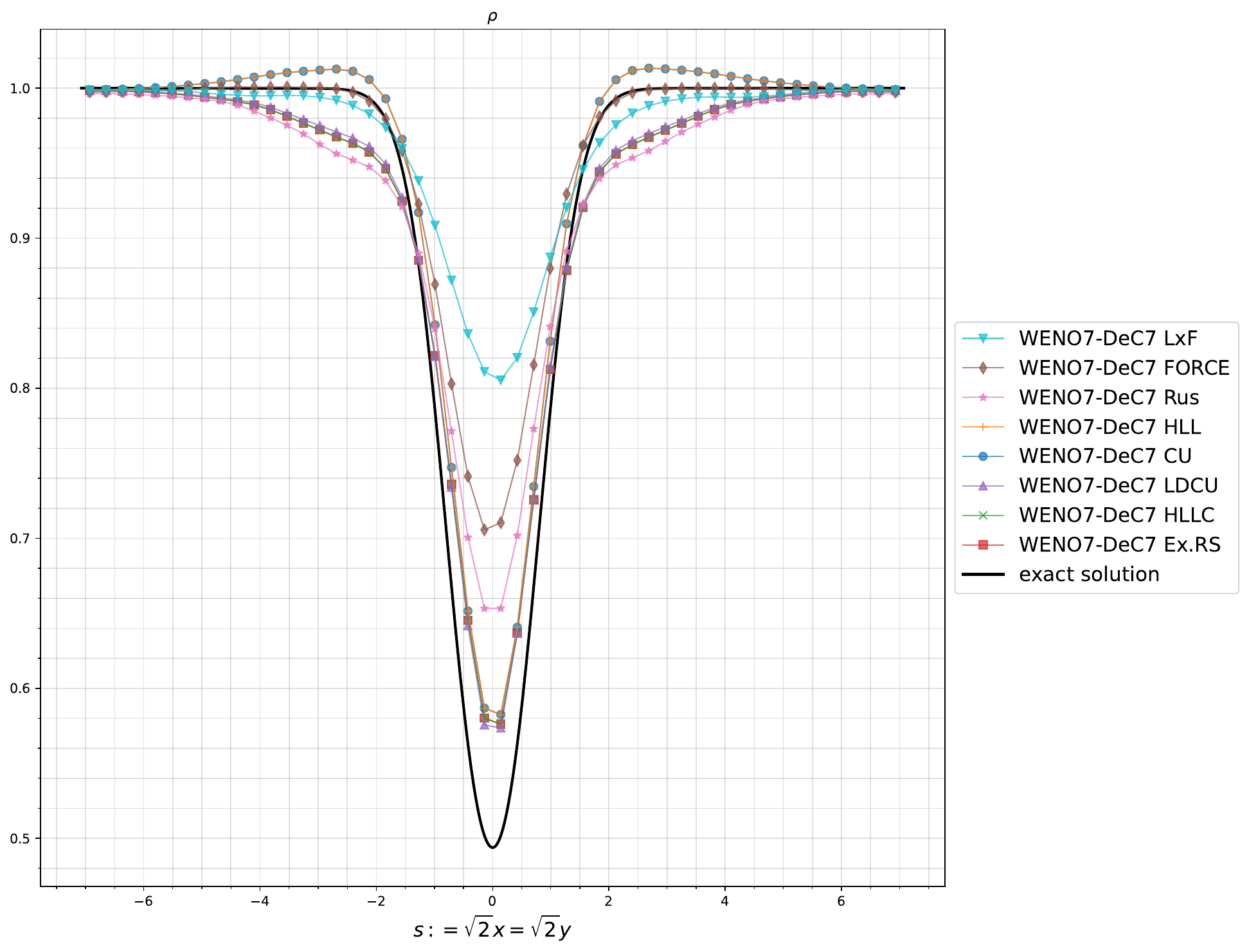}
		\caption{Order 7 with $T_f:=1600$}
	\end{subfigure}
	\caption{Two--dimensional Euler equations, Long--time evolution of smooth isentropic unsteady vortex: Slice of the density profile along the diagonal of the domain $y=x$ obtained with all numerical fluxes for different orders with $C_{CFL}:=0.45$ at different times}
	\label{fig:Euler_2d_unsteady_vortex_longer_time}
\end{figure}

\subsubsection{Explosion problem}\label{sec:Euler_2d_explosion_problem}
This test is taken from~\cite{ToroBook} and it can be regarded as the multidimensional radial version of the Sod shock tube~\cite{sod1978survey}.
On the computational domain $\Omega:=[-1,1]\times[-1,1]$, we consider the following initial conditions
\begin{align}
	\begin{pmatrix}
		\rho\\
		u\\
		v\\
		p
	\end{pmatrix}(x,y,0):=\begin{cases}
		(1,0,0,1)^T, \quad  &\text{if}~r:=\sqrt{x^2+y^2}< 0.4,\\
		(0.125,0,0,0.1)^T, \quad &\text{otherwise}.
	\end{cases} 
	\label{eq:Euler_2d_sod_IC}
\end{align}
We consider transmissive boundary conditions and a final time $T_f:=0.25$.
In Figure~\ref{fig:Euler_2d_explosion_problem}, we report the density obtained with WENO--DeC for all numerical fluxes and orders on computational meshes of $50\times 50$ cells with $C_{CFL}:=0.45$.
In particular, as for the previous test, we provide a one--dimensional slice of the two--dimensional numerical solution on the domain diagonal $y=x$.
We can see that also in this case LxF results are missing for order 3 due to its (expected) instability~\cite{toro2000centred}, however, FORCE results for all orders and LxF results for orders 5 and 7 are surprisingly present.

More or less, all numerical fluxes give similar behavior for fixed order with the exception of LxF and FORCE, which give worse results especially for order 3 and 5.
Moreover, as already remarked, LxF is unstable for order 3.
Also in this test, we see benefits for all numerical fluxes in increasing the order of accuracy of the adopted discretization. In particular, for order 7, results obtained through all numerical fluxes are very similar.

\begin{figure}[htbp]
	\centering
	\begin{subfigure}[b]{0.48\textwidth}
		\centering
		\includegraphics[width=\textwidth]{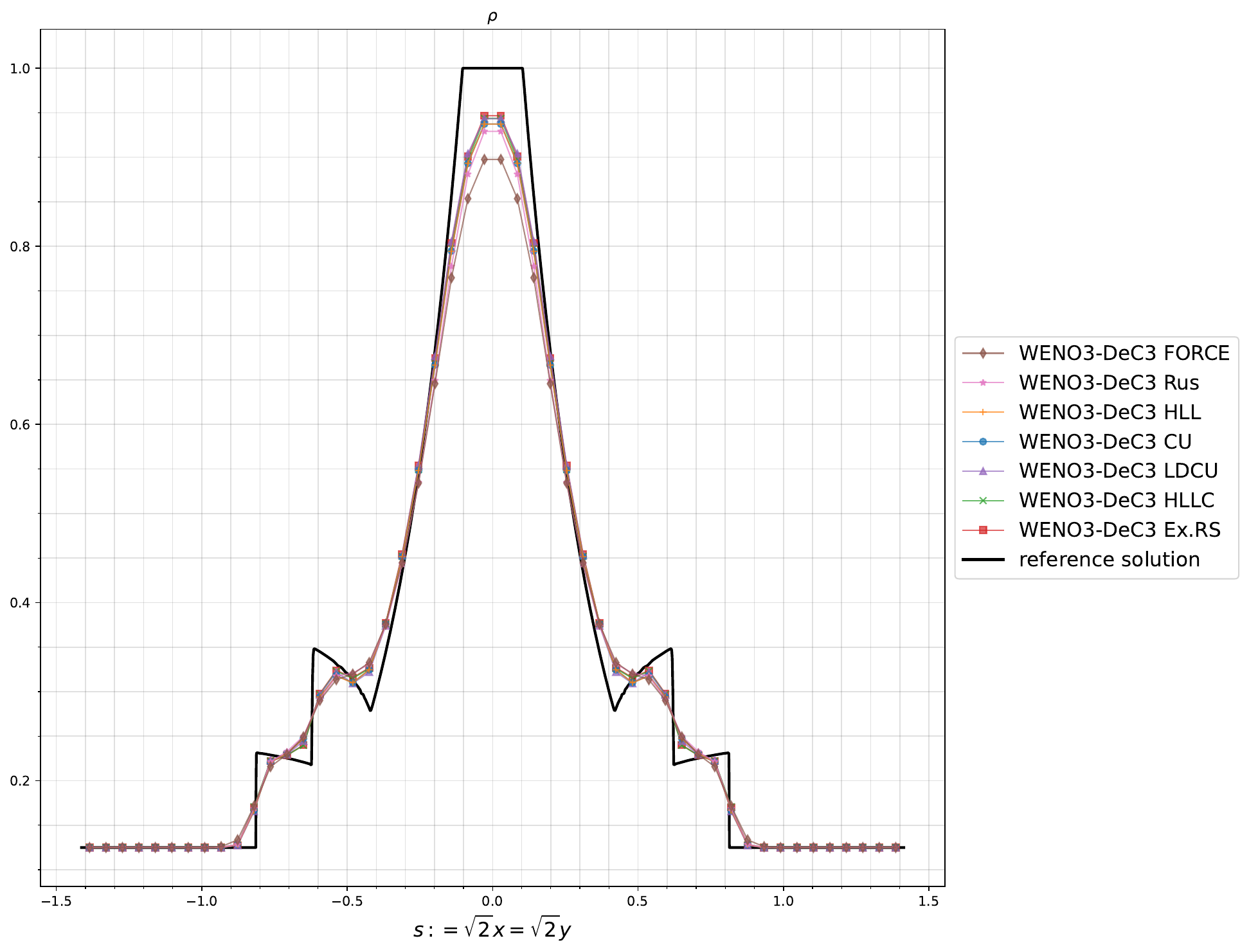}
		\caption{Order 3}
	\end{subfigure}\\
	\begin{subfigure}[b]{0.48\textwidth}
		\centering
		\includegraphics[width=\textwidth]{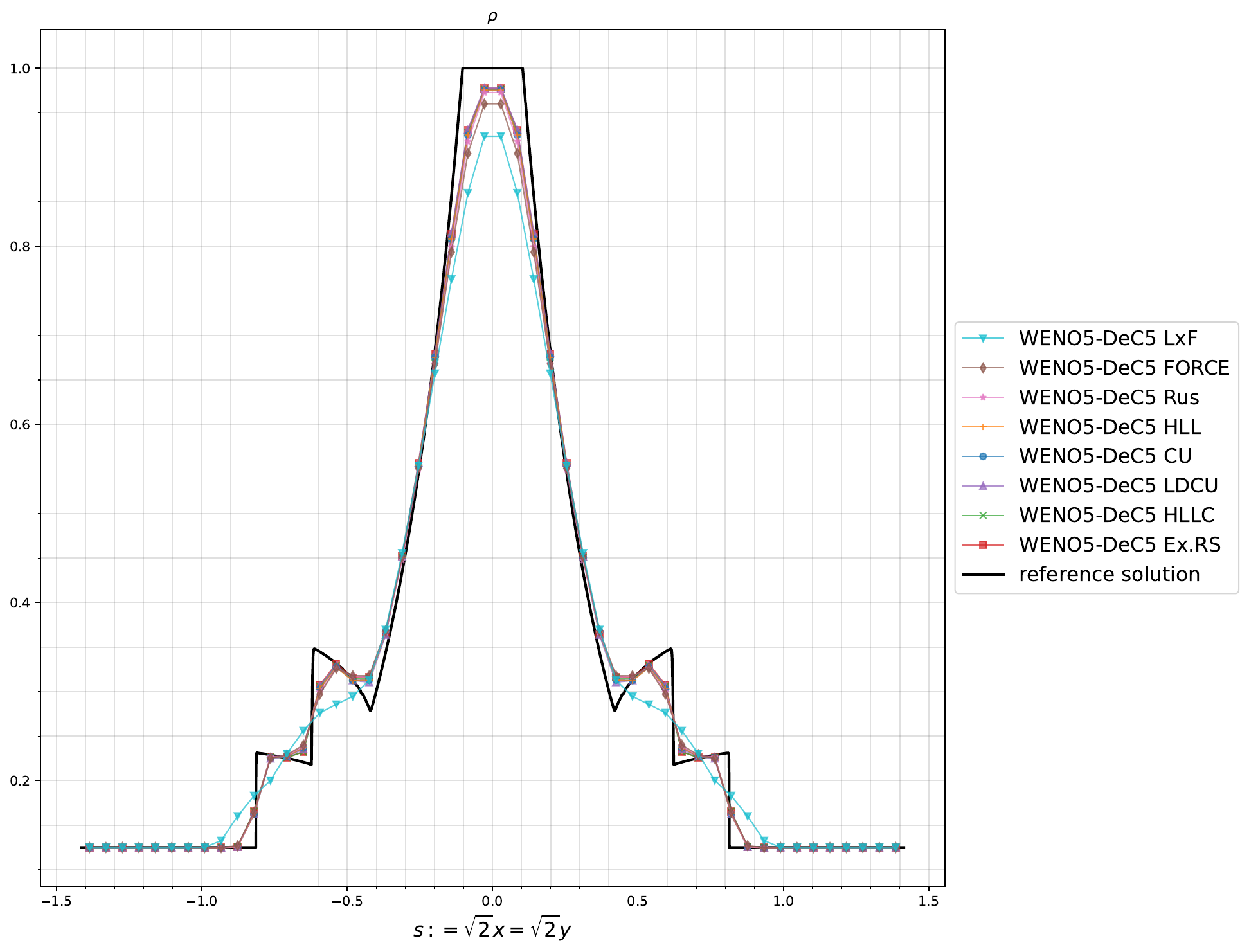}
		\caption{Order 5}
	\end{subfigure}
	\\
	\begin{subfigure}[b]{0.48\textwidth}
		\centering
		\includegraphics[width=\textwidth]{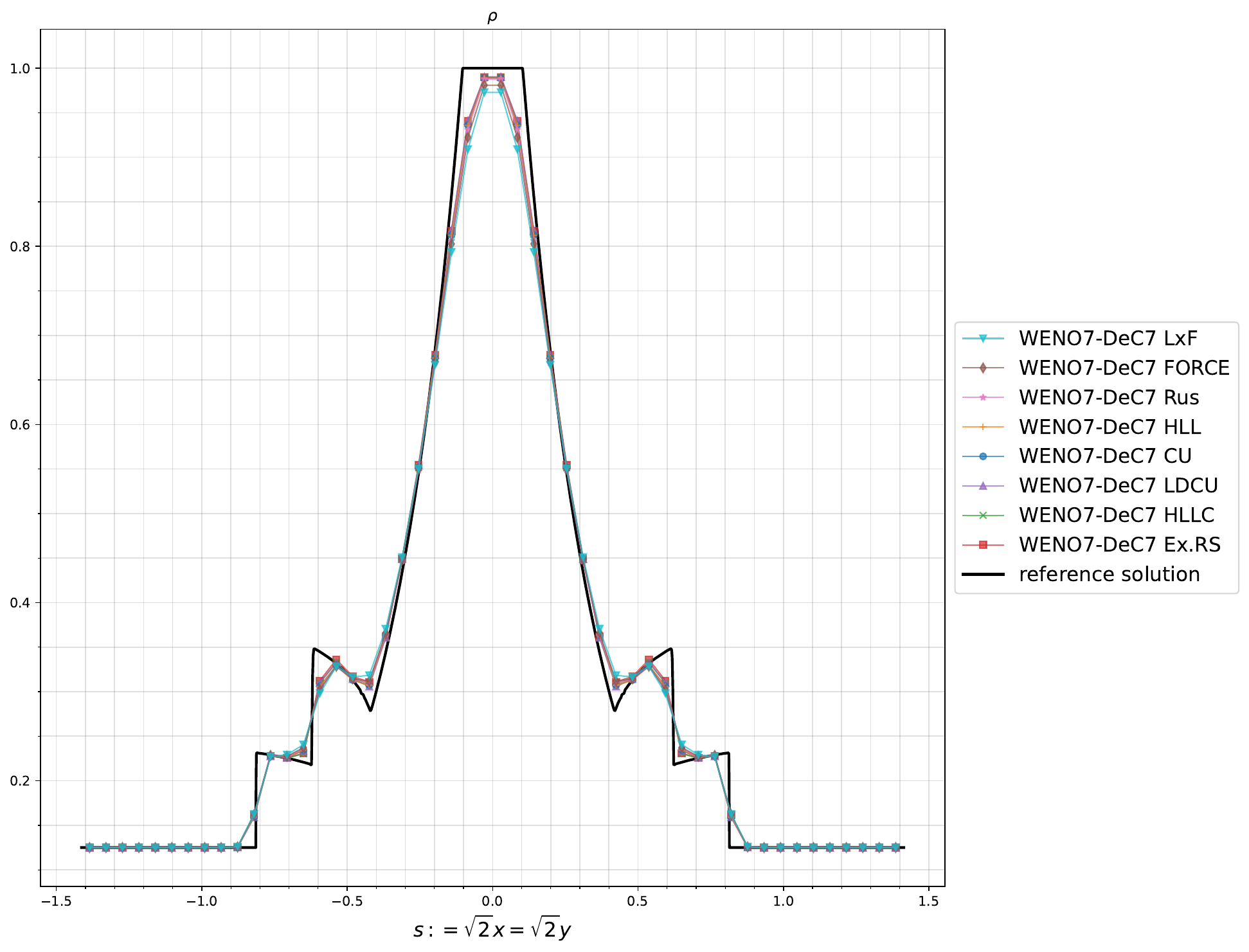}
		\caption{Order 7}
	\end{subfigure}
	\caption{Two--dimensional Euler equations, Explosion problem: Slice of the density profile along the diagonal of the domain $y=x$ obtained with all numerical fluxes for different orders with $C_{CFL}:=0.45$}
	\label{fig:Euler_2d_explosion_problem}
\end{figure}

\section{Conclusions and future perspectives}\label{sec:conclusions}
In this work, we have investigated the role of the numerical flux in an arbitrary high order semidiscrete finite volume (FV) framework, obtained through Weighted Essentially Non--Oscillatory (WENO) space reconstruction and Deferred Correction (DeC) time discretization.
We considered 8 numerical fluxes: Lax--Friedrichs (LxF)~\cite{lax1954weak}, First--Order Centred (FORCE)~\cite{Toro1996,toro2000centred,chen2003centred}, Rusanov~\cite{Rusanov1961}~(Rus), Harten--Lax--van Leer (HLL)~\cite{harten1983upstream}, Central--Upwind (CU)~\cite{kurganov2001semidiscrete,kurganov2000new}, Low--Dissipation Central--Upwind (LDCU)~\cite{kurganov2023new}, HLLC~\cite{toro1992restoration,toro1994restoration}, exact Riemann solver~\cite{Godunov}~(Ex.RS).
In particular, we have focused on order of accuracy up to 7.
We tested the methods on several benchmarks for one--dimensional and two--dimensional Euler equations.
Numerical results confirm the strong impact of the choice of the numerical flux in particular tests, where HLLC and Ex.RS outperform all other competitors and LxF, FORCE and Rus give the worst results.
However, we have also noticed that increasing the order of accuracy reduces the differences among the results obtained through different numerical fluxes.
Nonetheless, in many cases, within the range of investigated orders, this is not sufficient to bring the quality of the results obtained through more diffusive numerical fluxes to the level of the results obtained through more sophisticated ones.
We plan to investigate higher order realizations of WENO--DeC in future works.

Very interesting results have been obtained in the two--dimensional setting, where LxF and FORCE were expected to be unstable.
However, LxF appears to be stable for orders 5 and 7, while, FORCE appears to be stable for all investigated orders. 
Other investigations are planned in this direction, as these two centred numerical fluxes (and related ones) are good candidates for complex applications, where Riemann solvers may not available.

\subsubsection*{Acknowledgments}
LM has been funded by the LeRoy B. Martin, Jr. Distinguished Professorship Foundation. 
EFT gratefully acknowledges the support from the ``Special Programme on Numerical Methods for Non-linear Hyperbolic Partial Differential Equations'', sponsored by the Department of Mathematics and the Shenzhen International Center for Mathematics, Southern University of Science and Technology, Shenzhen, People Republic of China.
Both the authors would like to express gratitude to Alina Chertock, Alexander Kurganov and Chi--Wang Shu, the organizers of the Workshop on Numerical Methods for Shallow Water Models, held in Shenzhen, May 11-15 2024, which provided the initial spark for the development of this work.
Furthermore, the authors acknowledge Rémi Abgrall for granting access to the computational resources of the University of Zurich.

%
%


\printbibliography

\end{document}